\pdfoutput=1
\RequirePackage{ifpdf}
\ifpdf 
\documentclass[pdftex]{sigma}
\else
\documentclass{sigma}
\fi

\usepackage{bm}
\usepackage{bbm}
\usepackage{tikz-cd}

\usepackage{pgfplots}
\usetikzlibrary{calc}
\pgfplotsset{compat=1.18}
\usepgfplotslibrary{colormaps, external}

\usepackage{extarrows}
\usepackage{mathtools}
\usepackage{quiver}
\usepackage{mathrsfs}

\usepackage{aliascnt}

\usepackage{soul}

\usepackage{adjustbox}
\usepackage[shortlabels]{enumitem}

\DeclareMathOperator{\id}{id}
\DeclareMathOperator{\End}{End}
\DeclareMathOperator{\Hom}{Hom}
\DeclareMathOperator{\Aut}{Aut}

\DeclareMathOperator{\tr}{tr}
\DeclareMathOperator{\Bun}{\mathsf{Bun}}

\DeclareMathOperator{\vect}{\mathsf{Vect}}
\DeclareMathOperator{\Fun}{\mathsf{Fun}}
\DeclareMathOperator{\Rep}{\mathsf{Rep}}
\DeclareMathOperator{\K}{\mathbb{K}}
\DeclareMathOperator{\C}{\mathbb{C}}
\DeclareMathOperator{\R}{\mathbb{R}}
\DeclareMathOperator{\Sp}{\mathbb{S}}

\DeclareMathOperator{\Cob}{\mathsf{Cob}}
\DeclareMathOperator{\cala}{\mathcal{A}}
\DeclareMathOperator{\calc}{\mathcal{C}}
\DeclareMathOperator{\cald}{\mathcal{D}}
\DeclareMathOperator{\calo}{\mathcal{O}}
\DeclareMathOperator{\calx}{\mathcal{X}}
\DeclareMathOperator{\caly}{\mathcal{Y}}
\DeclareMathOperator{\calg}{\mathcal{G}}
\DeclareMathOperator{\calh}{\mathcal{H}}
\DeclareMathOperator{\cali}{\mathcal{I}}
\DeclareMathOperator{\Fac}{\mathrm{Fac}}

\newcommand{\sslash}{\mathbin{/\mkern-6mu/}}

\numberwithin{equation}{section}

\newtheorem{Theorem}{Theorem}[section]

\newtheorem*{Theorem*}{Theorem}

\newaliascnt{Corollary}{Theorem}
\newtheorem{Corollary}[Corollary]{Corollary}
\aliascntresetthe{Corollary}

\newaliascnt{Lemma}{Theorem}
\newtheorem{Lemma}[Lemma]{Lemma}
\aliascntresetthe{Lemma}

\newaliascnt{Proposition}{Theorem}
\newtheorem{Proposition}[Proposition]{Proposition}
\aliascntresetthe{Proposition}

\theoremstyle{definition}

\newaliascnt{Definition}{Theorem}
\newtheorem{Definition}[Definition]{Definition}
\aliascntresetthe{Definition}

\newaliascnt{Note}{Theorem}

\aliascntresetthe{Note}

\newaliascnt{Example}{Theorem}
\newtheorem{Example}[Example]{Example}
\aliascntresetthe{Example}

\newaliascnt{Remark}{Theorem}
\newtheorem{Remark}[Remark]{Remark}
\aliascntresetthe{Remark}

\newtheorem{Mainthm}{Theorem}

\newtheorem{Question}[Mainthm]{Question}
\newaliascnt{Hypothesis}{Theorem}
\newtheorem{Hypothesis}[Hypothesis]{Hypothesis}
\aliascntresetthe{Hypothesis}

\newaliascnt{Observation}{Theorem}
\newtheorem{Observation}[Observation]{Observation}
\aliascntresetthe{Observation}

\begin{document}

\allowdisplaybreaks

\newcommand{\arXivNumber}{2508.13759}

\renewcommand{\PaperNumber}{079}

\FirstPageHeading

\ShortArticleName{A Boundary Characterization of Turaev--Viro TQFTs}

\ArticleName{A Boundary Characterization of Turaev--Viro TQFTs}

\Author{Max-Niklas STEFFEN and Christoph SCHWEIGERT}

\AuthorNameForHeading{M.-N.~Steffen and C.~Schweigert}

\Address{Department of Mathematics, University of Hamburg, AZ -- Algebra and Number Theory,\\ Bundesstra{\ss}e 55,
Hamburg, Germany}
\Email{\mail{maxniklas.steffen@uni-hamburg.de}, \mail{christoph.schweigert@uni-hamburg.de}}

\ArticleDates{Received October 20, 2025, in final form August 07, 2026; Published online August 19, 2026}

\Abstract{We consider three-dimensional topological field theories on manifolds with boundary defects and identify explicit boundary locality conditions. We show that these conditions imply a state sum construction of the given TQFT. As a consistency check, we prove that Turaev--Viro state sum models obey the boundary locality conditions. Recent progress [Faria~Martins J., Meusburger C., \textit{Adv. Math.} \textbf{494} (2026), 110923, 102~pages, arXiv:2410.18049] in the description of defects in Dijkgraaf--Witten theories enables us to show that these theories likewise satisfy boundary locality. This directly implies that Dijkgraaf--Witten TQFTs with boundary defects admit a state sum description.}


\Classification{57K16; 57K31}

\section{Introduction}
\label{secIntro}

There are two main paradigms in topological quantum field theories (TQFTs). The
bottom-up approach assigns data to low-dimensional manifolds, for example to a circle or a point,
and derives a TQFT from these data. In particular, the well-known cobordism hypothesis starts with structure assigned to a point. Another approach, top-down, starts from top-dimensional mani\-folds and aims to assign data
to lower-dimensional manifolds. This perspective is sometimes motivated by formalizations of path integrals. It has the advantage of enabling explicit computations of quantities
associated with high-dimensional manifolds that are of interest for representation
theory, such as, for example, representations of mapping class groups or Frobenius--Schur
indicators.

The latter approach is the one we adopt in this paper. The central idea is to assume the extension of a TQFT
to a sufficiently rich boundary theory, establishing criteria that ensure the existence of a state sum description. The concept of boundary locality gives a precise meaning to
the words `sufficiently rich'. We show that it applies to Dijkgraaf--Witten
and Turaev--Viro state sum theories.

We begin with some motivation. Let $\omega_A$ be a 3-manifold invariant -- more specifically, a~map assigning to every three-dimensional closed oriented manifold $M$ a scalar $\omega_A(M)$ in an algebraically closed field $\K$ of characteristic zero. We assume that the construction of $\omega_A$ depends on some algebraic input datum $A$, without making concrete what this means in detail, and further assume that $A$ can be naturally linked to some spherical fusion category $\calc_A$. Of course, $A$ might itself be a spherical fusion category, in which case we simply write $\calc_A = \calc$.

One example of such an invariant is the Kuperberg invariant $\mathrm{Ku}_H$ associated with an involutary Hopf algebra $H$ over $\K$ and the spherical fusion category $\calc_H = \mathrm{mod}_H$ of finite-dimensional left $H$-modules. Another prominent example is the Dijkgraaf--Witten invariant $Z_{G, \nu}$ defined for an arbitrary finite group $G$ and the cohomology class $[\nu] \in H^3(G, \K^\times)$ of a 3-cocycle $\nu$. The spherical fusion category in question is the category $\calc_{(G, \nu)} = \vect_G^{\nu}$ of finite-dimensional $G$-graded vector spaces over $\K$. In the case of a trivial 3-cocycle, which is often referred to as `untwisted Dijkgraaf--Witten theory', these invariants will be studied in detail later in the text. Finally, for every spherical fusion category $\calc$, there is the Turaev--Viro state sum invariant $|\cdot|_{\calc}$.

A natural question to ask given an invariant $\omega_A$ is the following.
\begin{Question}
 \label{quest}
 Does $\omega_A$ admit a state sum description for the spherical fusion category $\calc_A$, that is, is it equal to the Turaev--Viro invariant $|\cdot|_{\calc_A}$?
\end{Question}
A striking feature of the Turaev--Viro invariant is that it extends to a topological quantum field theory (TQFT), which is a symmetric monoidal functor $|\cdot|_{\calc_A}\colon \Cob_3 \rightarrow \vect_{\C}$. A positive answer to Question~\ref{quest} then establishes the same for $\omega_A$, which means that the invariant can be computed locally by cutting the manifold into smaller pieces and using functoriality of the TQFT. A necessary requirement on $\omega_A$ is thus that it is multiplicative in the sense that it maps the disjoint union of two 3-manifolds to the product of the invariants of both pieces.

The Turaev--Viro invariant of a manifold $M$ is constructed and computed with the help of additional geometric data, in this paper more concretely skeleta $P \subset M$, which are certain stratified 2-polyhedra embedded in $M$ that give a combinatorial description of the manifold. Every skeleton $P$ can be equipped with a labeling of its 2-cells by simple objects in $\calc_A$. The invariant is then computed by drawing small balls around the vertices of the skeleton, whose boundary spheres are decorated with $\calc_A$-labeled graphs (\textit{link graphs}) capturing the local skeletal structure. These graphs are then evaluated and an appropriately weighted sum over all labelings is taken.

There is no reason to assume that the language of skeletons and link graphs is appropriate to describe a general invariant $\omega_A$.
One idea to introduce link graphs in such a general setting is to consider balls decorated with link graphs as independent geometric entities, namely cobordisms $\varnothing \rightarrow \varnothing$ in some suitably enlarged cobordism category $\Cob_3^{\calc_A}$. Since 3-balls have a nonempty boundary, $\Cob_3^{\calc_A}$ needs to allow for unparametrized boundary parts, called free boundaries, along which no composition (which is gluing in the cobordism category) is possible. Moreover, we want the boundary of the 3-ball to be decorated, so we allow $\calc_A$-labeled graphs embedded in the free boundary of any cobordism $M\colon \Sigma \rightarrow \Sigma'$.
We therefore make the following assumption: The invariant $\omega_A$ can be extended to allow for an evaluation on cobordisms in this category, which means that it extends to an invariant of oriented compact 3-manifolds with decorated free boundaries
\begin{align*}
 \omega_A\colon \ \Cob_3^{\calc_A}(\varnothing, \varnothing) \rightarrow \K.
\end{align*}
As long as we expect $\omega_A$ to have a state sum description, this is not an unreasonable assumption, since the Turaev--Viro TQFT itself can be extended to the category $\Cob_3^{\calc_A}$ (relying on Hypothesis~\ref{hyp} stating that the Turaev--Viro invariant is independent of the choice of skeleton, which was proved for defect structures in the interior of manifolds in \cite[Chapter~15]{TV}). The general concept is that of a \textit{defect} TQFT, and there are several such constructions in the literature, for instance, in \cite{M, TV}. We shall call~$\omega_A$ a \textit{boundary-defect invariant}, meaning that it allows not only free boundaries, but also boundary defects of the above kind. Generalizing to TQFTs, a \textit{boundary-defect TQFT} is a topological field theory that allows boundary defects (cf.~Definition~\ref{defTQFTDef}).

Starting from a closed 3-manifold $M$ with an embedded skeleton $P$, according to the principle outlined above, we would like to reduce $M$ to a collection of decorated 3-balls to compare the invariants $\omega_A$ and $|\cdot|_{\calc_A}$. Guided by the skeletal geometry, we can do this step by step, moving from the top-dimensional 3-cells (the connected components of $M \setminus P$) down to the 0-dimensional cells -- the vertices -- around which the 3-balls are located. For each codimension, we therefore require the existence of a corresponding geometric move decomposing the manifold. These moves are performed within the confines of the category $\Cob_3^{\calc_A}$ and, in particular, transform the free boundary: For the original manifold $M$, the first move removes a 3-ball from each 3-cell, thereby introducing free boundary 2-spheres. The second move breaks down the 2-cells of the skeleton and introduces defect lines, while the final move decomposes the 1-cells and introduces defect nodes at which the defect lines attach (cf. Section~\ref{subSecChar}). We are then left with a collection of decorated 3-balls, one for each vertex of the skeleton $P$.

We have to impose conditions that ensure that the invariant $\omega_A$ is compatible with these moves and the decorated 3-balls. For this purpose, we introduce the notion of \textit{boundary locality} (Definition~\ref{defPropInv}). This condition comprises four properties \ref{P1}, \ref{P2}, \ref{P3} and \ref{P4}, one for each move and a final one for decorated 3-balls. The first main result of this article is that these requirements are sufficient to answer Question~\ref{quest} positively.

\begin{Mainthm}[{Theorem~\ref{thmMain}}] \label{thmB}
 In the case that $\omega_A$ extends to a multiplicative boundary local invariant $\omega_A\colon \Cob_3^{\calc_A}(\varnothing, \varnothing) \rightarrow \K$, it follows that $\omega_A = |\cdot|_{\calc_A}$ on $\Cob_3^{\calc_A}(\varnothing, \varnothing)$, that is, it coincides with the Turaev--Viro invariant for $\calc_A$.
\end{Mainthm}
For the consistency of the notion of a boundary local theory, it is important to show that the state sum invariant $|\cdot|_{\calc}$ is boundary local for every spherical fusion category $\calc$.

This establishes a converse of Theorem~\ref{thmB}.
\begin{Mainthm}[{Theorem~\ref{thmMainTV}}]
 \label{thmC}
 The Turaev--Viro defect invariant $|\cdot|_{\calc}$ is boundary local for every spherical fusion category $\calc$.
\end{Mainthm}
Thus, the Turaev--Viro invariant is completely determined by the requirement of boundary locality. Relying on the compatibility with the moves transforming the defect boundary, we prove that not only the invariant, but in fact the entire boundary-defect TQFT $|\cdot|_{\calc}\colon \Cob_3^{\calc} \rightarrow \vect_{\K}$ is uniquely characterized by boundary locality.
\begin{Mainthm}[{Theorem~\ref{thmComb}}]
 \label{thmD}
 The Turaev--Viro boundary defect TQFT $|\cdot|_{\calc}$ is up to monoidal isomorphism the unique boundary local defect TQFT $\Cob_3^{\calc} \rightarrow \vect_{\K}$.
\end{Mainthm}

Turning to a concrete example, we study Dijkgraaf--Witten theory with trivial cocycle. In a~recent article~\cite{FMM}, Far\'{\i}a Martins and Meusburger construct untwisted Dijkgraaf--Witten theory as a TQFT with defects in all codimensions labeled by data associated to a finite group~$G$. From this description, we extract a boundary-defect TQFT $Z_G\colon \Cob_3^{\vect_G} \rightarrow \vect_{\C}$ in our sense, specializing to defects confined to the free boundary (cf. Section~\ref{secDW}). We obtain the subsequent result, which directly exhibits Dijkgraaf--Witten theories as state-sum TQFTs.
\begin{Mainthm} [{Theorem~\ref{thmDW}}]
 \label{thmE}
 The Dijkgraaf--Witten boundary-defect TQFT $Z_G$ for a finite group $G$ is boundary local, and therefore monoidally isomorphic to the Turaev--Viro boundary-defect theory $|\cdot|_{\vect_G}$ associated to the spherical fusion category of $G$-graded vector spaces $\vect_G$.
\end{Mainthm}

Our results make the slogan `bulk from boundary' more precise: Boundary locality comprises the conditions that ensure that state sum invariants and even topological field theories can be constructed from boundary data only, and that they are uniquely characterized by their transformation behavior under certain boundary moves. This perspective fits nicely with the known fact that bulk Wilson lines are naturally labeled by objects in a Drinfeld center, which is a quantity derived from the boundary labeling datum, a spherical fusion category. Related ideas have been put forward in several contexts and confirm the validity of this point of view. For instance, similar moves and their implications on state sum models are prominently featured in \cite{CGPT, LMWW, LMWW2, W}.

The main application pursued in this article is to Dijkgraaf--Witten theories with boundary defects. The results allow us to understand them as state sum constructions without the necessity to compute the two relevant theories explicitly and to set up an explicit comparison between them.

We outline the structure of this article. In Section~\ref{secPrel}, we summarize the pertinent background on spherical fusion categories and topological field theories (Sections~\ref{subsecCatCon} and~\ref{subSecTQFTs}) and recapitulate the explicit constructions of Turaev--Viro and Dijkgraaf--Witten theory with defects (Sections~\ref{subSecTV} and~\ref{subSecDW}) used in the sequel. Our main results are contained in Section~\ref{secTVRes}. We formulate the geometric moves for boundary locality in the cobordism category $\Cob_3^{\calc}$ with free defect boundaries and the corresponding algebraic properties and prove Theorem~\ref{thmMain} (Theorem~\ref{thmB} in this section) in Section~\ref{subSecChar}. In Section~\ref{subSecProof}, we prove Theorem~\ref{thmMainTV} (Theorem~\ref{thmC}) by showing that the Turaev--Viro TQFT satisfies the algebraic properties \ref{P1} through \ref{P4} one by one. Theorem~\ref{thmComb} (Theorem~\ref{thmD}) is proved in Section~\ref{subSecTQFTs} by establishing that the boundary-defect Turaev--Viro TQFT is non-degenerate (Proposition~\ref{propND}), which is achieved using boundary locality of the state sum. Section~\ref{subSecApp} contains smaller applications and comments and may be skipped in a first reading. In Section~\ref{secDW}, we study Dijkgraaf--Witten theory with boundary defects in detail. The purpose of Section~\ref{subSecFB} is to explain how free boundaries are introduced in the setting of~\cite{FMM} and how a defect TQFT with labels in the spherical fusion category of $G$-graded vector spaces is obtained in the sense of Definition~\ref{defTQFTDef}. Finally, in Section~\ref{subSecDWProp}, parallel to Section~\ref{subSecProof}, we establish the properties~\ref{P1} to~\ref{P4} for Dijkgraaf--Witten theory and prove Theorem~\ref{thmDW} (Theorem~\ref{thmE}).

The lemma and theorem numbers in references cited throughout the text refer to the arXiv version, whenever both an arXiv and a published version are available.

\section{Preliminaries}
\label{secPrel}
In this section, we set our general notation and recall the TQFT constructions that are relevant for the article. The main reference consulted for the Turaev--Viro construction is~\cite{TV}, and for the case that includes free boundaries~\cite{F} and \cite{FS}.\footnote{The reference~\cite{F} is available upon request from the second author. Its most relevant aspects are also discussed in the published article~\cite{FS}.} We will make use of a mix of the notation and conventions employed in these texts. For Dijkgraaf--Witten theory, we will mostly rely on a~recent defect construction presented in~\cite{FMM} and recall it to the level of detail we need at the end of this section.

\subsection{Spherical fusion categories}
\label{subsecCatCon}
Let $\mathcal{X}$ be a category. We will denote by $\mathcal{X}(X,Y) = \Hom_{\mathcal{X}}(X,Y)$ the set of morphisms between two objects $X, Y \in \mathcal{X}$. If $X = Y$, we will further write its set of endomorphisms as $\End_{\calx}(X)$. For an object $X \in \mathcal{X}$, we denote by $[X]$ its isomorphism class in $\mathcal{X}$ and by $\pi_0(\mathcal{X})$ the set of all isomorphism classes in $\calx$.

For another category $\mathcal{Y}$, we denote by $\Fun(\mathcal{X}, \mathcal{Y})$ the category of functors between $\calx$ and $\caly$ and by $\mathrm{Nat}(F, G)$ the set of natural transformations between two functors $F, G \in \Fun(\calx, \caly)$. If the categories $\calx$ and $\caly$ are $\K$-linear for some field $\K$, the set of natural transformations $\mathrm{Nat}(F, G)$ inherits the structure of a $\K$-vector space.

Let $(\calx, \otimes, \mathbbm{1})$ denote a monoidal category. In all of the examples occurring in this article, it is natural to suppress associators and left and right unitors, so we will refrain from introducing notation for them and implicitly assume strictness of the monoidal category $\calx$. Monoidal functors between monoidal categories $\mathcal{X}$ and $\mathcal{Y}$ will simply be written as $F\colon \calx \rightarrow \caly$, where we drop the structure isomorphisms relating the monoidal structure of the two categories.

All monoidal categories considered in the text will be rigid and equipped with a pivotal structure. Thus, we may identify left and right duals and use for the dual of an object $X \in \calx$ the notation $X^*$ with the left and right evaluation and coevaluation pairings
\begin{alignat*}{3}
 & \mathrm{ev}_X\colon \ X^* \otimes X \rightarrow \mathbbm{1}, \qquad && \mathrm{coev}_X\colon \ \mathbbm{1} \rightarrow X \otimes X^*, & \\
 & \widetilde{\mathrm{ev}}_X\colon \ X \otimes X^* \rightarrow \mathbbm{1}, \qquad && \widetilde{\mathrm{coev}}_X\colon \mathbbm{1} \rightarrow X^* \otimes X.&
\end{alignat*}
Furthermore, we introduce for an object $X$ the notion of its left and right dimensions $\dim_{\ell}(X)$ and $\dim_r(X)$, defined as the endomorphisms $\dim_{\ell}(X) := \mathrm{ev}_X \circ \widetilde{\mathrm{coev}}_X$ and $\dim_r(X) = \widetilde{\mathrm{ex}}_X \circ \mathrm{coev}_X$ in $\End_{\calx}(\mathbbm{1})$. More generally, for an endomorphism $f \in \End_{\calx}(X)$, we define its left trace as $\tr_{\ell}(f) = \mathrm{ev}_X \circ (\id_{X^*} \otimes f) \circ \widetilde{\mathrm{coev}}_X \in \End_{\calx}(\mathbbm{1})$ and analogously its right trace $\tr_r(f)$.

We denote by $\K$ an arbitrary algebraically closed field of characteristic zero. A \textit{fusion category}~$\calc$, following the conventions in~\cite{EGNO}, is a semisimple $\K$-linear finite abelian rigid monoidal category such that $\mathbbm{1}$ is a simple object in~$\calc$. By Schur's lemma, we can in particular identify $\End_{\calc}(\mathbbm{1}) \cong \K$, which we will do implicitly in the rest of the text. We use the notation $\calo(\calc) \subset \pi_0(\calc)$ for the set of isomorphism classes of simple objects. By the requirements above, this is a finite set. By abuse of notation, $\calo(\calc)$ will also stand for a specific (but arbitrary) choice of representatives of all simple objects in~$\calc$.

A \textit{spherical} (fusion) category $\calc$ is a (fusion) category in which the left and right traces of endomorphisms $f$ coincide, which we will denote simply by $\tr(f) = \tr_{\ell}(f) = \tr_r(f)$. In particular, for every object $X \in \calc$, its left and right dimensions of $X$ coincide, and we will simply speak of the dimension of $X$ and write $\dim(X) \in \K$.
The $\K$-linear structure on $\calc$ allows us to further define the categorical dimension
\begin{align*}
 \dim(\calc) = \sum_{x \in \calo(\calc)}\dim(x)^2 \in \K.
\end{align*}
It is a well-known fact that for $\K$ an algebraically closed field of characteristic zero, the categorical dimension $\dim(\calc)$ is nonzero (see, e.g., \cite[Theorem~7.21.12]{EGNO}).

Typical examples for spherical fusion categories which will show up repeatedly in this work are the category of finite-dimensional $\K$-vector spaces $\vect_{\K}$, the category of finite-dimensional $G$-graded $\K$-vector spaces $\vect_G$ and the category of finite-dimensional $\K$-linear representations $\Rep_{\K}(G)$. Here, $G$ denotes a finite group. The categorical dimensions of these categories are $\dim(\vect_{\K}) = 1$ and $\dim(\vect_G) = \dim(\Rep_{\K}(G)) = |G|$.

For $\vect_{\K}$, the rigid structure is defined by the (left) coevaluation morphism $\mathrm{coev}_{V}\colon \K \rightarrow V \otimes_{\K} V^*$, given by $\mathrm{coev}_V(1) = \sum_{i = 1}^k\alpha_i \otimes \alpha_i^*$ and the obvious (left) evaluation $\mathrm{ev}_V\colon V^* \otimes_{\K} V \rightarrow \K$. Here, $\{\alpha_i\}_{i = 1}^k$ is a basis of $V$ and $\{\alpha_i^*\}_{i = 1}^k$ denotes the corresponding dual basis of $V^*$, which says that $\alpha_i^*(\alpha_j) = \delta_{i,j}$. In this situation, we will speak of a \textit{pair of dual bases}. In particular, the case $V = \Hom_{\calc}(X, Y)$ for objects $X, Y \in \calc$ will be of interest in the construction and evaluation of the Turaev--Viro theory. The trace induces in any spherical fusion category $\calc$ isomorphisms
\begin{align}
 \label{eqDual}
 \Hom_{\calc}(X, Y)^* \cong \Hom_{\calc}(X^*, Y^*) \cong \Hom_{\calc}(Y, X).
\end{align}
Identification of the dual basis $\{\alpha_i^*\}_{i = 1}^k$ with its image under this isomorphism translates the dual bases relation above to $\tr(\alpha_i^* \circ \alpha_j) = \delta_{i,j}$. Assuming that $X$ is simple, denoting it by $x$ and using $\End_{\calc}(x) \cong \K$, this implies that $\alpha_i^* \circ \alpha_j = \frac{\delta_{i,j}}{\dim(x)}\id_x$.\footnote{Note that our normalization convention differs from the one in \cite[Section~4.4.1]{TV}. There, the authors use the notion of an $x$-partition of $Y$, which is a pair of bases $\{p_{\alpha}\}$, $\{q_{\alpha}\}$ of $\Hom_{\calc}(Y, x)$ (respectively $\Hom_{\calc}(x, Y)$) such that $p_{\alpha} \circ q_{\beta} = \delta_{\alpha, \beta}\id_x$ for all $\alpha$, $\beta$.} In particular, for $x = \mathbbm{1}$, it follows that $\alpha_i^* \circ \alpha_j = \delta_{i,j} \id_{\mathbbm{1}}$. Equation~(4.9) in~\cite{TV} now states that for any morphism $f \in \Hom_{\calc}(\mathbbm{1}, Y)$ from the tensor unit to the object~$Y$, the relation
\begin{align}
\label{eqTV49}
 \sum_{i = 1}^k \alpha_i \circ \alpha_i^* \circ f = f
\end{align}
holds, where $k = \dim_{\K}\left(\Hom_{\calc}(\mathbbm{1}, Y)\right)$ as introduced above. Recall also that for any cyclic permutation $\sigma \in \mathfrak{S}_n$, there exists an isomorphism
\begin{align}
 \label{eqCycl}
 \Hom_{\calc}(\mathbbm{1}, X_1 \otimes \dots \otimes X_n) \cong \Hom_{\calc}(\mathbbm{1}, X_{\sigma(1)} \otimes \dots \otimes X_{\sigma(n)})
\end{align}
induced by the pivotal structure of $\calc$.

We will use standard conventions for the graphical calculus in pivotal categories. In particular, by the  isomorphism~\eqref{eqCycl}, coupons representing morphisms in $\calc$ can be drawn in a circular shape. This is admissible whenever there is no ambiguity as to which morphism is represented by the coupon. Often, we will implicitly sum over a pair of dual bases of `dual' coupons by leaving out indices and sums,\footnote{For details regarding these conventions, see~\cite{FS} and~\cite{BK}. For a careful treatment and generalities on the graphical calculus, see~\cite{TV}.} as in the following picture, where the depicted morphisms $\alpha_i \otimes \alpha_i^*$ are in $\Hom_{\calc}({X \otimes Y}, Y \otimes X)$:
\begin{align*}
 \sum_{i = 1}^k \adjincludegraphics[valign=c, scale = 0.75]{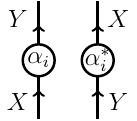} = \adjincludegraphics[valign=c, scale = 0.75]{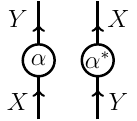}.
\end{align*}
We cite without proof the following lemmas from~\cite{BK} and~\cite{FS}, which will be very useful later when we work explicitly with the Turaev--Viro construction (see Section~\ref{subSecTV}). Here, pairs of lower case Greek letters always represent pairs of dual bases.
\begin{Lemma}[{\cite[Lemma~4]{FS}}]
 \label{lemTrace}
 Let $X, Y \in \calc$ be two objects and consider a linear endomorphism $F\colon \Hom_{\calc}(X, Y) \rightarrow \Hom_{\calc}(X, Y)$. The trace $\tr(F)$ of the map~$F$ can be expressed as
 \begin{align*}
 \tr(F)   =   \adjincludegraphics[valign=c, scale = 0.75]{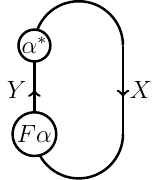}.
 \end{align*}
\end{Lemma}

\begin{Lemma}[{\cite[Lemma 1.1]{BK}}]\label{lemDecId}
For $X \in \calc$, we can rewrite the identity $\id_X \in \End_{\calc}(X)$ as
 \begin{align*}
 \sum_{i \in \calo(\calc)} \dim(i) \adjincludegraphics[valign=c, scale = 0.75]{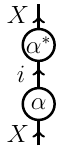} \,  =   \adjincludegraphics[valign=c, scale = 0.75]{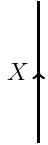}.
 \end{align*}
\end{Lemma}

\begin{Lemma}[{\cite[Lemma 1.3]{BK}}]\label{lemDisconn}
We have the following identity for an object $X \in \calc$ and morphisms $f \in \Hom_{\calc}(\mathbbm{1}, X)$, $g \in \Hom_{\calc}(X, \mathbbm{1})$, connecting two previously disconnected string diagrams in $\calc$:
\begin{align*}
 \adjincludegraphics[valign=c, scale = 0.75]{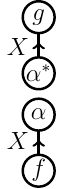} \,  =   \adjincludegraphics[valign=c, scale = 0.75]{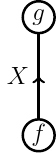}.
\end{align*}
\end{Lemma}

\begin{Lemma}[{\cite[Lemma~7]{FS}}]\label{lemQuad}
 We have the following rule for string diagrams:
 \begin{align*}
\sum_{i \in \calo(\calc)} \dim(i) \adjincludegraphics[valign=c, scale = 0.75]{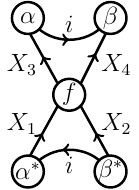} \, \, = \, \, \adjincludegraphics[valign=c, scale = 0.75]{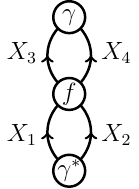},
 \end{align*}
where ${X_1, \dots, X_4 \in \calc}$ are four objects and $f$ is a morphism in $\Hom_{\calc}(X_1 \otimes X_2, X_3 \otimes X_4)$
\end{Lemma}

\begin{Lemma}[{\cite[Lemma 8]{FS}}]\label{lemBub}
For $X \in \calc$ and $f \in \End_{\calc}(X)$ an endomorphism, it holds that
 \begin{align*}
\sum_{i \in \calo(\calc)} \dim(i) \adjincludegraphics[valign=c, scale = 0.75]{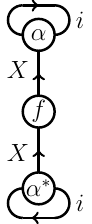} =   \dim(\calc) \adjincludegraphics[valign=c, scale = 0.75]{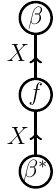}.
 \end{align*}
\end{Lemma}
Obviously, by the rules of the graphical calculus, these identities still hold if any of the diagrams is replaced by an isotopic version of it.

\subsection{Cobordism categories and boundary-defect TQFTs}
\label{subSecTQFTs}
In this subsection, we recall the necessary details and notational conventions on cobordism categories, topological field theories and certain auxiliary geometric data, skeleta of manifolds. The standard reference for this section is \cite{TV}, while the geometric generalizations discussed in this section were introduced in \cite{F} and \cite{FS}.

We start by giving the definition of the category of cobordisms in three dimensions. The term $n$-manifold will stand for a compact oriented $n$-dimensional topological manifold (unless stated otherwise), possibly with boundary.
In the case $n = 2$, we will also call a 2-manifold a surface. The orientation of a 3-manifold $M$ with boundary induces an orientation of its boundary surface~$\partial M$ such that an outward-pointing normal vector followed by the orientation of $\partial M$ recovers the one on~$M$.
\begin{Definition}\label{defCob}
The category $\Cob_3$ is the category with objects given by closed surfaces $\Sigma$ and morphisms being equivalence classes of 3-manifolds $M$ with a boundary decomposition into source and target manifolds.

More precisely, a morphism $\Sigma \rightarrow \Sigma'$ in $\Cob_3$ is an equivalence class of pairs $(M, h)$, called \textit{cobordisms}, of a compact 3-manifold $M$ together with an orientation-preserving homeomorphism $h\colon \overline{\Sigma} \sqcup \Sigma'\! \rightarrow \partial M$ that parametrizes the boundary of~$M$. Two cobordisms $(M,h), (M',h')\colon {\Sigma \rightarrow \Sigma'}$ are equivalent (and hence define the same morphism) whenever there is an orientation-preserving homeomorphism $f\colon M \rightarrow M'$ such that $f \circ h = h'$. From here on, we will omit the datum $h$ from the notation, treating it as implicitly provided by context whenever considering specific examples.

 Composition in $\Cob_3$ works by gluing: If $M \in \Cob_3(\Sigma, \Sigma')$ and $M' \in \Cob_3(\Sigma', \Sigma'')$, then $M' \circ M = M' \cup_{\Sigma'} M \in \Cob_3(\Sigma, \Sigma'')$. The identity cobordism on a surface $\Sigma$ can be represented by the cylinder $\Sigma \times I$, where $I = [0,1]$ denotes the unit interval.
\end{Definition}

The category $\Cob_3$ carries a natural symmetric monoidal structure given by disjoint union $\sqcup\colon \Cob_3 \times \Cob_3 \rightarrow \Cob_3$ and unit object $\varnothing$, the empty manifold. It is rigid and spherical, with duals given by reversing orientation, $\Sigma^* = \overline{\Sigma}$. For further details, see \cite[Section 10.1]{TV}. Similarly, one defines cobordism categories $\Cob_n$ in any dimension $n$.
\begin{Remark}
 \label{remTopPL}
 Analogous definitions exist for cobordism categories of smooth or piecewise linear manifolds. In three dimensions, all of these categories are monoidally equivalent, so we adopt the simplest choice of topological manifolds in this text. Nonetheless, as is standard, we will freely make use of concepts specific to differential manifolds only, such as normal vectors.
\end{Remark}
The following definition is standard.
\begin{Definition}
 \label{defTQFT}
 A \textit{topological $($quantum$)$ field theory (TQFT)} is a symmetric monoidal functor
 \begin{align*}
 Z\colon \ \Cob_3 \rightarrow \vect_{\K}
 \end{align*}
 between the category of cobordisms and the category of finite-dimensional $\K$-vector spaces.
\end{Definition}
In the TQFT-literature, various flavors of cobordism categories have been considered. For the purpose of this paper, we are interested in a slightly more general version of the standard cobordism category, incorporating a certain class of topological defects and corresponding algebraic labelings. We no longer demand surjectivity of the parametrizing maps as in Definition~\ref{defCob}, and may therefore consider defect graphs sitting in the unparametrized part of the boundary of generalized cobordisms, the free boundary, with labelings coming from a spherical fusion category~$\calc$.

In the following, we briefly outline the concept of the category $\Cob_3^{\calc}$ consisting of $\calc$-colored surfaces and cobordisms, which, for our purposes, is restricted to boundary defect structures. These concepts are natural generalizations of those discussed in~\cite{TV}, to which we refer the reader for additional details. Recall the following.
\begin{itemize}\itemsep=0pt
 \item (\cite[Definition 4.2.5]{F}, generalizing \cite[Section 15.1.1]{TV}) A $\calc$-\textit{colored surface} is a not necessarily closed surface $\Sigma$ with finitely many points inserted in the boundary $\partial \Sigma$ labeled by objects of $\calc$ and a sign. These points are called \textit{colored points}.
 \item (\cite[Definition 4.2.7]{F}, \cite[Section 15.1.2]{TV}) A $\calc$-\textit{colored graph} is an oriented graph $\Gamma$, regarded as a topological space, with rectangular coupons with distinguished bottom and top sides as vertices. The edges of $\Gamma$ are labeled by objects in $\calc$, and the vertices/coupons by morphisms in $\calc$.\footnote{Note that in \cite{F} and similarly in \cite{TV}, this object would be called a $\calc$-colored plexus. In this article, we go with the simpler terminology of a graph.\label{footNoLoops}} The edges connect to the top and bottom of the coupons.

 We also allow edges with free ends -- that is, edges that are attached to a coupon with only one end. A $\calc$-colored graph $\Gamma$ can be embedded in a surface $\Sigma$, and we will always assume that free ends embed into the boundary $\partial \Sigma$.

 Given such an embedding $\Gamma \subset \Sigma$, we will denote by $V(\Gamma), E(\Gamma)$ and $P(\Gamma)$ the sets of vertices, edges and plaquettes (faces), respectively.\footnote{Strictly speaking, we do not allow isolated loops consisting of a single edge with no vertex attached. Whenever loops labeled by some object $X \in \calc$ appear in the main text, it is implicitly understood that they are decorated with a `transparent' coupon labeled by $\id_X$. Since this has no effect on the values of the TQFTs we are considering, we may (and will) disregard this vertex in practice.}
\end{itemize}
We introduce some further notation for later use. Given an embedded graph $\Gamma \subset \Sigma$, we denote for a vertex $v \in V(\Gamma)$ by $E(v)$ and $P(v)$ the sets of edges and plaquettes adjacent to $v$, and for a plaquette $p \in P(\Gamma)$, we denote by $V(p)$ and $E(p)$ the vertices and edges adjacent to $p$.

\begin{Definition}\label{defCobDef}
The category $\Cob_3^{\calc}$ of $\calc$-colored cobordisms has as objects $\calc$-colored surfaces~$\Sigma$. Its morphisms generalize the notion of a cobordism from Definition~\ref{defCob} as follows:

Given two surfaces $\Sigma$, $\Sigma'$, a cobordism consists of a pair $(M, h)$, where $M$ is a 3-manifold and $h\colon \overline{\Sigma} \sqcup \Sigma' \hookrightarrow \partial M$ an orientation-preserving \textit{embedding}, referred to as the parametrization map. We call $\partial_gM := \mathrm{im}(h)$ the gluing boundary and its complement $\partial_fM = \partial M \setminus \partial_gM$ in $\partial M$ the free boundary of the manifold $M$.

A morphism in $\Cob_3^{\calc}(\Sigma, \Sigma')$ is then an equivalence class of cobordisms $(M, h)$, together with an embedding of a $\calc$-colored graph $\Gamma$ into the closure $\mathrm{cl}(\partial _f M)$ of the free boundary, such that~$h$ identifies the colored points of the $\calc$-colored surface $\overline{\Sigma} \sqcup \Sigma'$ with the free ends of $\Gamma$ in $\partial_gM$.

More precisely, the embedding of $\Gamma$ endows the gluing boundary $\partial_gM$ with the structure of a $\calc$-colored surface, by the convention that edges of $\Gamma$ associated with free ends terminating in~$\partial_gM$ are oriented away from colored points with negative sign and toward those with positive sign (this convention is opposite to that used in \cite{F}). The map $h$ is then required to preserve the positions, labels and orientations of the colored points.

 Two such cobordisms are equivalent whenever there is an orientation-preserving homeomorphism compatible with the parametrization maps, that maps one $\calc$-colored graph to the other.

 We will usually denote morphisms as $M_{\Gamma} \in \Cob_3^{\calc}(\Sigma, \Sigma')$, omitting the parametrization map and sometimes even the graph $\Gamma$ from the notation. By abuse of terminology, we will refer to morphisms simply as cobordisms.
 Composition in this category is defined by gluing the underlying topological structures. The identity cobordism for a $\calc$-colored surface $\Sigma$, with colored points $p_1, \dots, p_k$, is represented by the cylinder $\Sigma \times I$, decorated with the defect graph $\{p_1, \dots, p_k\} \times I$ with labels and orientations inherited from the colored points of $\Sigma$.
\end{Definition}
The category $\Cob_3^{\calc}$ can be equipped with a natural monoidal structure by disjoint union. Duals of $\calc$-colored surfaces are defined by swapping orientations and signs of colored points, and this assignment turns $\Cob_3^{\calc}$ into a spherical category.

We can now generalize the concept of a TQFT from Definition~\ref{defTQFT} to that of a defect TQFT~-- more precisely, a TQFT with (certain) boundary defects.
\begin{Definition}\label{defTQFTDef}
 A \textit{topological quantum field theory with boundary defects} is a symmetric monoidal functor $Z\colon \Cob_3^{\calc} \rightarrow \vect_{\K}$ from the category of $\calc$-colored cobordisms into the category of finite-dimensional vector spaces over $\K$.
\end{Definition}
For brevity, we will often use the shorthand \textit{boundary-defect TQFT} or simply \textit{defect TQFT}.
Hence, for our purposes, a defect theory knows how to assign values to manifolds with defects of dimension~0 (the vertices) and~1 (the edges) embedded in the (free) boundary. In the literature, significantly more general notions of defect theories exist, allowing for defects of every codimension including those in the interior of manifolds. As indicated earlier, the restricted class of defects considered in this paper suffices for our results and in particular for stating Theorem~\ref{thmMain}.\looseness=1

There is a natural inclusion functor $\Cob_3 \hookrightarrow \Cob_3^{\calc}$ that embeds the category of cobordisms into the category of $\calc$-colored cobordisms. Recall that in the former, no free boundaries are allowed. Any boundary-defect TQFT $\Cob_3^{\calc} \rightarrow \vect_{\K}$ can be restricted along the inclusion to obtain a TQFT (without defects) $\Cob_3 \rightarrow \vect_{\K}$.
We will encounter further variants of cobordism categories when we discuss Dijkgraaf--Witten theory in Section~\ref{subSecDW}.

Instead of working with the entire cobordism category $\Cob_3^{\calc}$, we may restrict attention to the morphism set $\Cob_3^{\calc}(\varnothing, \varnothing)$.

\begin{Definition}\label{defInv}
 We call a map $\omega\colon \Cob_3^{\calc}(\varnothing, \varnothing) \rightarrow \K$ a \textit{boundary-defect invariant} of $\calc$-colored 3-manifolds. If the invariant respects the monoidal product $\sqcup$ in $\Cob_3^{\calc}$ in the sense that it satisfies $\omega(M \sqcup N) = \omega(M) \cdot \omega(N)$, we call $\omega$ \textit{multiplicative}.
\end{Definition}
The terminology `invariant' for any such map $\omega$ is justified by the fact that $\Cob_3^{\calc}(\varnothing, \varnothing)$ already consists of equivalence classes of cobordisms. Thus, any map from $\Cob_3^{\calc}(\varnothing, \varnothing)$ to the ground field yields a genuine topological invariant of 3-manifolds with a $\calc$-colored boundary graph.
\begin{Example}
 Clearly, any defect TQFT $Z\colon \Cob_3^{\calc} \rightarrow \vect_{\K}$ defines a multiplicative defect invariant by setting $\omega = Z|_{\Cob_3^{\calc}(\varnothing, \varnothing)}$.
\end{Example}

We recall some auxiliar geometric structures embedded in cobordisms.
\begin{Definition}[{\cite[Definition~4.2.13]{F}, \cite[Section~15.7.1]{TV}}]
 \label{defSkel}
 A \textit{skeleton} $G$ of a $\calc$-colored surface~$\Sigma$ is an embedded oriented graph containing the boundary of~$\Sigma$, i.e., $\partial \Sigma \subset G$, such that the following conditions are satisfied: Every colored point $p \in \partial\Sigma$ is contained in the interior of an edge of~$G$, all vertices have valence greater than or equal to~2 and $\Sigma \setminus G$ is a disjoint union of open disks.
\end{Definition}
Thus, skeleta can be seen as combinatorial descriptions of surfaces. We need a similar description for 3-manifolds. This is somewhat more involved, as the following definition shows.
\begin{Definition}[{\cite[Definitions 4.2.17 and 4.2.18]{F}, \cite[Sections 11.1.2 and 11.5.1]{TV}}]
 \label{defGraphSkel}
 Let $M_{\Gamma}$ be a cobordism in $\Cob_3^{\calc}$ and $G$ a skeleton of its gluing boundary $\partial_gM$. A \textit{(graph) skeleton} $P$ of the pair $(M_{\Gamma}, G)$ is a stratified 2-polyhedron suitably embedded in the underlying mani\-fold~$M$. More precisely, by definition of a stratified 2-polyhedron, this says that $P$ is a compact space that can be triangulated with cells of dimension smaller than or equal to 2, and comes with a~distinguished stratification defined by a graph $P^{(1)} \subset P$. By definition, the stratification is such that all points that do not admit an open neighborhood in $P$ homeomorphic to $\mathbb{R}^2$ are contained in the graph $P^{(1)}$.

The connected components of $P\setminus P^{(1)}$ are called \textit{regions} of $P$ and the set containing them is denoted by~$\mathrm{Reg}(P)$.
Every region is endowed with an orientation.
A \textit{branch} of~$P$ is a homotopy class of paths starting from the interior of an edge and ending in $P \setminus P^{(1)}$.

 The skeleton $P$ is required to satisfy several conditions in relation to $M$ and $G$: It has to be compatible with $G$ in the sense that $\partial P = G$ as oriented graphs, where by $\partial P \subset P$ we denote the graph given by edges with only one branch (and induced orientation) and their vertices, and it must be $\partial$-cylindrical (a technical condition ensuring that $\partial P$ has a neighborhood homeomorphic to $\partial P \times [0,1]$, see \cite[Section~11.1.4]{TV} for details). Moreover, $P \setminus \partial P$ lies in the complement $M \setminus \partial _gM$, the skeleton covers the free boundary, $\partial _f M \subset P$, and $M \setminus P$ is a disjoint union of open 3-balls and a rest homeomorphic to $\partial_gM \setminus G \times [0,1)$. Finally, we demand that $P$ behaves well with respect to the $\calc$-colored graph $\Gamma$ embedded in $\partial_fM$: For this, we require that the intersection points of $\Gamma$ and $P^{(1)}$, called switches, form a finite set, and that no vertex of $P^{(1)}$ or $\Gamma$ is in $P^{(1)} \cap \Gamma$.
\end{Definition}
\begin{Remark}
 Technically, in \cite{F} and \cite{TV}, graph skeleta are defined not in relation to $\calc$-colored graphs embedded in $M$ (or plexuses in their terminology), but to certain embeddings of plexuses contained entirely in $P$, called \textit{knotted plexuses}. These represent a $\calc$-colored graph $\Gamma$ via an ambient isotopy (here in $\partial_f M$). This approach is primarily motivated by the inclusion of additional internal defect lines, which are not required to lie in the skeleton from the outset. Since we consider merely boundary defect structures, which are already contained in $P$, we ignore this distinction and tacitly assume that a suitable homotopy has been applied to $\Gamma$ if necessary.
\end{Remark}
For a given cobordism $M_{\Gamma} \colon \Sigma \to \Sigma'$ in $\Cob_3^{\calc}$, existence of a graph skeleton can be shown using
\cite[Theorem 11.5]{TV}.\footnote{We are grateful to one of the anonymous referees for pointing out
the role of \cite[Theorem 11.5]{TV} in this context.}
More precisely, for empty gluing boundary, one picks a triangulation~$t$ of~$\partial M$, which, after possibly applying a suitable isotopy, satisfies the requirement that all vertices of~$\Gamma$ are disjoint from the 1-skeleton $t^{(1)}$ of $t$ and that all edges of $\Gamma$ meet (transversely) only the interiors of edges of $t^{(1)}$. Since the valence of vertices in $t$ is at least~2, applying \cite[Theorem~11.5]{TV} to the pair $\bigl(M, t^{(1)}\bigr)$ produces a 2-polyhedron that gives, after adjoining the faces of~$t$, a valid graph skeleton for $M_{\Gamma}$. With some extra care, this argument extends to the situation with a nonempty gluing boundary.

For a defect cobordism $M_{\Gamma}$ and a graph skeleton $P$ chosen with respect to a skeleton~$G$ of~$\partial_gM$, we fix terminology and notation for geometric substructures related to the combined graph $P^{(1)} \cup \Gamma$.

The switches together with the vertices of $P^{(1)}$ and $\Gamma$ are referred to as nodes, the connected components of the complement of the nodes in $P^{(1)} \cup \Gamma$ are the rims, and the connected components of $P \setminus \bigl(P^{(1)} \cup \Gamma\bigr)$ are the faces. We call a rim or node \textit{internal} if it does not lie inside the gluing boundary and denote the sets of internal rims of~$P$ by~$R(P)$ and the complete set of nodes by~$N(P)$. The latter set decomposes into the set of nodes $N_{\partial}$ in the gluing boundary~$\partial_gM$ and the internal nodes $N_0$ (which are allowed to lie inside the free boundary). The set of faces, which subvivide the regions contained in $\mathrm{Reg}(P)$, is denoted by $\mathrm{Fac}(P)$ and can be decomposed into the set of faces $\mathrm{Fac}_{\partial}(P)$ adjacent to the gluing boundary and the fully internal faces $\mathrm{Fac}_0(P)$. Branches with respect to faces are defined analogously as with respect to regions.

Each rim $r$ also defines two half-rims. We denote by $R^h$ the set of internal half-rims, which are those that do not lie in the gluing boundary $\partial_gM$, by $R_{\partial}^h \subset R^h$ the subset of half-rims adjacent to the gluing boundary and by $R_0^h$ its complement in $R^h$, the internal rims.

\begin{Example}
 The following figure depicts a local portion of a skeleton $P$ for some cobordism:
$$
 \adjincludegraphics[valign=c, scale = 0.85]{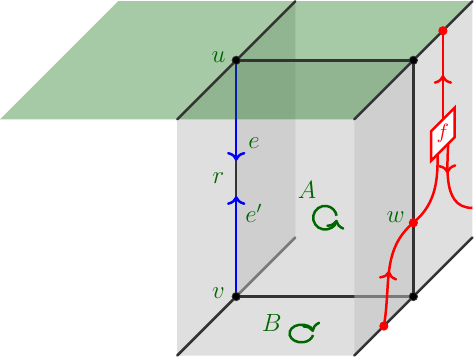}
$$
 Here, the gluing boundary is drawn in green. Note that by definition, none of the green faces are a part of the skeleton, while the gray faces are. For example, $A \in \Fac_{\partial}(P)$ (whose orientation is indicated by the loop) is a face (and simultaneously a region) adjacent to the gluing boundary, while $B \in \Fac_0(P)$ labels a fully internal face. The two rightmost regions form a part of the free boundary of the manifold and are decorated by some defect graph. Both of these regions are subdivided into several faces, with some of them adjacent to the gluing boundary.

 There are three (parts of) red defect strands shown, that decompose into four rims. The internal rim labeled $r \in R(P)$ is adjacent to the gluing boundary and distinguishes two half rims $e, e' \in R^h$ drawn in blue. Here, $e \in R_{\partial}^h$ and $e' \in R_0^h$. There are also three (parts of) rims shown which lie inside the gluing boundary and thus are not in $R(P)$.

 Finally, all four types of nodes appear in the image: For the two labeled vertices of~$P$, $u \in N_{\partial}$ is a vertex in the gluing boundary and $v \in N_0$ is an internal node. There are several switches, one of which is labeled~$w$. Finally, the coupon labeled by the morphism $f$ counts as a node in~$N(P)$ as well.
\end{Example}

\subsection{The Turaev--Viro construction with boundary defects}
\label{subSecTV}
In this section, we briefly outline the construction of Turaev--Viro theory generalized to the setting of free boundaries and surface defect graphs. The outcome is a boundary-defect TQFT $|\cdot|_{\calc}\colon \Cob_3^{\calc} \rightarrow \vect_{\K}$ constructed using geometric data, the graph skeleta introduced in Definition~\ref{defSkel}. A priori, for a cobordism $M$, the linear map $|M|_{\calc}$ depends on a chosen graph skeleton for $M$. We rely on the following hypothesis.
\begin{Hypothesis}
 \label{hyp}
 For every cobordism $M\colon \Sigma \rightarrow \Sigma'$ in $\Cob_3^{\calc}$,the linear map $|M|_{\calc}\colon |\Sigma|_{\calc} \rightarrow |\Sigma'|_{\calc}$ constructed below is independent of the choice of skeleton of $M$.
\end{Hypothesis}
We formulate this as a hypothesis, since the proof requires a subtle analysis of topological moves relating different graph skeleta in the presence of free boundaries and defect lines, for which the authors are not experts. In the case of internal defect lines and in the absence of free boundaries, this assumption has been proved in \cite[Theorems~15.7 and~15.8]{TV}.

A similar situation appears in alterfold theory, which is based
on handle decompositions of a 3\nobreakdash-manifold rather than skeleta. In this context,
independence of the choice of handle decomposition has been demonstrated
(cf. \cite[Theorems~3.9 and~3.21, Remark~3.22]{LMWW}). Unfortunately, the
relation between skeleta and handle decompositions is somewhat subtle, which can
already be seen from the fact that 2-cells in a skeleton may have nontrivial topology while
1-handles (the natural duals of skeletal 2-cells) do not. As a consequence, a direct adaptation of the arguments
of \cite{LMWW} is beyond the scope of this article. Still, alterfold
theory provides a nontrivial substantiation of Hypothesis~\ref{hyp}.\footnote{We thank one of the anonymous referees for suggesting this argument.}

We now recall the construction of $|\cdot|_{\calc}$ and set the necessary conventions that will be relevant in Section~\ref{secTVRes}. Once again, we refer to the sources \cite{TV, FS} for a detailed presentation of the TQFT-construction. The exposition here follows closely that of \cite{F}, which in turn is based on \cite[Chapters 12--15]{TV}.

\paragraph{Overview of the construction.} Before we give details, let us briefly outline the individual steps of the construction. For a given cobordism, the first step is to choose skeleta for the source and target surfaces and a compatible graph skeleton for the bulk, together with a coloring of its faces with simple objects of the fusion category. Steps~2 and~3 comprise assigning vector spaces to every half-rim, and to every node by tensoring the spaces of the adjacent half-rims. Two half-rims coming from a rim give dual vector spaces, and in step~4 it is described how this gives rise to an associated vector in their tensor product. In step~5, by drawing a small 3-ball around each node and intersecting it with the skeleton, the graphical calculus of the fusion category is used to induce an evaluation map, which is evaluated in step 6 on the vectors constructed previously. Renormalizing and summing over colorings, one obtains a state sum invariant, which still depends on the surface graphs. In steps~7 and~8, this is transformed into the linear map that one wishes to associate to the cobordism, while simultaneously eliminating the dependence on the surface graphs and colorings.

We now present these steps in more detail.
\begin{enumerate}\itemsep=0pt
 \item \textbf{Choice of skeleton and coloring.}
 Let $M = M_{\Gamma} \in \Cob_3^{\calc}(\Sigma, \Sigma')$ be a nonempty cobordism with a $\calc$-colored graph $\Gamma$ in its free boundary $\partial_fM$. Pick skeleta $G$ and $G'$ of $\Sigma$ respectively $\Sigma'$ and a graph skeleton $P$ of the pair $(M, G^{\mathrm{op}} \sqcup G')$. Moreover, fix a \textit{coloring} of $P$, namely a map
 \begin{align*}
 c\colon \ \mathrm{Fac}(P) \rightarrow \calo(\calc)
 \end{align*}
 labeling each face of $P$ with a simple object in $\calc$.
 \item \label{itemTVCon2} \textbf{Vector spaces for half-rims.} Let $r \in R(P)$ be an internal rim that is an edge of $P^{(1)}$ and denote by $e, e' \in R^h$ the half-rims associated to $r$. We define two vector spaces $H_c(e)$ and $H_c(e')$ that are dual to each other.

 The orientation of $M$ and the choice of orienting each half-rim away from its attached node defines two opposite cyclic orders of the branches of $r$, one for $e$ and one for~$e'$. Let $x_1, \dots, x_k \in \calo({\calc})$ denote the objects labeling the branches of $r$ determined by the coloring~$c$, listed in accordance with the cyclic order with respect to $e$. The orientation of each branch defines a sign $\varepsilon(i)$, which is positive if the orientation of the branch induces the one of~$e$, and negative if it does not. Let $x_i^+ = x_i$, and $x_i^- = x_i^{*}$ and set
 \begin{align*}
 H_c(e) = \Hom_{\calc}\bigl(\mathbbm{1}, x_1^{\varepsilon(1)} \otimes \dots \otimes x_k^{\varepsilon(k)}\bigr).
 \end{align*}
 Note that due to \eqref{eqCycl}, the vector space $H_c(e)$ does not depend on the choice of a first object, only the cyclic order of the $x_i$. With this definition, it is clear that $H_c(e') = \Hom_{\calc}\bigl(\mathbbm{1}, x_k^{-\varepsilon(k)} \otimes \dots \otimes x_1^{-\varepsilon(1)}\bigr)$ and that the two vector spaces are in fact dual to one another (cf.~\eqref{eqDual}).

 If $r$ is a rim coming from a defect edge, the procedure that assigns a vector space is very similar. Note however that here, we have an additional object $Y \in \calc$ labeling $r$. The rim $r$ has only two neighboring branches, a left and a right one, and both orientations are compatible. For concreteness, we assume that the orientation of the left one induces the orientation of the half-rim $e$ (which is possibly different from the orientation of the defect edge in $\Gamma$). Then, we set $H_c(e) = \Hom_{\calc}\left(\mathbbm{1}, x_{\ell} \otimes Y^{\varepsilon} \otimes x_r^*\right)$, where $x_{\ell}$ decorates the left branch, $x_r$ the right one and the sign $\varepsilon$ is positive precisely if the orientation of the edge in $E(\Gamma)$ defining $r$ agrees with the one of the half-edge $e$.
 \item \label{itemSpaces} \textbf{Vector spaces for nodes and skeleton.} To each node $v \in N(P)$, we assign the unordered tensor product \smash{$H_c(v) = \bigotimes_{e \in R^h(v)}H_c(e)$} over the internal half-rims adjacent to $v$. Tensoring further over all $v \in N(P)$, we assign to $M$ the space $H_c(M) = \smash{\bigotimes_{v \in N(P)}H_c(v)} = \smash{\bigotimes_{e \in R^h}H_c(e)}$. Here, we have organized the space in the first expression in terms of nodes, and in the second in terms of internal half-rims. The space $H_c(M)$ can moreover be decomposed in two other ways: We have
 \begin{align}
 \label{eqDec}
 H_c(M) = H_c^0(M) \otimes_{\K} H_c^{\partial}(M) = \bigotimes_{r \in R(P)} H_c(e_r) \otimes_{\K} H_c(e'_r),
 \end{align}
 where in the first tensor decomposition, the tensorand \smash{$H_c^{\partial}(M) = \bigotimes_{e \in R_{\partial}^h}H_c(e)$} is a tensor product running over all half-rims adjacent to the gluing boundary, and \smash{$H_c^0(M)$} is the tensor product over the remaining half-rims in $R_0^h$. In the second decomposition, $e_r$,~$e'_r$~denote the half-rims of a chosen rim~$r$.

 Since the gluing boundary decomposes as $\partial_gM \cong \overline{\Sigma} \sqcup \Sigma'$, we can rewrite the former space as $H_c^{\partial}(M) \cong H_c(\overline{\Sigma}) \otimes_{\K} H_c(\Sigma')$, where we collect in the first factor all half-rims adjacent to $\Sigma$, and in the second all adjacent to $\Sigma'$.
 \item \textbf{Vector from duality pairings of rims.} Let $r$ be an internal rim and $e, e'$ its two half-rims. The duality $H_c(e') \cong H_c(e)^*$, as for any pair of dual vector spaces, allows to identify a basis-independent vector $*_r^c = \mathrm{coev}_{H_c(e)}(1_{\K}) \in H_c(e) \otimes H_c(e')$. By the latter decomposition in \eqref{eqDec}, we therefore find a distinguished vector
 \begin{align*}
 *_c = \bigotimes_{r \in R(P)} *_r^c \in H_c(M).
 \end{align*}
 For later purposes, we introduce here what we will call a \textit{rim basis} $\alpha = \{\alpha_{j_r}^r\}_{j_r=1}^{k_r}$. Pick a basis \smash{$\{\alpha_{j_{e'}}^{e'}\}_{j_{e'}=1}^{k_r}$} of $H_c(e')$ and the corresponding dual basis $\{\alpha_{j_{e}}^{e}\}_{j_{e}=1}^{k_r}$ of $H_c(e)$, that is $\alpha_{j_e}^e = (\alpha_{j_{e'}}^{e'})^* $. We then set $\alpha_{j_r}^r := \alpha_{j_{e'}}^{e'}$ and express the vector $*_r^c$ as
 \begin{align*}
 *_r^c = \sum_{j_e = 1}^{k_r} \alpha_{j_e}^e \otimes (\alpha_{j_e}^e)^* = \sum_{j_r = 1}^{k_r} (\alpha_{j_r}^r)^*\otimes \alpha_{j_r}^r.
 \end{align*}
 Occasionally, we will simply write $*_r = \alpha \otimes \alpha^*$ with implied indices and sum. Note that with the conventions from step~\ref{itemTVCon2}, the collection \smash{$\{\alpha_{j_r}^r\}$} is a basis of the vector space \smash{$ \Hom_{\calc}\bigl(X_1^{\varepsilon(1)} \otimes \dots \otimes X_k^{\varepsilon(k)}, \mathbbm{1}\bigr)$}, where the objects $X_i$ might come from simple objects labeling regions defined by the coloring $c$ or from the decoration of defect lines. Note that we have made a choice of a preferred half-rim $e$ inducing the order of the objects and the signs.
 \item \textbf{Evaluation maps from link graphs.} We construct a family of maps $\Gamma_c(v)\colon H_c(v) \rightarrow \K$, one for each internal node $v \in N_0$, which are ultimately used to evaluate the vector $*_c \in H_c(M)$ defined in the previous step to a scalar. The key idea is the observation that the skeletal geometry and the coloring induce graphs with $\calc$-colored edges on small spheres surrounding the nodes, which, due to sphericity of $\calc$, define scalars in $\K$ when evaluated on a vector in $H_c(v)$.

 Explicitly, this works as follows. If $v$ is a vertex of the skeleton, embed a small 3-ball~$B^3$ around~$v$.\footnote{As explained in \cite{F}, for free boundary vertices one embeds a neighborhood in $\mathbb{R}^3$ and performs the construction there.} The intersection $P \cap B^3$ defines a graph on the boundary sphere $\partial B^3$, called the \textit{link graph}, whose vertices are the intersection points with the edges of $P$ and whose edges are the intersection lines with the faces (or equivalently regions) of $P$ with induced orientation and object labelings. With appropriate sign conventions, each vertex of this graph -- corresponding to a half-edge $e$ -- can naturally be labeled by morphisms in \smash{$H_c(e) = \Hom_{\calc}\bigl(\mathbbm{1}, x_1^{\varepsilon(1)} \otimes \dots \otimes x_k^{\varepsilon(k)}\bigr)$}. A choice of vector in $H_c(e)$ for each half-edge $e \in R^h(v)$ thus determines a labeling of the link graph's vertices. Since $\calc$ is spherical, we can evaluate this colored graph on the sphere to a morphism in $\End_{\calc}(\mathbbm{1}) \cong \K$. Extending linearly, this defines the desired map $\Gamma_c(v)$.

 For $v \in N_0$ a switch (respectively a coupon), we proceed similarly, with the exception that we add the defect edge (respectively the coupon with its attached defect edges) to the link graph. These defect edges are attached at their intersection points with the sphere bounding the 3-ball $B^3$. For more details, we refer to \cite[Section~15.5.1]{TV}.

 \item \textbf{Vector from coloring and state sum invariant.} Evaluating the tensor product of the maps $\Gamma_c(v)$ for internal nodes $v \in N_0$ with the identity morphism on nodes in $N_{\partial}$ on the vector $*_c$ defines a vector
 \begin{align}
 \label{eqCol}
 |c| = V_c(*_c) := \biggl(\bigotimes_{v \in N_0}\Gamma_c(v) \otimes \id_{H_c^{\partial}(M)} \biggr)(*_c) \in H_c^{\partial}(M).
 \end{align}
 If the gluing boundary $\partial_gM$ is empty, the vector $|c|$ is a scalar in the field $\K$. For a given coloring $c$, we sometimes use the notation $|c| = |c(a_1), \dots, c(a_{\ell})|$, where the $a_i$ label the faces $\mathrm{Fac}(P)$.

 Recall the notation $\Fac_{\partial}(P)$ for faces adjacent to the gluing boundary and $\Fac_0(P)$ for fully internal faces. Denote by $c_{\partial} = c|_{\Fac_{\partial}(P)}$ the restriction of the coloring $c$ to the boundary faces. Clearly, $H_c^{\partial}(M)$ only depends on $c_{\partial}$, so we can denote it by $H_{c_{\partial}}^{\partial}(M) = H_c^{\partial}(M)$. Regarding $c_{\partial}$ as fixed, every coloring of fully internal faces $c_0\colon \Fac_0(P) \rightarrow \calo(\calc)$ defines a~coloring $c =c_0 \sqcup c_{\partial}$ of all faces. We define the following vector in $H_{c_{\partial}}^{\partial}(M)$, where the sum ranges over all internal colorings of $P$, with the help of \eqref{eqCol}:
 \begin{align}
 \label{eqVecSS}
 |M, c_{\partial}|_{\calc} = \dim(\calc)^{-|M \setminus P|} \sum_{c_0\colon \Fac_0(P) \rightarrow \calo(\calc)} \dim(c_0 \sqcup c_{\partial}) |c_0 \sqcup c_{\partial}|.
 \end{align}
 Here, $|M \setminus P|$ denotes the number of connected components of the space $M \setminus P$ and
 \begin{align*}
 \dim(c) = \prod_{a \in \Fac(P)}\dim(c(a))^{\chi(a)}
 \end{align*}
 is the dimension of the coloring, where $\chi(a)$ denotes the Euler characteristic of $a$.

 We pause for a moment and remark that the vector $|M, c_{\partial}|_{\calc}$ still depends on the boundary coloring $c_{\partial}$ and the skeleton $G^{\mathrm{op}} \sqcup G'$ of the gluing boundary, as well as on the graph skeleton $P$. The latter dependence is taken care of by assuming Hypothesis~\ref{hyp}, and the former two will be dealt with in the next steps. In the special case where the gluing boundary is empty, we obtain a scalar invariant for cobordisms in $\Cob_3^{\calc}(\varnothing, \varnothing)$, that is,
 \begin{align}
 \label{eqInv}
 |M_{\Gamma}|_{\calc} = \dim(\calc)^{-|M \setminus P|} \sum_{c\colon \Fac(P) \rightarrow \calo(\calc)} \dim(c) |c| \in \K,
 \end{align}
 which fully determines the functor $| \cdot |_{\calc}$ on endomorphism spaces of the empty manifold in~$\Cob_3^{\calc}$. The result is a boundary-defect invariant in the sense of Definition~\ref{defInv}, the \textit{$($defect$)$ Turaev--Viro invariant}, and it is easy to see that it is in fact multiplicative. In order to extend this to a full boundary-defect TQFT, a few additional steps are required.
 \item \textbf{Vector spaces and maps from boundary skeleta.} We enrich the boundary skeleta~$G^{\mathrm{op}}$ and $G'$ of the boundary components $\overline{\Sigma}$ and $\Sigma'$, respectively, by inserting vertices at the location of the colored points. By properties of graph skeleta, the respective restrictions $c_{\Sigma}$ and $c_{\Sigma'}$ of the boundary coloring $c_{\partial}$ are in one-to-one correspondence with colorings of the enriched graphs.
 In this way, the spaces $H_{c_{\Sigma}}(\overline{\Sigma})$ and $H_{c_{\Sigma'}}(\Sigma')$ are defined independently from the cobordism $M$ and the graph skeleton $P$, and depend only on $G^{\mathrm{op}}$ and $G'$.

 Moreover, one can observe that there is a duality $H_{c_{\Sigma}}(\overline{\Sigma}) \cong H_{c_{\Sigma}}(\Sigma)^*$, from which it follows that $H_{c_{\partial}}^{\partial}(M) \cong H_{c_{\Sigma}}(\Sigma)^* \otimes_{\K} H_{c_{\Sigma'}}(\Sigma')$. We may therefore interpret the vector from \eqref{eqVecSS} as a linear map
 \begin{align*}
 |M, c_{\partial}|_{\calc} \in \Hom_{\K}\bigl(H_{c_{\Sigma}}(\Sigma), H_{c_{\Sigma'}}(\Sigma')\bigr).
 \end{align*}
 Summing over all possible colorings, we define the vector space $H(\Sigma, G) = \bigoplus_{c_{\Sigma}}H_{c_{\Sigma}}(\Sigma)$, which is now independent of the choice of boundary colorings. The maps defined above induce linear maps $|M, G, G'|_{\calc}\colon H(\Sigma, G) \rightarrow H(\Sigma', G')$, given by the normalized sum\footnote{Here, for a family of linear maps $f_{ij}\colon V_i \rightarrow W_j$, their sum $\bigoplus_{ij}f_{ij}\colon \bigoplus_iV_i \rightarrow \bigoplus_j W_j$ is defined by $\bigoplus_{ij}f_{ij}((v_k)_k) = \bigl(\sum_i f_{ij}(v_i)\bigr)_j$.}
 \begin{align*}
 |M, G, G'|_{\calc} = \bigoplus_{c_{\Sigma}, c_{\Sigma'}} \frac{\dim(\calc)^{|\Sigma' \setminus G'|}}{\dim^{\mathrm{tr}}(c_{\Sigma'})}|M, c_{\Sigma} \sqcup c_{\Sigma'}|_{\calc},
 \end{align*}
 where the trivial dimension $\dim^{\mathrm{tr}}(c_{\Sigma'})$ of the coloring $c_{\Sigma'}$ is defined as
 \begin{align*}
 \dim^{\mathrm{tr}}(c_{\Sigma'}) = \prod_{a \in \Fac_{\partial}(P), \, a \text{ adjacent to } \Sigma'} c(a).
 \end{align*}
 The normalization is chosen to ensure the functoriality of the construction with respect to composition via gluing. Under Hypothesis~\ref{hyp}, the maps $|M, G, G'|_{\calc}$ can be assumed to be independent of the chosen graph skeleton $P$.
 \item \textbf{Independence of boundary skeleta.} The following is readily verified and ensures both skeleton independence and compatibility of the functor with identities: The linear map $|\Sigma \times I, G, G|_{\calc}$ is an idempotent for any graph skeleton $G$, suggesting the definition of the subspace $|\Sigma, G|_{\calc} = \mathrm{im}(|\Sigma \times I, G, G|_{\calc}) \subset H(\Sigma, G)$ associated to $\Sigma$ and $G$. In addition, for any other skeleton $G'$ of $\Sigma$, the map $|\Sigma \times I, G, G'|_{\calc}$ restricts to an isomorphism $|\Sigma, G|_{\calc} \rightarrow |\Sigma, G'|_{\calc}$. Taking the limit
 \begin{align*}
 |\Sigma|_{\calc} = \lim_G |\Sigma, G|_{\calc}
 \end{align*}
 over the subcategory of objects $|\Sigma, G|_{\calc}$ and isomorphisms $|\Sigma \times I, G, G'|_{\calc}$ for any skeleta $G$,~$G'$ of $\Sigma$ in $\vect_{\K}$ ensures independence of surface skeleta. The universal property of the limit then induces a map $|M|_{\calc}\colon |\Sigma|_{\calc} \rightarrow |\Sigma'|_{\calc}$ from the previously defined maps $|M, G, G'|_{\calc}$. Finally, setting $|\varnothing|_{\calc} = \K$, one has the following theorem, proven in~\cite{TV} and generalized in~\cite{F} to the setting with free boundaries, subject to Hypothesis~\ref{hyp}.
\end{enumerate}
\begin{Theorem}[{\cite[Theorem 15.9]{TV}, \cite[Section 4.2)]{F}}]
 The assignment $| \cdot |_{\calc}\colon \Cob_3^{\calc} \rightarrow \vect_{\K}$ defines a symmetric monoidal functor from the category of $\calc$-colored cobordisms to vector spaces. In other words, it constitutes a~topological quantum field theory with boundary defects.
\end{Theorem}
Restricting the defect Turaev--Viro boundary-defect TQFT to the subcategory $\Cob_3$ gives the standard Turaev--Viro theory without defects \cite[Section~13]{TV}, which we denote by the same symbol $|\cdot|_{\calc}$.

\subsection{Dijkgraaf--Witten theory with general defects}
\label{subSecDW}
In this subsection, we present a concise summary of the construction of untwisted defect Dijk\-graaf--Witten theory as defined in \cite{FMM}, primarily to establish notation and to make the article self-contained. Dijkgraaf--Witten theory will feature in Section~\ref{secDW} and we will recall in detail only those parts of the construction that are relevant for our later discussion. Readers not interested in Dijkgraaf--Witten theory may skip ahead to Section~\ref{secTVRes}, where the results reviewed here are not needed.

Let us first review the definition of Dijkgraaf--Witten theory without defects. For a detailed introduction to this subject, we refer the reader to \cite{FSW}. Let $G$ be a finite group and fix the ground field $\K = \C$, the complex numbers. The untwisted Dijkgraaf--Witten invariant assigns to an arbitrary closed manifold $M$ of dimension $n$ the groupoid cardinality of the groupoid of principal $G$-bundles on $M$. Recall that for any essentially finite groupoid $\calg$, its groupoid cardinality is defined as the rational number
\begin{align}
 \label{eqGrCard}
 |\calg| = \sum_{[X] \in \pi_0(\calg)}\frac{1}{|{\Aut_{\calg}(X)}|} \in \mathbb{Q}.
\end{align}
The Dijkgraaf--Witten invariant is then defined as $Z_G(M) = |{\Bun_G(M)}|$, where $\Bun_G(M)$ denotes the groupoid of principal $G$-bundles. If $M$ is connected, the category $\Bun_G(M)$ is equivalent to the action groupoid $\Hom(\pi_1(M), G)\sslash^{\mathrm{ad}} G$, where objects are group homomorphisms from the fundamental group of $M$ to $G$ and morphims act by conjugation. This leads to the alternative expression
\begin{align*}
 Z_G(M) = \frac{|{\Hom(\pi_1(M), G)|}}{|G|}.
\end{align*}
The invariant can, in fact, be extended to linear maps for arbitrary cobordisms $M\colon \Sigma \rightarrow \Sigma'$ and the result is an $n$-dimensional TQFT $Z_G^n\colon \Cob_n \rightarrow \vect_{\C}$. The process of obtaining the invariant $Z_G(M)$ can be seen as a two-step procedure: In a first step, encode the geometry in the category of principal bundles, and in a second step, linearize the output by taking the groupoid cardinality of this category.

In \cite{FMM}, this construction is generalized to a TQFT with defects $Z_{\mathrm{DW}}$ in three dimensions. Following the aforementioned two-step principle, it arises as the composition
\begin{align}
 \label{eqDefCob}
 \Cob_3^{\mathsf{def}} \, \xrightarrow{C} \, \mathsf{Span}^{\mathsf{fib}}(\mathsf{repGrpd)} \, \xrightarrow{L} \, \vect_{\C}
\end{align}
of two symmetric monoidal functors. The two categories $\Cob_3^{\mathsf{def}}$ and $\mathsf{Span}^{\mathsf{fib}}(\mathsf{repGrpd)}$ will be explained in detail in the subsequent sections. For now, let us point out that $\Cob_3^{\mathsf{def}}$, as the name suggests, is a category of defect cobordisms built from stratified manifolds. This category allows, in contrast to the defect cobordism category $\Cob_3^{\calc}$ from Definition~\ref{defCobDef}, also for defects contained in the interior of 3-manifolds, and of any codimension.

On the flipside, $\Cob_3^{\mathsf{def}}$ is more restrictive than $\Cob_3^{\calc}$ in the sense that free boundaries are not allowed from the outset. However, we will see in Section~\ref{secDW} that they are practically already accounted for, and we will moreover specialize to the framework from Definitions~\ref{defCobDef} and~\ref{defTQFTDef}, which essentially amounts to restriction to a subcategory, once free boundaries are properly included.

The second category appearing is a category of (fibrant) spans of groupoid representations. The precise definition of this category and the linearization functor $L$ will be recalled in the subsequent section.

\subsubsection{The category of fibrant spans of groupoid representations}
We follow \cite[Section~2]{FMM} with occasional changes in notation.
The category $\mathsf{Grpd}$ of groupoids is enriched over itself, so the functor category $\Fun(\calg, \calh) = \mathsf{Grpd}(\calg, \calh)$ naturally carries the structure of a groupoid.

A functor $p\colon \mathcal{E} \rightarrow \mathcal{B}$ between groupoids is called a fibration if for every morphism $f \in \Hom_{\mathcal{B}}(B, p(E))$ in $\mathcal{B}$ there exists a morphism $h \in \Hom_{\mathcal{E}}(E', E)$ such that $p(h) = f$. Let us give an example of a fibration: the interval groupoid $\mathcal{I}$ has two objects $0$ and $1$ and one isomorphism ${0 \rightarrow 1}$. It gives rise to the groupoid $\mathcal{G}^{\mathcal{I}} = \Fun(\mathcal{I}, \calg)$, whose objects can be identified with morphisms $f\colon G \rightarrow G'$ in $\calg$ and whose morphisms $f \rightarrow f'$ are pairs $(h, h')$ of morphisms in $\calg$ making the obvious naturality square commute. The groupoid $\calg^{\mathcal{I}}$ comes with two natural projection functors $p_0, p_1\colon \mathcal{G}^{\mathcal{I}} \rightarrow \mathcal{G}$ defined by $p_0(f) = G, p_1(f) = G'$ and $p_0(h, h') = h, p_1(h, h') = h'$. The functor $(p_0, p_1)\colon \calg^{\cali} \rightarrow \calg \times \calg$ sending $f$ to the pair $(G, G')$ is a fibration.

For objects $X$, $X_1$, $X_2$ of a category $\calx$, a diagram \smash{$X_1 \xleftarrow{f_1} X \xrightarrow{f_2} X_2$} is called a span in $\calx$. If $\calx = \mathsf{Grpd}$, by definition, the span is a fibrant span if the functor $(f_1, f_2)\colon X \rightarrow X_1 \times X_2$ is a~fibration of groupoids. In particular, the span \smash{$\calg \xleftarrow{p_0} \calg^{\cali} \xrightarrow{p_1} \calg$} is fibrant.

A functor $\rho\colon \calg \rightarrow \vect_{\C}$ is called a groupoid representation. For example, a group representation of a group $G$ amounts to a groupoid representation of the single-object groupoid~$\bullet \sslash G$. Many known facts can be extended from finite groups to essentially finite groupoids. In fact, the category \smash{$\vect_{\C}^{\calg} = \Fun(\calg, \vect_{\C})$} of groupoid representations is a rigid symmetric monoidal category. The monoidal structure is induced by the tensor product in $\vect_{\C}$, namely $(\rho \otimes \rho')(G) := \rho(G) \otimes_{\C} \rho'(G)$ for $\rho, \rho' \in \vect_{\C}^{\calg}$ and $G, G' \in \calg$, together with the monoidal unit $\C\colon \calg \rightarrow \vect_{\C}$ mapping every object to the tensor unit $\C$ in $\vect_{\C}$ and each morphism to the identity $\id_{\C}$. Dual objects are given by $\rho^*(G) = \rho(G)^*$ and $\rho^*(f) = \rho\bigl(f^{-1}\bigr)^*$ for $f \in \Hom_{\calg}(G, G')$ a morphism and~$\vect_{\C}^{\calg}$. This endows $\vect_{\C}^{\calg}$ with a spherical structure coming from $\vect_{\C}$.
\begin{Remark}
 The discrete groupoid $\underline{G}$, with objects elements of a group $G$ and no nontrivial morphisms, gives rise to the category \smash{$\vect_{\C}^{\underline{G}}$}. As a category, it is equivalent to the category $\vect_G$ of $G$-graded vector spaces. However, as a monoidal category, \smash{$\vect_{\C}^{\underline{G}}$} is different~from $\vect_G$. If the latter category is represented
 as a functor category, the natural monoidal product can be identified with Day convolution, rendering it a spherical fusion category. The former category \smash{$\vect_{\C}^{\underline{G}}$}, equipped with the tensor product outlined above, is not spherical fusion in the sense of Section~\ref{subsecCatCon}. For instance, the monoidal unit is not simple.
\end{Remark}

Henceforth, unless stated otherwise, we will tacitly assume all groupoids to be essentially finite. For the following definition, recall that the category~$\mathsf{Grpd}$ has pullbacks.

\begin{Definition}[{\cite[Lemmas 2.7 and~2.8]{FMM}}]
 \label{defFibSpan}
 The category of fibrant spans of groupoid representations $\mathsf{Span}^{\mathsf{fib}}(\mathsf{repGrpd)}$ consists of
 \begin{itemize}\itemsep=0pt
 \item objects given by pairs $(\calg, \rho)$ of (essentially finite) groupoids $\calg$ and representations $\rho\colon \calg \rightarrow \vect_{\C}$,
 \item morphisms from $(\calg_1, \rho_1)$ to $(\calg_2, \rho_2)$ given by equivalence classes of pairs $\bigl(\calg_1 \xleftarrow{f_1} \calg \xrightarrow{f_2} \calg_2, \sigma\bigr)$ of fibrant spans of groupoids and natural transformations $\sigma\colon \rho_1 \circ f_1 \Rightarrow \rho_2 \circ f_2$.
 \end{itemize}
 Two pairs \smash{$\bigl(\calg_1 \xleftarrow{f_1} \calg \xrightarrow{f_2} \calg_2, \sigma\bigr)$} and \smash{$\bigl(\calg_1 \xleftarrow{f_1'} \calg' \xrightarrow{f_2'} \calg_2, \sigma'\bigr)$} representing morphisms in the category $\mathsf{Span}^{\mathsf{fib}}(\mathsf{repGrpd)}$ are equivalent if there are functors $F\colon \calg \substack{\longrightarrow \\ \longleftarrow} \calg' : \!F' $ satisfying $f_i' \circ F = f_i$ and $f_i \circ F' = f_i'$ and $\sigma' F = \sigma$ and $\sigma F' = \sigma'$ such that there are natural isomorphisms $\eta\colon F' \circ F \Rightarrow \id_{\calg}$ and $\eta'\colon F \circ F' \Rightarrow \id_{\calg'}$ satisfying $f_i \eta = \id_{f_i}$ and $f_i' \eta' = \id_{f_i'}$. One defines composition by pullback,
 \begin{align*}
 \bigl[\calg_2 \xleftarrow{h_1} \calh \xrightarrow{h_2} \calg_3, \tau\bigr] \circ \bigl[\calg_1 \xleftarrow{f_1} \calg \xrightarrow{f_2} \calg_2, \sigma\bigr] = \bigl[\calg_1 \xleftarrow{f_1 \circ T_1} \calg \times_{\calg_2} \calh \xrightarrow{h_2 \circ T_2} \calg_3, \tau T_2 \circ \sigma T_1\bigr]
 \end{align*}
 with the two projection functors $T_1$, $T_2$ from the pullback to the respective factors. The identity with respect to this composition law and an object $(\calg, \rho)$ is given by the class $\bigl[\calg \xleftarrow{p_0} \calg^{\cali} \xrightarrow{p_1} \calg, \sigma^{\rho}\bigr]$, where $\sigma^{\rho}_{f} = \rho(f)$ for any morphism $f\colon G \rightarrow G'$.
\end{Definition}
The category \smash{$\mathsf{Span}^{\mathsf{fib}}(\mathsf{repGrpd)}$} is a symmetric monoidal category, where the monoidal structure is defined by $(\calg, \rho) \otimes (\calh, \rho') = (\calg \times \calh, \rho \otimes \rho')$ with $(\rho \otimes \rho')(G, H) = \rho(G) \otimes_{\C} \rho'(H)$, and the monoidal unit is $(\bullet, \C)$, the pair of the singleton groupoid and the trivial representation. Restricting to trivial groupoid representations and identity natural transformations defines a subcategory $\mathsf{Span}^{\mathsf{fib}}(\mathsf{Grpd)} \subset \mathsf{Span}^{\mathsf{fib}}(\mathsf{repGrpd)}$ of groupoids and fibrant spans between groupoids.

\subsubsection{Linearization functor to vector spaces}
\label{subSecLin}
Next, we recall the definition of the linearization functor $L\colon \mathsf{Span}^{\mathsf{fib}}(\mathsf{repGrpd)} \rightarrow \vect_{\C}$, following \cite[Section~2.3]{FMM}.

For a functor $F\colon \calg \rightarrow \calh$ between groupoids and an object $H \in \calh$, the fiber $\mathrm{Fib}(F,H)$ is the groupoid consisting of objects $G \in \calg$ such that $F(G) = H$ and morphisms $f\colon G \rightarrow G'$ such that $F(f) = \id_{H}$. For a fibrant span \smash{$\calg_1 \xleftarrow{f_1} \calg \xrightarrow{f_2} \calg_2$}, denote the fiber of \smash{$(f_1, f_2)\colon \calg \rightarrow \calg_1 \times \calg_2$} and $(G_1, G_2) \in \calg_1 \times \calg_2$ by
\begin{align*}
 \{G_1|\calg|G_2\} = \mathrm{Fib}\left((f_1, f_2), (G_1, G_2) \right).
\end{align*}
Now, let $\rho\colon \calg \rightarrow \vect_{\C}$ be a groupoid representation. The limit of $\rho$ can be written as
\begin{align*}
 \lim \rho = \bigoplus_{[G] \in \pi_0(\calg)} \rho(G)^{\Aut_{\calg}(G)},
\end{align*}
the sum of invariants with respect to the natural $\Aut_{\calg}(G)$-action on the components $\rho(G)$. We denote by $\iota_G\colon \rho(G)^{\Aut_{\calg}(G)} \rightarrow \lim \rho$ and $\pi_G\colon \lim \rho \rightarrow \rho(G)^{\Aut_{\calg}(G)}$ the inclusion and projection morphisms for the direct sum.

Consider a pair \smash{$\bigl(\calg_1 \xleftarrow{f_1} \calg \xrightarrow{f_2} \calg_2, \sigma\bigr)$} of a fibrant span together with a natural transformation representing a morphism between two objects $(\calg_1, \rho_1)$ and $(\calg_2, \rho_2)$ in $\mathsf{Span}^{\mathsf{fib}}(\mathsf{repGrpd)}$. For every tuple $(G_1, G_2) \in \calg_1 \times \calg_2$, there is a linear map \cite[equation~(7)]{FMM}
\begin{align*}
 S_{\sigma}^{G_1, G_2}:= \frac{1}{|{\Aut_{\calg_2}(G_2)}|} \sum_{[H] \in \pi_{0}\left(\{G_1|\calg|G_2\} \right)} \frac{1}{|{\Aut_{\{G_1|\calg|G_2\}}(H)}|} \sigma_H\colon \ \rho_1(G_1) \rightarrow \rho_2(G_2),
\end{align*}
which is well defined in the sense that it does not depend on representatives of the isomorphism classes $[H] \in \pi_0\{G_1|\calg|G_2\} $, and whose image is contained in the subspace of invariants $\rho_2(G_2)^{\Aut_{\calg_2}(G_2)} \subset \rho_2(G_2)$ by \cite[Lemma 2.13]{FMM}. We are ready to give the definition of the functor $L\colon \mathsf{Span}^{\mathsf{fib}}(\mathsf{repGrpd)} \rightarrow \vect_{\C}$.
\begin{Proposition}[{\cite[Proposition 2.14 and Corollary 2.15]{FMM}}]\label{propFunL}
 There exists a symmetric monoidal functor \smash{$L\colon \mathsf{Span}^{\mathsf{fib}}(\mathsf{repGrpd)} \rightarrow \vect_{\C}$}, assigning to an object $(\calg, \rho)$ the limit $L(\calg, \rho) = \lim \rho$ and to a morphism \smash{$\bigl[s = \bigl(\calg_1 \xleftarrow{f_1} \calg \xrightarrow{f_2} \calg_2\bigr), \sigma\bigr]$} between $(\calg_1, \rho_1)$ and $(\calg_2, \rho_2)$ the linear map
 \begin{align*}
 L(s, \sigma) = \sum_{[G_1] \in \pi_0\calg_1, \, [G_2] \in \pi_0\calg_2} \iota_{G_2} \circ S_{\sigma}^{G_1, G_2} \circ \pi_{G_1}\colon \ \lim \rho_1 \rightarrow \lim \rho_2.
 \end{align*}
 On the subcategory $\mathsf{Span}^{\mathsf{fib}}(\mathsf{Grpd)}$, the functor $L$ takes the explicit form $L(\calg, \C) = \langle \pi_0(\calg) \rangle_{\C}$ on objects, and the matrix elements $\langle G_1|L(s, \id)|G_2\rangle\colon \C \! G_1 \rightarrow \C \! G_2$ of the map
 \begin{align*}
 L(s, \id) = \bigoplus_{[G_1] \in \pi_0\calg_1, \, [G_2] \in \pi_0\calg_2} \langle G_1|L(s, \id)|G_2\rangle\colon \ \langle \pi_0(\calg_1) \rangle_{\C} \rightarrow \langle \pi_0(\calg_2) \rangle_{\C}
 \end{align*}
 are defined, using groupoid cardinality introduced in \eqref{eqGrCard}, as
\begin{align}\label{eqSpGrpd}
\langle G_1|L(s, \id)|G_2\rangle := \big| \mathrm{PC}_{G_2}(\calg_2) \big| \cdot \big| \{G_1|\calg|G_2\} \big| \in \C,
\end{align}
 where $\mathrm{PC}_{G_2}(\calg_2) \subset \calg_2$ denotes the full subgroupoid of objects in the equivalence class $[G_2] \in \pi_0(\calg_2)$.
\end{Proposition}

\subsubsection{Categories of defect cobordisms for Dijkgraaf--Witten theory}
\label{subSubSecDWCob}
\paragraph{Stratifications and local strata.}
Here, we recall the definition of the defect cobordism category $\Cob_3^{\mathsf{def}}$ from \eqref{eqDefCob} as defined in~\cite{FMM}. Again, we will mostly, but not always, stay with the terminology and notation from the original source. Although in~\cite{FMM} the authors work with piecewise linear (or PL) manifolds, we will ignore this subtle difference and continue working with topological manifolds. This is permissible since every topological manifold of dimension less than or equal to 3 admits a PL structure unique up to PL homeomorphism by Moise's theorem.

In the following, we assume $n \leq 3$ whenever we use the term $n$-manifold. Following \cite[Section~3.2]{FMM}, a \textit{stratification} of an $n$-manifold $X$ is a sequence
\begin{align*}
 \varnothing = X^{-1} \subset X^0 \subset X^1 \subset \dots \subset X^n = X
\end{align*}
of closed subspaces such that $X^k \setminus X^{k-1}$ is a (not necessarily compact) $k$-manifold and $ \partial \bigl(X^k \setminus X^{k-1}\bigr) \subset \partial X$ \cite[Definition~3.6]{FMM}. Here, a subspace $X^k$ is referred to as the $k$-skeleton of $X$ and the connected components of $X^k \setminus X^{k-1}$ as $k$-strata. The codimension of a stratum $s \in S_k^X := \pi_0\bigl(X^k \setminus X^{k-1}\bigr)$ is the natural number $\mathrm{codim}(s) = n-k$ and the strata of codimension~1,~2 and~3 are the faces, edges and vertices, respectively. A continuous map $f\colon X \rightarrow Y$ between stratified manifolds satisfying $f\bigl(X^k\bigr) \subset Y^k$ is called filtered.

Unless otherwise stated, all stratifications considered in this text will be \textit{homogeneous} and \textit{saturated}, which are two conditions ensuring that local neighborhoods of points on strata have a somewhat regular appearance; among other things, they guarantee that strata are transversal at the boundary and that $k$-strata are contained in the closure of at least one $\ell$-stratum for any dimension $\ell \geq k$ (we refer to \cite[Definition~3.7]{FMM} for details).

A stratification is called \textit{oriented} if strata of codimension 1 and 2 are endowed with an orientation. For an $(n-1)$-stratum $s \in S_{n-1}^X$ in an $n$-manifold $X$ with an oriented stratification, $L(s)$ denotes the neighboring $n$-stratum the normal vector of $s$ is directed away from, and $R(s)$ denotes the one it points toward.
Finally, a stratification is \textit{framed} if one can make a locally coherent choice of a neighboring $n$-stratum for any $(n-2)$-stratum (cf.\ \cite[Definition~3.10]{FMM} for details). We will call a compact $n$-manifold together with a homogeneous, saturated, oriented and framed stratification a \textit{stratified} $n$-\textit{manifold}.

The four properties discussed in the last paragraph are inherited by the boundary, which carries for a homogeneous and saturated stratified manifold $X$ a natural stratification specified by $\partial X^k = X^{k+1} \cap \partial X$ with the same properties. Edges intersecting the boundary are oriented away from boundary vertices with negative sign and toward boundary vertices with positive sign. Every boundary $k$-stratum $t \in S_k^{\partial X}$ is associated with a unique $(k+1)$-stratum $s$ such that $t \subset s$, defining a map $\iota_k\colon S_k^{\partial X} \rightarrow S_{k+1}^X$.

For a given stratum $s \in S_k^X$, there is the notion of a \textit{local} $j$-\textit{stratum} of $s$, where $j \geq k$. The local strata are essentially germs of $j$-strata with respect to local neighborhoods of $s$. The set of local $j$-strata of $s$ is denoted by $S_j^{\mathrm{loc}}(s)$ \cite[Definition 3.13]{FMM}. Since any local stratum $q \in S_j^{\mathrm{loc}}(s)$ is associated with a $j$-stratum $t$, one also writes $q\colon s \rightarrow t$. For example, for an edge $e \in S_1^X$ starting at a vertex $v \in S_0^X$, there is a local stratum $s(e)\colon v \rightarrow e$, and for an edge $e$ that ends at $v'$, there is a local stratum $t(e)\colon v' \rightarrow e$. The notation suggests that strata and local strata are grouped into a category $\mathcal{Q}^X$. It is the category freely generated by the graph with vertices consisting of strata and edges consisting of local strata (cf.\ \cite[Definition~3.22]{FMM}).

\paragraph{Defect labels and cobordisms.} The following is based on \cite[Section~4]{FMM}. To a stratified $n$-manifold $X$ one wishes to assign group data in order to produce a defect Dijkgraaf--Witten theory. In a first step, one assigns \textit{classical} defect data to $n$- and $(n-1)$-strata: To every $n$-stratum $s$, one associates a finite group $G_s$, and to every $(n-1)$-stratum $t$ that is positioned between the two strata $L(t)$ and $R(t)$, one associates a finite \smash{$\bigl(G_{L(t)} \times G_{R(t)}^{\mathrm{op}}\bigr)$}-set $M_t$.

From these data, one obtains action groupoids $\cald_u$ for any stratum $u \in S_k^X$ by collecting the data assigned to local strata at $u$. Define the groupoid
\begin{align*}
 \cald_u := \biggl(\prod_{(q\colon u \rightarrow t) \in S_{n-1}^{\mathrm{loc}}(u)} M_t \biggr) \sslash \biggl( \prod_{(q\colon u \rightarrow s) \in S_n^{\mathrm{loc}}(u)} G_s \biggr)
\end{align*}
with the action induced by the action of the groups $G_{L(t)}$ and $G_{R(t)}$ on the sets $M_t$. Similarly, for a local stratum $q\colon u \rightarrow u'$ one can assign a functor $D_q\colon \cald_u \rightarrow \cald_{u'}$ (see \cite[Lemma 4.3]{FMM} and the paragraphs leading up to it). This assignment can be repackaged into a functor $D^X\colon \mathcal{Q}^X \rightarrow \mathsf{Grpd}$ defined on objects as $D^X(u) = \cald_u$.

For a (full) assignment of defect data, the classical defect data are augmented with \textit{quantum} defect data. Recall that $n \leq 3$.
\begin{Definition}[{\cite[Definition~4.4]{FMM}}]
 \label{defDefData}
 An \textit{assignment of defect data} to a stratified $n$-mani\-fold~$X$ comprises assigning
 \begin{itemize}\itemsep=0pt
 \item to any $n$-stratum $s$ a finite group $G_s$,
 \item to any $(n-1)$-stratum $t$ a finite $\bigl(G_{L(t)} \times G_{R(t)}^{\mathrm{op}}\bigr)$-set $M_t$,
 \item to any $(n-2)$-stratum $u$ a representation $\rho_u\colon \cald_u \rightarrow \vect_{\C}$ and
 \item to any $(n-3)$-stratum $v$ a natural transformation $\sigma_v\colon \rho_v^t \Rightarrow \rho_v^s$, where for example
 \begin{align*}
 \rho_v^s := \bigotimes_{s(u)\colon v \rightarrow u} \rho_u \circ D_{s(u)} \colon \ \cald_v \rightarrow \vect_{\C}
 \end{align*}
 is the tensor product over all local $(n-2)$-strata associated to outgoing edges at~$v$, and~$\rho_v^t$ is defined similarly with respect to all incoming edges at~$v$.
 \end{itemize}
\end{Definition}
A morphism of stratified manifolds with defect data is a filtered map respecting orientations, framings and the assignment of defect data. We can now define the category $\Cob_3^{\mathsf{def}}$ of defect cobordisms in the Dijkgraaf--Witten context.
\begin{Definition}[{\cite[Definition~4.6]{FMM}}] \label{defCobDW}
 The symmetric monoidal category $\Cob_3^{\mathsf{def}}$ is obtained by repeating Definition~\ref{defCob} with the following replacements: Objects are stratified surfaces $\Sigma$ with an assignment of defect data and morphisms $\Sigma \rightarrow \Sigma'$ are equivalence classes represented by cobordisms, which are pairs $(M, h)$ of a stratified manifold $M$ with an assignment of defect data together with an isomorphism $h\colon \Sigma \sqcup \overline{\Sigma} \rightarrow \partial M $ of stratified surfaces with defect data.

 The identity cobordism on $\Sigma$ is represented by the cylinder $\Sigma \times I$ with the stratification $(\Sigma \times I)^{k+1} = \Sigma^k \times I$ and induced framing, orientations and defect data.
\end{Definition}
\begin{Remark}
 We emphasize that cobordisms in $\Cob_3^{\mathsf{def}}$ do not possess a free boundary since we require that the maps $h\colon \Sigma \sqcup \overline{\Sigma} \rightarrow \partial M $ in the above definition be isomorphisms. As already pointed out, we will see how to include free boundaries in the formalism in Section~\ref{secDW}.
\end{Remark}
By abuse of terminology, we will call a symmetric monoidal functor $Z\colon \Cob_3^{\mathsf{def}} \rightarrow \vect_{\C}$ a~defect TQFT, even though the cobordism category differs from that used for Definition~\ref{defTQFTDef}.

\subsubsection{Gauge groupoids and Dijkgraaf--Witten TQFT with defects}
\label{subSecGG}
\paragraph{Gauge groupoids.} We follow \cite[Section~5]{FMM}. The functor $D^X\colon \mathcal{Q}^X \rightarrow \mathsf{Grpd}$ assigns groupoids assembled from the classical defect data to the strata of the manifold $X$, based on the incidence data of the stratification. To handle the topology of the strata, in a first step, one assigns to any $k$-stratum $s$ a related topological space $\hat{s}$. This is done by picking a triangulation of $X$ containing the $k$-skeleta $X^k$ as subcomplexes, and then gluing all $k$-simplices intersecting $s$ along their common faces. The result is a (compact) $k$-manifold $\hat{s}$ with interior $\mathrm{Int}(\hat{s}) = s$. For a local stratum $q\colon s \rightarrow t$, one can further assign a continuous map $\iota_q\colon \hat{s} \rightarrow \hat{t}$ such that the assignment is functorial and does not depend on the chosen triangulation. This defines a functor $\mathcal{Q}^X \rightarrow \mathrm{Top}$ \cite[Proposition 5.1]{FMM}.

In a second step, postcomposing with the fundamental groupoid functor defines the functor $T^X\colon \mathcal{Q}^X \rightarrow \mathsf{Grpd}, T^X(s) = \Pi_1 (\hat{s})$. The gauge groupoid of $X$ relates the functors $D^X$ and $T^X$.

\begin{Definition}[{\cite[Definition 5.7]{FMM}}]
 The \textit{gauge groupoid} $\cala^X \sslash \calg^X$ is defined as the end of the functor $\mathsf{Grpd}\bigl(T^X(-), D^X(-)\bigr)\colon \mathcal{Q}^X \times \bigl(\mathcal{Q}^X\bigr)^{\mathrm{op}} \rightarrow \mathsf{Grpd}$. In formulas, this means
 \begin{align*}
 \cala^X \sslash \calg^X := \int_{s \in \mathcal{Q}^X}\mathsf{Grpd}\left( \Pi_1\left(\hat{s}\right), \cald_s\right) = \int_{s \in \mathcal{Q}^X}\cald_s^{\Pi_1\left(\hat{s}\right)}.
 \end{align*}
\end{Definition}
The objects of this groupoid are called \textit{gauge configurations} and the morphisms \textit{gauge transformations}. By definition of the end there are structure morphisms \smash{$P_s\colon \cala^X \sslash \calg^X \rightarrow \cald_s^{\Pi_1 (\hat{s})}$} for any stratum~$s$ of~$X$. For a manifold $X$ with boundary, there is also a natural projection functor $P_{\partial}\colon \cala^X \sslash \calg^X \rightarrow \cala^{\partial X} \sslash \calg^{\partial X}$ that restricts gauge configurations and transformations to the boundary \cite[Corollary~5.8]{FMM}. For a cobordism $M \in \Cob_3^{\mathsf{def}}(\Sigma, \Sigma')$, this defines a fibrant span \smash{$\cala^{\Sigma} \sslash \calg^{\Sigma} \xleftarrow{P} \cala^M \sslash \calg^M \xrightarrow{P'} \cala^{\Sigma'} \sslash \calg^{\Sigma'}$}, where $P_{\partial} = (P, P')$ \cite[Corollary 5.13]{FMM}.

As our goal is the explicit computation of gauge groupoids, we will mainly use a more concrete characterization of the gauge groupoid, which makes use of basepoints. A coherent basepoint set is a finite set $B \subset X$ that meets the closure of every stratum and boundary stratum in a~way that inclusions defined by local strata and boundary inclusions are respected (for a more precise definition, see \cite[Definition~5.14]{FMM}). For a stratum $s$, denote by $V_s$ the preimage of $B$ under the natural map $\hat{s} \rightarrow \mathrm{cl}(s)$ to the closure of $s$. The full subgroupoid $\Pi_1(\hat{s}, V_s) \subset \Pi_1(\hat{s})$ of the fundamental groupoid has as set of objects $V_s$.

\begin{Proposition}[{\cite[Corollary 5.18 and Proposition 5.19]{FMM}}]
 \label{propGG}
 For a stratified $n$-manifold $X$ and any coherent basepoint set $B \subset X$, there is a groupoid $\cala^{X \bullet} \sslash \calg^{X \bullet}$, called the \textit{reduced gauge groupoid}, which is equivalent to $\cala^X \sslash \calg^X$ and is explicitly characterized by the following data:
 \begin{itemize}\itemsep=0pt
 \item Objects $($\textnormal{reduced gauge configurations}$)$ are collections $(A, h) = (\{A_s\}, \{h_t\})$ of functors $A_s\colon \Pi_1(\hat{s}, V_s) \rightarrow \bullet \sslash G_s$ indexed by $n$-strata $s \in S_n^X$ and maps $h_t\colon V_t \rightarrow M_t$ indexed by $(n-1)$-strata $t \in S_{n-1}^X$ such that for every path $\delta\colon [0,1] \rightarrow \hat{t}$ with $\delta(0) = y$, $\delta(1) = y'$ and $y, y' \in V_t$, the condition
 \begin{align*}
 h_t(y') = A_{L(t)} ([ \iota_{\ell(t)} \circ \delta] ) \triangleright h_t(y) \triangleleft A_{R(t)} ([\iota_{r(t)} \circ \delta] )^{-1}
 \end{align*}
 holds, where $\ell(t)\colon t \rightarrow L(t)$ and $r(t)\colon t \rightarrow R(t)$ are the two local $n$-strata at $t$. Here, $G_s$ are the groups assigned to the $n$-strata and $M_t$ the $G_{L(t)} \times G_{R(t)}^{\mathrm{op}}$-sets assigned to the $(n-1)$-strata with the left and right action denoted by $\triangleleft$ and $\triangleright$.
 \item Morphisms (\textnormal{reduced gauge transformations}) $(A, h) \Rightarrow (A', h')$ are collections of maps $\gamma_s\colon V_s \rightarrow G_s$ parametrized by $n$-strata $s$ such that for any $n$-stratum $s$ and any path $\tau\colon [0,1] \rightarrow \hat{s}$ from $\tau(0) = x$ to $\tau(1) = x'$, where $x, x' \in V_s$, and for any $(n-1)$-stratum $t$ and any $y \in V_t$, one has the two conditions
 \begin{align*}
& A_s'([\tau])= \gamma_s(x') \cdot A_s([\tau]) \cdot \gamma_s(x)^{-1}, \\
&h_t'(y)= (\gamma_{L(t)} \circ \iota_{\ell(t)})(y) \triangleright h_t(y) \triangleleft (\gamma_{R(t)} \circ \iota_{r(t)})(y)^{-1}.
 \end{align*}
 \end{itemize}
\end{Proposition}

\paragraph{Dijkgraaf--Witten theory with defects.} The gauge groupoid processes the codimension 0 and 1 layers of an assignment of defect data. To account for the remaining representation data of the codimension 2 and 3 layers, one needs to construct representations of the gauge groupoids of surfaces, and a natural transformation relating them for every cobordism between the surfaces. All of these data are then assembled into a functor $C\colon \Cob_3^{\mathsf{def}} \rightarrow \mathsf{Span}^{\mathsf{fib}}(\mathsf{repGrpd)}$.

We follow \cite[Section~6]{FMM}. Let $X$ be a stratified manifold with defect data. Recall that there~are projection functors $P_s$ from the gauge groupoid to the groupoid \smash{$\cald_s^{\Pi_1 (\hat{s} )}$} assigned to every stratum~$s$. By \cite[Lemma~6.5]{FMM}, for each edge $e \in S_2^X$ connected to at least one vertex, there~are functors \smash{$P_e^s, P_e^t\colon \cala^X \sslash \calg^X \rightarrow \cald_e $} such that $P_e^s = D_{s(e)} \circ P_v$, whenever~$e$ is oriented away~from a~vertex~$v$, and $P_e^t = D_{t(e)} \circ P_v$ whenever $e$ is oriented toward~$v$. Here, we view~$P_v$ as a~functor~\smash{$\cala^X \!\sslash \calg^X \!\rightarrow\! \cald_v^{\Pi_1(\hat{v})} \cong \cald_v$}.

There are similar compatibilities with the functor $P_{\partial}$ for incoming or outgoing edges at a~boundary vertex. There is also a natural transformation $P_e^d\colon P_e^s \Rightarrow P_e^t$ relating the two functors \cite[Lemma~6.6]{FMM}.

We describe the representations and natural transformations $C$ assigns to stratified surfaces and cobordisms with defect data, respectively.
\begin{enumerate}\itemsep=0pt
 \item Let $\Sigma$ be a stratified surface with defect data associating to vertices representations $\rho_v\colon \cald_v \rightarrow \vect_{\C}$ (cf. Definition~\ref{defDefData}). If a vertex $v \in S_0^{\Sigma}$ has positive orientation, define $F_v^{\Sigma}\colon \cala^{\Sigma} \sslash \calg^{\Sigma} \rightarrow \vect_{\C}$ as $F_v^{\Sigma} := \rho_v \circ P_v$, and if $v$ has negative orientation, set $F_v^{\Sigma} := \rho_v^* \circ P_v$. Then, define the representation of the gauge groupoid as the tensor product of functors
 \begin{align}
 \label{eqGGRep}
 F_{\Sigma}:= \bigotimes_{v \in S_0^{\Sigma}} F_v^{\Sigma}\colon \ \cala^{\Sigma} \sslash \calg^{\Sigma} \rightarrow \vect_{\C}.
 \end{align}
 \item Let $M\colon \Sigma \rightarrow \Sigma'$ be a stratified cobordism with defect data assigning to edges representations $\rho_e\colon \cald_e \rightarrow \vect_{\C}$ and to vertices natural transformations $\sigma_v\colon \rho_v^t \Rightarrow \rho_v^s$. Inserting a~new bivalent vertex in the middle of each edge defines an oriented graph $\Gamma^M$ whose free ends lie in the boundary~$\partial M$. One colors this graph by assigning to the two edges in~$E\bigl(\Gamma^M\bigr)$ corresponding to a defect edge $e \in S_1^M$ the functors
 \begin{align*}
 F_e^s := \rho_e \circ P_e^s, \qquad F_e^t := \rho_e \circ P_e^t\colon \ \cala^M \sslash \calg^M \rightarrow \vect_{\C},
 \end{align*}
 where the first is assigned to the edge in $E\bigl(\Gamma^M\bigr)$ that is outgoing, and the second is assigned to the edge that is incoming at a vertex of~$M$ or~$\partial M$.

 For a vertex $v$ of $M$, set $\mu_v := \sigma_v P_v\colon \rho_v^t \circ P_v \Rightarrow \rho_v^s \circ P_v$ and for a bivalent vertex $v_e \in V\bigl(\Gamma^M\bigr)$ not coming from $S_0^M$ sitting on an edge $e \in S_1^M$, set $\mu_e:= \rho_eP_e^d\colon F_e^s \Rightarrow F_e^t$. There is an assignment to edges that form isolated loops, which we will not recap in detail as we will not work explicitly with edges of this kind (compare footnote~\ref{footNoLoops}).

 Altogether, the result is a graph $\Gamma^M$ labeled in the monoidal category \smash{$\vect_{\C}^{\cala^M \sslash \calg^M}$}. The graph has a well-defined evaluation to a morphism in \smash{$\vect_{\C}^{\cala^M \sslash \calg^M}$} because the category is spherical and symmetric.
 \item Evaluating the \smash{$\vect_{\C}^{\cala^M \sslash \calg^M}$}-labeled graph $\Gamma^M$ as a morphism defines a natural transformation
 \[
 \mu_M := \bigl\langle \Gamma^M\bigr\rangle\colon \bigotimes_{v \in S_0^{\Sigma}} (F_{e_v}^{x_v})^v \Rightarrow \bigotimes_{u \in S_0^{\Sigma'}} (F_{e_u}^{x_u})^u.
\]
Here, the notation is as follows: For a boundary vertex $v$ induced by an outgoing edge, put $x_v = s$, and $x_v = t$ if it is incoming. If $v$ is in the part of $\partial M$ associated to $\Sigma$ and positively oriented or in the part of $\partial M$ associated to $\Sigma'$ and negatively oriented, set $(-)^v = (-)^*$, and otherwise $(-)^v$ acts as the identity on the functor inside the brackets. Recall the functor $P_{\partial}$ consisting of the components $P$ and $P'$ projecting from the bulk to the boundary groupoids. Then, since $\rho_e = \rho_v$ for a vertex in $\partial M$ and associated edge~$e$ in~$M$ and $P_v \circ P_{\partial} = P_e^t$ or $P_v \circ P_{\partial} = P_e^s$ depending on whether~$e$ is in- or outgoing at~$v$, $\mu_M$~indeed defines a natural transformation $F_{\Sigma} \circ P \Rightarrow F_{\Sigma'} \circ P'$.
\end{enumerate}
\begin{Theorem}[{\cite[Proposition 6.7]{FMM}}]
 There is a symmetric monoidal functor $C\colon \Cob_3^{\mathsf{def}} \rightarrow \mathsf{Span}^{\mathsf{fib}}(\mathsf{repGrpd)} $ that assigns to stratified surfaces with defect data $\Sigma$ the pair of gauge groupoid and its representation constructed in~\eqref{eqGGRep}, i.e., $C(\Sigma) = \bigl(\cala^{\Sigma} \sslash \calg^{\Sigma}, F_{\Sigma}\bigr)$, and to cobordisms $M\colon \Sigma \rightarrow \Sigma'$ the equivalence class defined by the pair of fibrant span and natural transformation
 \begin{align*}
 C(M) = \bigl[\cala^{\Sigma} \sslash \calg^{\Sigma} \leftarrow \cala^M \sslash \calg^M \rightarrow \cala^{\Sigma'} \sslash \calg^{\Sigma'}, \mu_M\bigr].
 \end{align*}
\end{Theorem}
This finally makes it possible to define the Dijkgraaf--Witten TQFT with defects as the symmetric monoidal functor \cite[Theorem 6.8]{FMM}
\begin{align}
 Z_{\rm DW} := L \circ C\colon \ \Cob_3^{\mathsf{def}} \rightarrow \vect_{\C},
\end{align}
given by composing the two functors $C$ and the linearization functor $L$ from Proposition~\ref{propFunL}.

For a finite group $G$, restricting $Z_{\rm DW}$ to the subcategory of surfaces and cobordisms with trivial stratification and codimension-0 strata labeled by $G$ specializes to ordinary Dijkgraaf--Witten theory $Z_G$ for the finite group $G$. This finishes our summary of the results of \cite{FMM}.

\section{A boundary perspective on Turaev--Viro theories}
\label{secTVRes}
The results we develop in this section characterize the Turaev--Viro TQFT through a set of conditions (cf.\ Definition~\ref{defPropInv}) that give rise to the notion of \textit{boundary local} TQFTs (or more generally, defect invariants). These properties are formulated with respect to certain local topological moves of decorated manifolds that transform the defect boundary.

Throughout the section, we denote by $\calc$ an arbitrary spherical fusion category.
\subsection{Boundary local invariants}
\label{subSecChar}
Let us introduce the moves $\Delta_1$, $\Delta_2$ and $\Delta_3$, which can be applied locally to the underlying defect manifolds for cobordisms $M \in \Cob_3^{\mathcal{C}}(\Sigma, \Sigma')$ to produce new cobordisms $\Delta_i(M) \in \Cob_3^{\mathcal{C}}(\Sigma, \Sigma')$.
\begin{enumerate}\itemsep=0pt
 \item \textit{Pocket move:} The move $\Delta_1$ removes an open 3-ball (\textit{pocket}) from the interior of $M$. More explicitly, for any 3-ball $B^3 \subset \mathrm{Int}(M) $ embedded in the interior of $M$, we define $\Delta_1(M) = M\setminus \bigl(B^3\bigr)^{\circ}$, which we can again view as a cobordism in $\Cob_3^{\mathcal{C}}(\Sigma, \Sigma')$. Graphically, a neighborhood of the pocket transforms as shown in the following figure, where the vacuum region exposed by the removal of the pocket is depicted in white:
 \begin{align*}
 \adjincludegraphics[valign=c, scale = 0.72]{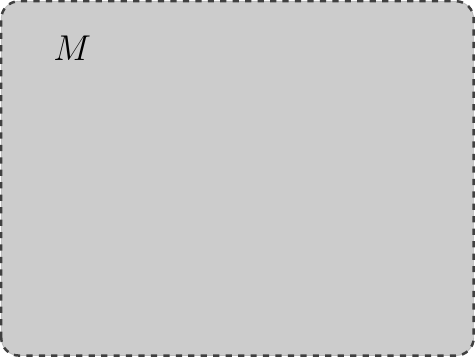} \, \,\, \, \xrightarrow{\Delta_1} \,\,\,\,\adjincludegraphics[valign=c, scale = 0.72]{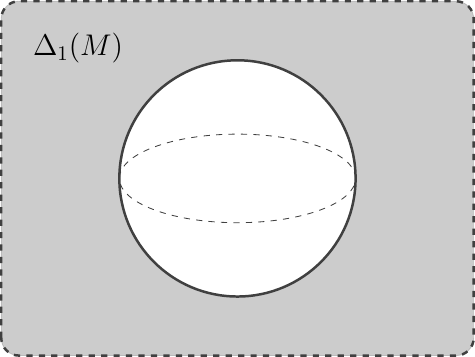}.
 \end{align*}
 \item \textit{Tunnel move:} This move removes a cylindrically shaped tunnel from $M$ and decorates the remaining part of the manifold with a labeled defect loop. The tunnel move can only be applied in the case of a nonempty free boundary, so we assume that $\partial_f M \neq \varnothing$. Consider a cylinder $C = \mathbb{D}^2 \times I$ together with an embedding $\iota\colon C \hookrightarrow M\setminus \partial_gM$ such that $C \cap \partial_f M = \mathbb{D}^2 \times \{0\} \cup \mathbb{D}^2 \times \{1\}$ and $C \cap \partial_f M$ does not intersect any boundary defect line or vertex. We call such a cylinder embedding a \textit{tunnel}. Let $X \in \calc$ be an object of the fusion category $\calc$. Then, we define $\Delta_2^X(\iota)(M) = M \setminus C^{\circ}$ with boundary decorated by an additional defect loop labeled by $X$ embedded in $M$ along the boundary of the slice $\mathbb{D}^2 \times \{1/2\} \subset C \cap M$.\footnote{Note that there are two ways to orient the $X$-labeled defect loop. Instead of defining two different moves, we will deal with this ambiguity by simply specifying the orientation of the loop whenever we apply $\Delta_2^X$.\label{footOrAmb}} A pictorial representation is as follows:
 \begin{align*}
 \adjincludegraphics[valign=c, scale = 0.72]{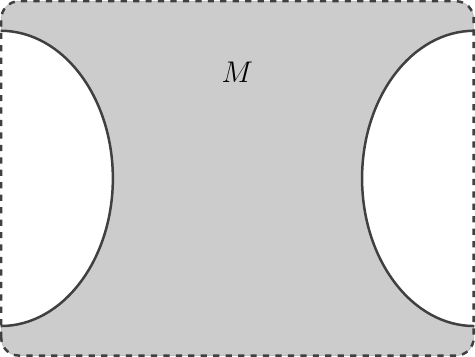} \, \xrightarrow{\Delta_2^X(\iota)} \, \adjincludegraphics[valign=c, scale = 0.72]{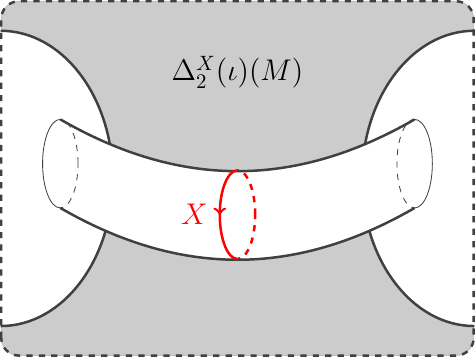}.
 \end{align*}
 The reason why we explicitly keep the embedding notationally is that for different embeddings $\iota, \iota'\colon C \hookrightarrow M$, the cobordisms $\Delta_2^X(\iota)(M)$ and $\Delta_2^X(\iota')(M)$ need not be equivalent.
 \item \textit{Tube-capping move:} The effect of the move described in the following is the removal of a cylindrical piece of manifold in combination with capping the unveiled boundary disks with 3-balls labeled by a pair of dual bases.

 Consider again a cobordism $M$ with nonempty free boundary. Consider further a~cylinder~$C$ embedded in $M\setminus\partial_gM$ such that precisely its lateral surface is on the free boundary, i.e., $C \cap \partial_f M = \partial \mathbb{D}^2 \times [0,1]$. We moreover require that the $\calc$-colored graph $\Gamma$ embedded in $\partial_f M$ has a particularly simple shape in the intersection $\Gamma \cap C \cap M$: It does not contain any vertex of $\Gamma$ and consists of a collection of finitely many parallel intervals such that one of the ends of each interval embeds in the boundary of the bottom disk $\mathbb{D}^2 \times \{0\}$ of the cylinder, and the other end embeds in the boundary of the top disk $\mathbb{D}^2 \times \{1\}$. The cylinder is oriented in the sense that we choose an orientation of its height, i.e., the interval $\{0\} \times [0,1]$. The intervals are labeled by objects $X_1, \dots, X_n \in \calc$, listed in agreement with the cyclic order defined by the orientation of $M$ and the orientation of $C$, similar to step~\ref{itemTVCon2} of the construction of $|\cdot|_{\calc}$ in Section~\ref{subSecTV}. We also define signs $\varepsilon(i)$ in the same manner.

 We define the two caps to be glued to the ends exposed by removing the cylindrical region~$C$ from the manifold. Pick a basis \smash{$\{\alpha_i\} \subset \Hom_{\mathcal{C}}\bigl(X_1^{\varepsilon(1)} \otimes \dots \otimes X_n^{\varepsilon(n)}, \mathbbm{1}\bigr)$} and let~$E^{\alpha_i}$ be a 3-ball with an open disk embedded in its boundary. This disk is decorated with a~coupon labeled by~$\alpha_i$ in the center of the disk, and strands labeled by~$X_1, \dots, X_n$ radially connecting (the bottom of) the coupon to the boundary of the disk. We identify the cylinder cap~$\mathbb{D}^2 \times \{0\}$ with the complement of the embedded open disk in the boundary of the 3-ball. It is a~closed disk denoted by $\overline\Sigma_X$ whose boundary is decorated with~$n$ points, each colored by a~pair consisting of an object~$X_i$ together with the sign $-\varepsilon(i)$.

 Considering the manifold $E^{\alpha_i}$ as a cobordism $\Sigma_X \rightarrow \varnothing $, we may reverse orientations, exchange source and target and decorate the vertex with an element of the dual basis \smash{$\{\alpha_i^*\} \subset \Hom_{\mathcal{C}}\bigl(\mathbbm{1}, X_1^{\varepsilon(1)} \otimes \dots \otimes X_n^{\varepsilon(n)}\bigr)$} to obtain a cobordism \smash{$\overline{E}^{\alpha_i^*}\colon \varnothing \rightarrow \Sigma_X$} in $\Cob_3^{\calc}$. We use it to define the cap on the opposite end, whose gluing boundary we identify with the top lid $\mathbb{D}^2 \times \{1\}$ of $C$. Now, we set
 \begin{align}
 \label{eqMovTube}
 \Delta_3^{\alpha_i}(M) = M \setminus C^{\circ} \cup_S \bigl(E^{\alpha_i} \cup \overline{E}^{\alpha_i^*}\bigr),
 \end{align}
 where the gluing happens along the set $S = \mathbb{D}^2 \times \{0\} \cup \mathbb{D}^2 \times \{1\}$ so that the manifold $\Delta_3^{\alpha_i}(M)$ emerging by this process is a well-defined defect manifold, regarded as a cobordism in $\Cob_3^{\calc}(\Sigma, \Sigma')$. If the set of defect edges intersecting the lateral surface of $C$ is empty, we will simply denote the resulting manifold by $\Delta_3(M)$. We illustrate the situation in the case of four defect lines,
 \begin{align*}
 \adjincludegraphics[valign=c, scale = 0.72]{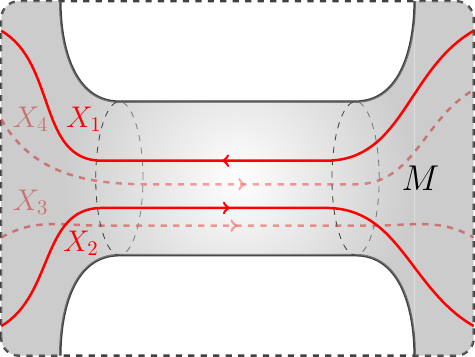} \,\, \xrightarrow{\Delta_3^{\alpha_i}} \,\, \adjincludegraphics[valign=c, scale = 0.72]{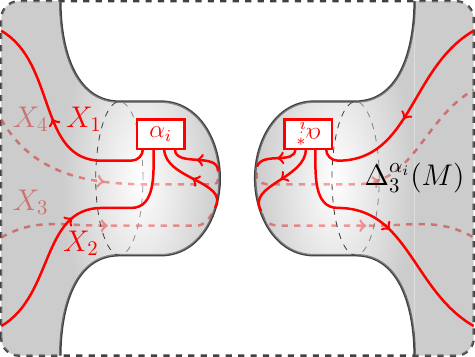},
 \end{align*}
 where $\{\alpha_i\}$ is a basis of $\Hom_{\calc}(X_1^* \otimes X_2 \otimes X_3 \otimes X_4, \mathbbm{1})$. The bottom of the right coupon in the second image appears at the top, as indicated by the rotation of its label. The black dotted lines trace the boundary of the two disks defining the gluing set $S$.
 \item Finally, we introduce a special defect 3-manifold $B_{\Gamma} \in \Cob_3^{\calc}(\varnothing, \varnothing)$, which consists of an underlying 3-ball $B^3$ whose boundary $\partial B^3$ is decorated with a $\mathcal{C}$-labeled graph $\Gamma$. The figure below displays $B_{\Gamma}$ for a graph $\Gamma$ that consists of five edges labeled by objects $V, W, X, Y, Z \in \calc$ and three vertices labeled by morphisms $\alpha$, $\beta$, $\gamma$ in $\calc$:
 \begin{align} \label{eqBGamma}
 B_{\Gamma} = \adjincludegraphics[valign=c, scale = 0.72]{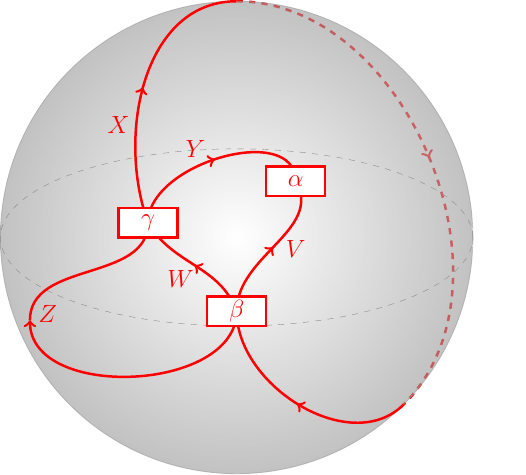}
 \end{align}
 A crucial feature of this cobordism is that the Turaev--Viro defect invariant only remembers its surface graph upon evaluation (see property~\ref{P4} in Definition~\ref{defPropInv}), which reflects the trivial topological nature of the 3-ball. This fact will be proven in Section~\ref{subSecProof}.
\end{enumerate}

\begin{Remark}
 For a uniform treatment of the moves, one can associate a move $\Delta_4^{\Gamma}$ with $B_{\Gamma}$. For a cobordism $M\colon \Sigma \rightarrow \Sigma'$ with a connected component $B_{\Gamma}$, the move simply removes that component, that is, $\Delta_4^{\Gamma}(M) = M \setminus B_{\Gamma}$. In the remainder of the text, we will not adopt this viewpoint and simply view the algebraic property \ref{P4} (cf. Definition~\ref{defPropInv}) as a normalization condition on decorated 3-balls.
\end{Remark}
\begin{Remark}
 Each of the three moves (and the decorated 3-ball) can be associated with the removal of a handle. More precisely, the move $\Delta_{4-k}^*$ removes a $k$-handle from a 3-manifold (for $k = 0, \dots, 3$). Since any compact 3-manifold admits a handle decomposition, this gives an indication why the four moves presented are necessary and sufficient. Alternatively, one can decompose 3-manifolds along (certain) skeleta $P \subset M$. Namely, considering the 3-cells defined by the embedding together with the stratification $P^{(1)} \subset P$, the manifold $M$ decomposes into cells of all codimensions. Every codimension/cell then corresponds to one of the four moves. Moreover, this also explains the topology of the occurring defect strata: For example, the tunnel move pierces the 2-cells of $P$, and the intersection of a cell with the boundary of a tunnel gives a distinguished loop on its boundary. Analogously, the tube-capping move leaves a boundary that intersects the 1-cell it is applied to in two distinct points, creating two vertices. For this point of view, see the proof of Theorem~\ref{thmMain}. We will provide some motivation for the labels of the defects in Remark~\ref{remDef}, after we have introduced the algebraic properties associated to the various moves in Definition~\ref{defPropInv}. Similar moves in related contexts were used, for example, in \cite{CGPT, LMWW, LMWW2, W}.
\end{Remark}
We introduce a set of algebraic properties of defect invariants (cf.\ Definition~\ref{defInv}) corresponding to the local moves $\Delta_1$, $\Delta_2^X(\iota)$, $\Delta_3^{\alpha_i}$ and the cobordism $B_{\Gamma}$. The following definition is central.
\begin{Definition} \label{defPropInv}
Let $\omega\colon \Cob_3^{\calc}(\varnothing, \varnothing) \rightarrow \K $ be a boundary-defect invariant. We say that $\omega$ satisfies properties \ref{P1}, \ref{P2}, \ref{P3} or \ref{P4}, whenever the respective property in the following list holds with respect to every defect manifold $M \in \Cob_3^{\calc}(\varnothing, \varnothing)$.
 \begin{enumerate}[(P1)]\itemsep=0pt
 \item \label{P1} Whenever the pocket move $\Delta_1$ can be applied locally to $M$ to produce a defect manifold $\Delta_1(M)$, the relation $\omega(\Delta_1(M)) = \dim (\mathcal{C})\, \omega(M)$ holds.
 \item \label{P2} For any tunnel $\iota\colon C \hookrightarrow M$, the relation $\omega(M) = \sum_{x \in \mathcal{O}(\mathcal{C})}\dim(x) \, \omega(\Delta_2^x(\iota)(M))$ between the invariant of $M$ and of the modified manifolds $\Delta_2^x(\iota)(M)$ is valid.
 \item \label{P3} For any possible tube-capping move on $M$, the equality $\omega(M) = \sum_{i = 1}^{k}\omega(\Delta_3^{\alpha_i}(M))$ holds, where $X_1, \dots, X_n \in \calc$ are objects labeling the lateral surface of the tube, $\{\alpha_i\}$ is a basis of the vector space \smash{$\Hom_{\calc}\bigl(X_1^{\varepsilon(1)} \otimes \dots \otimes X_n^{\varepsilon(n)}, \mathbbm{1}\bigr)$} and $k$ is its dimension.
 \item \label{P4} On the cobordism $B_\Gamma$, the invariant $\omega$ evaluates as $\omega(B_{\Gamma}) = \langle \Gamma \rangle$, where $\langle \Gamma \rangle \in \K$ denotes the value of the string diagram represented by $\Gamma$ considered as a morphism in $\End_{\calc}(\mathbbm{1})$.
 \end{enumerate}
 We will speak of a \textit{boundary local} defect invariant if $\omega$ satisfies all of the properties listed above.
\end{Definition}
\begin{Remark}
 \label{remDef}
 We give some motivation for the conditions~\ref{P1} to~\ref{P4}. As our eventual goal is to describe state sum theories, we will view these conditions through that lens. First, it is natural to assume that a local modification that cuts out a 3-ball only modifies the invariant by a nonzero scalar. The reason for fixing this scalar in~\ref{P1} to the global dimension of the fusion category can be seen by embedding an appropriate skeleton; for this, see the proof of Proposition~\ref{propProp1TV}.

 Second, for \ref{P2}, we can informally view the defect loop as labeled by the Kirby color, which corresponds to the expression $\sum_{x \in \calo(\calc)}\dim(x) x$. It has been established in several contexts that the Kirby color has the cloaking property, which makes it possible to pull labeled strands past a Kirby colored loop which occurs in an annular region (see, for instance, \cite{Ba, CGPV} in the context of string-net constructions). One way to motivate its appearance in our context is to imagine a~defect line on one of the boundary components prior to removing the tunnel. The tunnel could now attach to the region to the left of the defect line or to the region to the right of it. Since we want to demand invariance under this choice, one needs a way of pulling the defect line past this obstruction, which becomes possible through the Kirby color labeling of the loop on the boundary of the tunnel.

 Properties \ref{P3} and \ref{P4} can be motivated by the way the Turaev–Viro invariant is computed, by locally evaluating around vertices using link graphs and the scalar $*_c$ (see Section~\ref{subSecTV}) constructed from pairs of dual bases. Further motivation is given in Section~\ref{secIntro}, and details can be found in the proof of Theorem~\ref{thmMain}.
\end{Remark}
If an invariant satisfying some of the properties listed in Definition~\ref{defPropInv} comes from a defect TQFT, we will also say that the defect TQFT has the respective properties.
\begin{Remark}
 \label{remPPrime}
 For a boundary-defect TQFT $Z\colon \Cob_3^{\calc} \!\rightarrow\! \vect_{\C}$, it is a natural question whether~$Z$ satisfies the generalizations (P1)*, (P2)*, (P3)* and (P4)* of properties~\ref{P1} to~\ref{P4} to arbitrary cobordisms $\Sigma \rightarrow \Sigma'$, not just cobordisms in $\Cob_3^{\calc}(\varnothing, \varnothing)$. This is true for the Turaev--Viro TQFT, cf. Theorem~\ref{thmMainTV} and Remark~\ref{remTVPPrime}.
\end{Remark}
\begin{Remark} \label{remGenTun}
It is natural to generalize tunnels by allowing embeddings of cylinders over bases~$A_n$, which are disks with $n$ disjoint open disks removed from their interior. For example, $A_1$ is the annulus $\Sp^1 \times I$ and $A_0 \cong \mathbb{D}^2$ the disk. For an object $X \in \calc$, we can generalize the tunnel move by removing the interior of $C = A_n \times I$ from a defect cobordism $M$ (with base and top of the embedded cylinder at the free boundary) and embed one defect loop labeled by~$X$ for every boundary component of~$A_n$ along the boundary of the slice~$A_n \times \{1/2\} \subset C \cap M$. The orientations of the loops are induced from a chosen orientation of~$A_n$ (cf. footnote~\ref{footOrAmb}). By convention, a~tunnel without further mention of the base will always have a disk $A_0$ as the base of its defining cylinder.

For a tunnel with base $A_n$, we can modify~\ref{P2} and demand for a defect invariant~$\omega$
\begin{align}\label{eqGenTun}
\omega(M) = \sum_{x \in \mathcal{O}(\mathcal{C})}\dim(x)^{\chi(A_n)} \, \omega(\Delta_2^x(\iota, A_n)(M)),
 \end{align}
 where $\Delta_2^x(\iota, A_n)$ is the generalized tunnel move and $\chi(A_n) = 1-n$ is the Euler characteristic of~$A_n$. Note that the case $n = 0$ specializes to property~\ref{P2}.
\end{Remark}
We sketch several (admittedly somewhat contrived) examples of defect invariants derived from the Turaev--Viro invariant that satisfy none, some or all of the properties~\ref{P1} to~\ref{P4}. We hope this allows the interested reader to gain some insight into the mutual independence of the requirements.
\begin{Example}\label{exProp}\quad
 \begin{enumerate}[(i)]\itemsep=0pt
 \item The invariant assigning $1 \in \K$ to every cobordism in $\Cob_3^{\calc}(\varnothing, \varnothing)$ is multiplicative, satisfies properties~\ref{P1} and~\ref{P2} if and only if $\dim(\calc) = 1$ and $\sum_{x \in \calo(\calc)}\dim(x) = 1$ respectively, and fails property~\ref{P4} for any~$\calc$. It never satisfies~\ref{P3}, as that would require $\dim_{\K}\Hom_{\calc}(X_1 \otimes \dots \otimes X_n, \mathbbm{1}) = 1$ for any collection $X_i \in \calc$. But for $n = 1$ and the object $X_1 = \mathbbm{1} \oplus \mathbbm{1}$, which exists in every fusion category, this is obviously false.
 \item The Turaev--Viro invariant $| \cdot |_{\calc}$ is boundary local, i.e., it satisfies all of the properties \ref{P1} up to \ref{P4}. This is the content of Theorem~\ref{thmMainTV}.
 \item The forgetful functor $U\colon \mathcal{Z}(\calc) \rightarrow \calc$ induces a symmetric monoidal functor $\widetilde{U}\colon \Cob_3^{\mathcal{Z}(\calc)} \rightarrow \Cob_3^{\calc}$. The restriction of the composition $|\cdot|_{\calc} \circ {\widetilde{U}}$ to a defect invariant $\omega_U\colon\! {\smash{\Cob_3^{\mathcal{Z}(\calc)}\!(\varnothing, \varnothing)}\! \rightarrow\! \K}$ is multiplicative, fails properties \ref{P1} to \ref{P3} for general $\calc$, but satisfies \ref{P4}. The latter claim directly follows from the fact that $U$ faithfully embeds $\Hom_{\mathcal{Z}(\calc)}(X, X') \subset \Hom_{\calc}(UX, UX')$. That \ref{P1} (respectively~\ref{P2}) are not satisfied is immediately checked on, for example, $\Sp^3$ (respectively $\Sp^2 \times I$) using that $\dim(\mathcal{Z}(\calc))$ is not equal to $\dim(\calc)$ in general. As a counterexample to \ref{P3}, let for instance $\calc = \vect_{\mathbb{Z}/2}$, consider the nontrivial simple object $\mathbbm{1}\ncong e \in \mathcal{Z}(\calc) \simeq \vect_{\mathbb{Z}/2 \times \mathbb{Z}/2}$ with $U(e) = \mathbbm{1} \in \calc$ and use it to decorate the defect manifold $\mathbb{T}_{e, \mathbbm{1}}^s$ introduced in Lemma~\ref{lemTorxy} in Section~\ref{secTVRes}. Assuming \ref{P3} for $\omega_U$, we see as in Remark~\ref{remTorus} that
 \begin{align*}
 \omega_U (\mathbb{T}_{e, \mathbbm{1}}^s ) = \dim_{\K}\Hom_{\mathcal{Z}(\calc)}(e, \mathbbm{1}) = 0.
 \end{align*}
 On the other hand, by Lemma~\ref{lemTorxy}, it must hold that $\omega_U (\mathbb{T}_{e, \mathbbm{1}}^s ) = \bigl|\mathbb{T}_{U(e), \mathbbm{1}}^s\bigr|_{\calc}$ is equal to $\dim_{\K}\Hom_{\calc}(U(e), \mathbbm{1}) = 1$, a contradiction.
 \item \label{itemExProp} Define a symmetric monoidal functor $F\colon \Cob_3^{\calc} \rightarrow \Cob_3^{\calc}$ which sends any defect cobordism $M_{\Gamma}\colon \Sigma_{\Gamma} \rightarrow \Sigma_{\Gamma}'$ decorated with a $\calc$-colored graph to the undecorated underlying cobordism $M\colon \Sigma \rightarrow \Sigma'$, forgetting the defect structure. The invariant associated to the composition $|\cdot|_{\calc} \circ F$ is multiplicative and satisfies \ref{P1}, but generally fails \ref{P2} to \ref{P4}.
 \item It suffices to check the properties \ref{P1} to \ref{P4} on connected defect manifolds. The direct sum $|\cdot|_{\calc} \oplus |\cdot|_{\calc}$ of TQFTs gives a multiplicative defect invariant that coincides on connected manifolds with $2|\cdot|_{\calc}$. Hence, it satisfies \ref{P1}--\ref{P3}, but not~\ref{P4}.
 \item Postcomposing $| \cdot |_{\calc}$ with a nontrivial field automorphism $f\colon \K \rightarrow \K$, the invariant $f \circ | \cdot |_{\calc}$ is multiplicative, satisfies \ref{P3} and never satisfies \ref{P4}. If we fix $\K = \C$ and $f = \overline{(-)}$ complex conjugation, we can say more: As for any spherical fusion category over the complex numbers, the dimensions $\dim(X)$ for $X \in \calc$ are real numbers \cite[Corollary~2.10]{ENO}, which are fixed by $f$, properties~\ref{P1} and~\ref{P2} hold as well.
 \item \label{itemAlg} Consider the example from the last point for $\K = \mathbb{A} \subset \C$ the field of algebraic numbers. Denote by $\varphi, \psi \in \mathbb{A}\setminus \mathbb{Q}$ the positive and negative solution of the equation $\lambda^2 = \lambda + 1$, respectively. Fix the standard inclusion $\mathbb{Q} \hookrightarrow \mathbb{A}$. Since $\varphi$ is algebraic over $\mathbb{Q}$, we can extend the inclusion to a field homomorphism $\mathbb{Q}(\varphi) \hookrightarrow \mathbb{A}$ sending $\varphi$ to $\psi$, which can further be extended to a field automorphism $f\colon \mathbb{A} \rightarrow \mathbb{A}$ by the isomorphism extension theorem, since the field extension $\mathbb{A}/\mathbb{Q}$ is algebraic. Let $\calc = \mathcal{F}_{\varphi}$ denote the Fibonacci category with two simple objects $\mathbbm{1}, \tau$ satisfying the fusion rule $\tau \otimes \tau \cong \mathbbm{1} \oplus \tau$. It can be equipped with the structure of a spherical fusion category with dimension $\dim(\tau) = \varphi$ of the nontrivial simple object and categorical dimension $\dim(\mathcal{F}_a) = 2 + \varphi$.

 Let $M = \mathbb{T}^s \setminus (B^3)^{\circ}$ be a solid torus with a pocket removed from its interior. Then,
 \begin{align*}
 |M|_{\mathcal{F}_{\varphi}} \overset{\rm \ref{P3}}{=} \bigl|{\Sp^2 \times I}\bigr|_{\mathcal{F}_{\varphi}} = \bigl|\Delta_1\bigl(B^3\bigr)\bigr|_{\mathcal{F}_{\varphi}} \overset{\rm \ref{P1}}{=} \dim(\mathcal{F}_{\varphi})\bigl|B^3\bigr|_{\mathcal{F}_{\varphi}} \overset{\rm \ref{P4}}{=} 2+\varphi,
 \end{align*}
 where we have used that $|\cdot|_{\mathcal{F}_{\varphi}}$ is boundary local. Consider the tunnel $\iota\colon C \hookrightarrow M$ connecting the pocket with itself along a non-contractible loop in interior of the solid torus. One can verify that \smash{$\bigl|\Delta_2^{\mathbbm{1}}(\iota)(M)\bigr|_{\mathcal{F}_{\varphi}} = 2$} and $|\Delta_2^{\tau}(\iota)(M)|_{\mathcal{F}_{\varphi}} = 1$, see Observation~\ref{obs}\,(\ref{obsEx}) (where the notation $\vartheta(H, \mathbbm{1},\mathbbm{1})$ (respectively $\vartheta(H, \mathbbm{1},\tau)$) is used for the two invariants). The equation~\ref{P2} for $|\cdot|_{\mathcal{F}_{\varphi}}$ then turns into
 \begin{align*}
 2+\varphi = 1\cdot 2 + \varphi \cdot 1.
 \end{align*}
 However, for the invariant $\omega_f = f \circ |\cdot|_{\mathcal{F}_{\varphi}}$, the right-hand side stays the same, but the left-hand side changes into $2 + \psi$. Thus, $\omega_f$ does not satisfy~\ref{P2}, and one can easily show that it also does not satisfy~\ref{P1}.
 \end{enumerate}
\end{Example}

We are now ready to state and prove the first main result of this article, which asserts that a multiplicative defect invariant $\omega \colon \Cob_3^{\calc}(\varnothing, \varnothing) \rightarrow \K$ having all of the properties listed in Definition~\ref{defPropInv} is in fact the Turaev--Viro invariant.

\begin{Theorem}
 \label{thmMain}
 Let $\mathcal{C}$ be a spherical fusion category and $\omega\colon \Cob_3^{\mathcal{C}}(\varnothing, \varnothing) \rightarrow \K$ a multiplicative boundary-defect invariant. If $\omega$ is boundary local $($cf.\ Definition~{\rm \ref{defPropInv})}, then $\omega$ is the Turaev--Viro invariant for the spherical fusion category $\calc$, that is, $\omega = | \cdot |_{\calc}$.
\end{Theorem}
\begin{proof}
 We divide the proof into several steps.

 (1) For simplicity, assume first that $M$ is a closed 3-manifold, thus necessarily without defects. Then, the nodes, rims and faces of any graph skeleton coincide with the vertices, edges and regions of the skeleton, respectively. We select a skeleton $P$ of $M$ with the feature that every region $a \in \mathrm{Reg}(P)$ is homeomorphic to an open disk, $a \cong \bigl(\mathbb{D}^2\bigr)^{\circ}$. This choice is always possible, since a compact 3-manifold admits a finite triangulation $t$ with 2-skeleton $t^{(2)}$, and we may put, for example, $P = t^{(2)}$ to obtain a skeleton of the above kind. Note that the constraint on $P$ imposes $\chi(a) = 1$ for the Euler characteristic of any region $a$.

 We now remove a 3-ball from each 3-cell of $P$ by applying the pocket move $\Delta_1$ a total of $|M\setminus P|$ times. If $M'$ denotes the manifold after this modification (which has a free boundary homeomorphic to a disjoint union of $|M\setminus P|$ spheres), we have by condition \ref{P1}
 \begin{align}
 \label{eqT1}
 \omega(M) = \dim(\mathcal{C})^{-|M\setminus P|} \, \omega(M').
 \end{align}
 Notice that by properties of a skeleton, every 3-cell is homeomorphic to an open 3-ball. This implies that $P$ is a deformation retract of~$M'$.

 (2) Let $c\colon \mathrm{Reg}(P) \rightarrow \mathcal{O}({\mathcal{C}})$ be an arbitrary coloring of the skeleton. Every region $a$ is adjacent to at most two different 3-cells. For every $a$ we intend to apply the tunnel move \smash{$\Delta_2^{c(a)}(\iota_a)$} along a natural tunnel $\iota_a\colon C \rightarrow M'$ piercing through $a$, connecting a pair of neighboring previously removed 3-balls in the case of two adjacent 3-cells, and connecting a removed pocket to itself in case of only one adjacent cell. The idea is to avoid that this tunnel is knotted, and we can build such a tunnel as follows.

 First, assume that $a$ has two distinct neighboring 3-cells $o_1, o_2$ in the original manifold $M$. Since 3-cells are homeomorphic to open balls, there is a homeomorphism $f$ that identifies the glued space $o_1 \cup a \cup o_2$ with a standard open 3-ball $B_1(0)^{\circ} \subseteq \R^3$ in such a way that $f(a)$ is the intersection $B_1(0)^{\circ} \cap \R^2 \times \{0\}$ of the ball with the $xy$-plane, dividing it into two solid hemispheres. Additionally, we require $f$ to identify the two pockets in $o_1$ respectively $o_2$ with small balls $B_{\varepsilon/2}((0, 0, 1-\varepsilon)$ and $B_{\varepsilon/2}((0, 0, -1+\varepsilon)$ on the $z$-axis. We can now define an embedding of a~cylinder~$C$ whose cylinder axis lies on the $z$-axis and connects both pockets, and postcompose this embedding with $f^{-1}$ to obtain a tunnel $\iota_a\colon C \hookrightarrow M$ of the desired shape. If $a$ is neighbored by only one 3-cell~$o$ on both sides, then $o \cup a$ is homeomorphic to an open solid torus $\bigl(\Sp^1 \times \mathbb{D}^2\bigr)^{\circ}$ and we can drill the tunnel along the core curve $\{0\} \times \Sp^1$ in the presence of the appropriate identifications.

 By abuse of notation, we view $\iota_a$ as a tunnel in $M'$ or any modification of $M'$ that arises by applying tunnel moves for regions $a' \neq a$, where we tacitly assume that the homeomorphisms chosen in the construction of the tunnels $\iota_a$ are such that different tunnels do not intersect. Iterating over the regions $a \in \mathrm{Reg}(P)$, we successively perform the tunnel moves \smash{$\Delta_2^{c(a)}(\iota_a)$}. We let the orientation of the defect loop on each region be induced by the encircled interior (removed by the tunnel) and denote by $M'_c$ the resulting defect manifold. We denote by $\mathrm{im}(c) = \{x_1, \dots, x_{s}\} \subset \mathcal{O}({\mathcal{C}})$ the collection of simple objects that decorate the regions of~$P$, where we have put $s = |\mathrm{Reg}(P)|$.

 Then, it follows by condition \ref{P2} that the invariant transforms as
 \begin{align}
 \label{eqT2}
 \omega(M') = \sum_{x_1, \dots, x_{s} \in \mathcal{O}({\mathcal{C}})} \dim(x_1) \cdots \dim(x_{s}) \omega (M'_c ).
 \end{align}

 (3) Note that topologically, the 3-manifold with boundary $M'_c$ deformation retracts onto the 1-skeleton $P^{(1)}$ of the skeleton $P$. This follows from the fact that $P$ is a deformation retract of $M'$, every region $r$ of $P$ is homeomorphic to a disk, and the shape of the tunnels $\iota_r$. Thus, along every edge $r \in R(P)$, we can embed a cylinder $C_r = \mathbb{D}^2 \times I$ in $M'_c$ with its height in the interior of $r$ in a way that the tube-capping move is applicable at $C_r$. Each of these cylinders $C_r$ is decorated with as many parallel defect lines as adjacent branches of $r$. We now pick for each edge $r \in R(P)$ a rim basis (cf.\ Section~\ref{subSecTV}) $\{\alpha_{j_r}^r\}_{j_r=1}^{k_r}$ of \smash{$\Hom_{\calc}\bigl(y_1^{\varepsilon(1)} \otimes \dots \otimes y_{n_r}^{\varepsilon(n_r)}, \mathbbm{1}\bigr)$}, where $y_1, \dots, y_{n_r} \in \mathrm{im}(c)$ decorate the branches of $r$.

 Consider the index set $J = \prod_{r \in R(P)} \{1, \dots, k_r\}$. For a multiindex $j = (j_r)_r \in J$, we may perform for all $r \in R(P)$ the tube-capping move \smash{$\Delta_3^{\alpha_{j_r}^r}$} on $C_r$ and obtain a~mani\-fold~$M'_{c,j}$ from~$M'_c$. Topologically, this manifold is homeomorphic to a disjoint union of 3-balls, that is, $M'_{c,j} \cong \bigsqcup_{v \in N(P)}B_{\Upsilon_v^c(j)}$, where for each vertex of~$P$ there is one 3-ball decorated with a $\calc$-colored graph $\Upsilon_v^c(j)$. By construction, this graph is given as the link graph of the vertex $v$ in the Turaev--Viro construction, evaluated on the chosen basis elements. More precisely, we have that $\langle\Upsilon_v^c(j)\rangle = \Gamma_c(v)\bigl( \bigotimes_{e \in R^h(v)}\alpha_{j_e}^e\bigr)$.

 We now evaluate the invariant $\omega$ on this collection of balls to obtain
 \begin{align*}
 \omega (M'_{c,j} ) = \omega\biggl(\bigsqcup_{v \in N(P)}B_{\Upsilon_v^c(j)}\biggr) = \prod_{v \in N(P)}\omega ( B_{\Upsilon_v^c(j)} ) \overset{\rm \ref{P4}}{=} \prod_{v \in N(P)} \langle \Upsilon_v^c(j) \rangle,
 \end{align*}
 where we used multiplicativity of $\omega$ and applied condition~\ref{P4}. With the choice of rim bases above, we can express the vector $*_c$ by reindexing over $J$ as
 \begin{align*}
 *_c = \bigotimes_{r \in R(P)} \sum_{j_r = 1}^{k_r} (\alpha_{j_r}^r)^* \otimes \alpha_{j_r}^r = \sum_{j \in J} \bigotimes_{e \in R^h} \alpha_{j_e}^e.
 \end{align*}
 Using property \ref{P3} and the characterization of the graph $\Upsilon_v^c(\alpha)$ obtained above yields
 \begin{align}
 \label{eqT3}
 \omega (M'_c ) &\overset{\rm \ref{P3}}{=} \sum_{j \in J} \prod_{v \in N(P)}\langle \Upsilon_v^c(j) \rangle = \sum_{j \in J} \prod_{v \in N(P)}\Gamma_c(v)\biggl( \bigotimes_{e \in R^h(v)} \alpha_{j_e}^e\biggr) \nonumber \\
 &= \sum_{j \in J} \biggl( \bigotimes_{v \in N(P)}\Gamma_c(v)\biggr)\biggl( \bigotimes_{e \in R^h} \alpha_{j_e}^e\biggr) = \biggl( \bigotimes_{v \in N(P)}\Gamma_c(v)\biggr)(*_c) = V_c(*_c),
 \end{align}
 where $V_c(*_c) = |c| \in \K$ is the $c$-dependent scalar appearing in the definition of the Turaev--Viro invariant $|\cdot|_{\mathcal{C}}$ (cf.~\eqref{eqInv}).

 (4) Combining the results from \eqref{eqT1}, \eqref{eqT2} and \eqref{eqT3} and recalling that all regions have Euler characteristic equal to~1, we see that{\samepage
 \begin{align*}
 \omega(M) &= \dim(\mathcal{C})^{-|M\setminus P|}\sum_{x_1, \dots, x_{s} \in \mathcal{O}({\mathcal{C}})} \dim(x_1) \cdots \dim(x_{s})\ |x_1, \dots, x_{s}| \\
 &= \dim(\mathcal{C})^{-|M\setminus P|} \sum_{c\colon \mathrm{Reg}(P) \rightarrow \mathcal{O}({\mathcal{C}})} \dim(c) |c|.
 \end{align*}
 But this expression is simply the state sum formula~\eqref{eqInv}, proving that $\omega(M) = |M|_{\mathcal{C}}$.}

 (5) Now, suppose that $M \in \Cob_3^{\mathcal{C}}(\varnothing, \varnothing)$, which generalizes the previous situation to cobordisms with free boundary decorated with defect graphs. It is easy to adapt the proof for closed manifolds given above to this more general setting: Since there always exists a triangulation for a manifold with boundary, we can find a skeleton $P$ with contractible regions compatible with the boundary of $M$. We can ensure compatiblity of $P$ with the defect structure in the sense of Definition~\ref{defSkel} by deforming $P$ by a homotopy if necessary. The pocket move $\Delta_1$ is applied in the same manner as before to every 3-cell. The tunnel move \smash{$\Delta_2^{c(a)}$} is now applied to every face $a \in \mathrm{Fac}(P)$, including in particular the faces $a \subset \partial M$ on the boundary. The tube-capping move is applied to every rim, including the defect rims (note that the cylinders are decorated with defect lines from the tunnel move and additional ones from $M$), and the skeletal vertices numbering the collection of balls are replaced by the entire collection of nodes. With these replacements, the resulting collection of 3-balls is now decorated with the generalized link graphs including information obtained from the $\calc$-colored defect graph, and the argument carries through as before. 
\end{proof}

\begin{Remark} \label{remSkelTun}
 If $\omega$ is compatible with the generalized tunnel move from Remark~\ref{remGenTun}, the argument in the proof of the preceding theorem no longer relies on the fact that skeletal regions can be chosen homeomorphic to open disks. Let $P$ be \textit{any} skeleton of a defect manifold $M$. The condition that 3-cells be homeomorphic to 3-balls from Definition~\ref{defGraphSkel} restricts the possible topologies of the faces of $P$. Every face $a$ embeds in a 2-sphere and is therefore homeomorphic to (the interior of) a sphere with multiple (including none at all) disjoint open disks removed. This leaves the options $a \cong \Sp^2$ and $a \cong A_n^{\circ}$.

 Assume first that there is a region $a \cong \Sp^2$. Since $\omega$ is multiplicative, we can restrict attention to connected 3-manifolds. The only possible connected (defect) 3-manifolds admitting a face homeomorphic to a 2-sphere are the 3-ball and the 3-sphere. The case of the 3-ball is covered by~\ref{P4} and for the 3-sphere it is easily verified that the conditions on $\omega$ imply $\omega\bigl(\Sp^3\bigr) = \dim(\calc)^{-1}$.

 In the case where for every face $a \in \Fac(P)$ there is a natural number $n$ with $a \cong A_n^{\circ}$, we perform the first step of the proof as before and then use the generalized tunnel move \smash{$\Delta_2^{c(a)}(\iota_a, A_n)$} for suitable tunnels $\iota_a$ at each face to obtain once more a manifold of which \smash{$P^{(1)}$} is a deformation retract. Otherwise, we can leave the proof of Theorem~\ref{thmMain} unchanged. This adds Euler characteristic exponents and modifies the final expression in the above proof to
 \begin{align*}
 \omega(M) &= \dim(\mathcal{C})^{-|M\setminus P|}\sum_{c\colon \Fac(P) \rightarrow \calo(\calc)} \prod_{a \in \mathrm{Fac}(P)}\dim(c(a))^{\chi(a)} |c|,
 \end{align*}
 the right-hand side of which is exactly the state sum formula from~\eqref{eqInv} for an arbitrary skeleton~$P$.
\end{Remark}

\begin{Remark}
 Assume that the invariant $\omega = Z|_{\Cob_3^{\calc}(\varnothing, \varnothing)}$ is the restriction of a full-fledged boundary-defect TQFT $Z$. This is a strong requirement, as it allows verifying \ref{P1} up to \ref{P3}, which correspond to the local topological moves $\Delta_1$, $\Delta_2^x(\iota)$ and $\Delta_3^{\alpha_i}$, for small pieces of manifold excised around the local modification. In this approach, we use that a TQFT is local in the sense that it is functorial with respect to gluing in the cobordism category. We will apply this strategy for Turaev--Viro theory while proving Theorem~\ref{thmMainTV} in Section~\ref{secTVRes} and similarly for Dijkgraaf--Witten theory in Section~\ref{secDW}.
\end{Remark}

We reiterate that under the assumptions of Theorem~\ref{thmMain}, in particular the restrictions of the invariants $\omega$ and $| \cdot |_{\calc}$ to closed 3-manifolds have to agree, which are morphisms in the subcategory $\Cob_3 \subset \Cob_3^{\calc}$ and therefore necessarily come without any defect structure. However, a compatibility of the invariant with defects is necessary to even state the result. Example~\ref{exProp}\,\ref{itemExProp} moreover shows that a defect invariant that only agrees with Turaev--Viro on closed manifolds does not have to be boundary local. Nevertheless, there is a converse of Theorem~\ref{thmMain}, stating that the Turaev--Viro defect invariant is boundary local. The proof of this comprises Section~\ref{subSecProof}.

\subsection{Boundary-defect characterization of Turaev--Viro TQFTs}
\subsubsection{Boundary locality of the Turaev--Viro defect invariant}
\label{subSecProof}
In this subsection, we will give the converse of Theorem~\ref{thmMain} and its proof.
\begin{Theorem}
 \label{thmMainTV}
 For every spherical fusion category $\calc$, the Turaev--Viro boundary-defect theory $| \cdot |_{\calc}\colon \Cob_3^{\calc} \rightarrow \vect_{\K}$ defines a boundary local {\rm (}Definition~{\rm \ref{defPropInv})} defect TQFT.
\end{Theorem}
Let us start by proving that $|\cdot|_{\calc}$ satisfies \ref{P1}.
\begin{Proposition}
 \label{propProp1TV}
 Let $M \in \Cob_3^{\mathcal{C}}(\varnothing, \varnothing)$. Then, it holds that $|\Delta_1(M)|_{\mathcal{C}} = \dim(\mathcal{C})|M|_{\mathcal{C}}$ under application of the pocket move $\Delta_1$.
\end{Proposition}
 \begin{proof}
 The move $\Delta_1$ (as all moves introduced in Section~\ref{subSecChar}) is local. Hence, the general idea is to compare the value of the Turaev--Viro TQFT on a local piece of manifold excised in the neighborhood of the region modified by $\Delta_1$ both before and after application of the pocket move, and then utilize functoriality of $|\cdot |_{\mathcal{C}}$ to compare the TQFT-values on the entire cobordisms $M$ and $\Delta_1(M)$.

 Let us fix a pocket and a small embedded 2-sphere $\Sp^2 \subset M$ just surrounding it. We decompose the original cobordism as $M\colon \varnothing \xrightarrow{B} \mathbb{S}^2 \xrightarrow{M'} \varnothing$ and the cobordism modified by the pocket move as $\Delta_1(M)\colon \varnothing \xrightarrow{C} \mathbb{S}^2 \xrightarrow{M'} \varnothing$, where $B\colon \varnothing \rightarrow \mathbb{S}^2$ is the 3-ball and $C\colon \varnothing \rightarrow \mathbb{S}^2$ is a cylinder $C \cong \mathbb{S}^2 \times I$ over $\Sp^2$. The cobordism $C$ has as free boundary $\partial_fC \cong \mathbb{S}^2 \times \{0\}$, the inner sphere. The outer sphere $\Sp^2 \times \{1\}$ is the gluing boundary (the same as for $B$). Noting that $M'$ is the same cobordism in both decompositions, namely $M' = M \setminus B = \Delta_1(M) \setminus C$, it is then enough to show that $|C|_{\mathcal{C}} = \dim(\mathcal{C})|B|_{\mathcal{C}}$, as this implies
 \begin{align*}
 |\Delta_1(M)|_{\mathcal{C}} = |M' \cup_{\mathbb{S}^2} C|_{\mathcal{C}} = |M'|_{\mathcal{C}} \circ |C|_{\mathcal{C}} = \dim(\mathcal{C})\cdot |M'|_{\mathcal{C}} \circ |B|_{\mathcal{C}} = \dim(\mathcal{C})|M|_{\mathcal{C}}.
 \end{align*}
 To show the identity claimed above, we pick a 2-polyhedron $P$ serving as a graph skeleton simultaneously for both of the cobordisms $B$ and $C$, which we define as
 \begin{align*}
 P := \Sp^2 \cup_{\Sp^1} \bigl(\Sp^1 \times I\bigr) = \adjincludegraphics[valign=c, scale = 0.7]{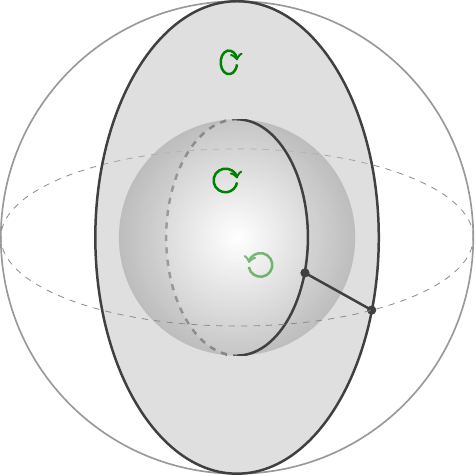},
 \end{align*}
 by gluing an annulus and a 2-sphere. Here, the gluing happens along a meridian of the inner 2-sphere $\Sp^2$, and the stratification $P^{(1)}$, alongside the chosen orientations, can be read from the picture: There are three faces, shaded in gray, three edges (the thick, gray lines) and two vertices (the gray dots). The gluing boundary is indicated by the outer thin lines. For the cobordism $C$, we demand the extra condition that the inner sphere $\Sp^2 \subset P$ covers the free boundary $\partial_fC$, cf. Definition~\ref{defSkel}.

 With this choice it is clear that $B$, in comparison with $C$, has an additional 3-cell with respect to the embedded skeleton $P$, namely the inner ball. Thus, $|C\setminus P| = |B \setminus P| +1$. From this we see immediately from the definition of the Turaev--Viro invariant (cf.~\eqref{eqInv}) that
 \begin{align}
 \label{eqDim}
 |C, \varnothing, G_P|_{\mathcal{C}} = \dim(\mathcal{C}) |B, \varnothing, G_P|_{\mathcal{C}},
 \end{align}
 where $G_P$ denotes the skeleton of $\partial_gC = \partial_gB = \Sp^2$ induced by the graph skeleton $P$. Now, since $|B|_{\mathcal{C}} = \lim_G|B, \varnothing, G|_{\mathcal{C}}$, by considering the commutative diagram
 \[
 \centering
 \adjincludegraphics[valign=c, scale = 1.16]{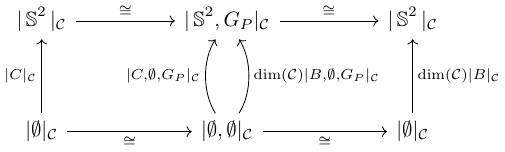}
 \]
 it becomes evident that the relation in \eqref{eqDim} holds identically in the limit over graph skeleta, since the horizontal isomorphisms, the structure isomorphisms of the limit, compose to the identity.
\end{proof}

We will now prove two lemmas from which we can obtain~\ref{P2} as an easy corollary.

\begin{Lemma}
\label{lemTorxy}
 Let $x, y \in \mathcal{C}$ be two objects and define the cobordism $\mathbb{T}_{x,y}^s\colon \varnothing \rightarrow \varnothing$ in $\Cob_3^{\mathcal{C}}$ as the solid torus $\mathbb{T}^s = \mathbb{D}^2 \times \Sp^1$ whose free boundary surface is decorated with two defect lines labeled by $x$ and $y$, as depicted in the following figure:
 \begin{align*}
 \mathbb{T}_{x,y}^s = \adjincludegraphics[valign=c, scale = 0.7]{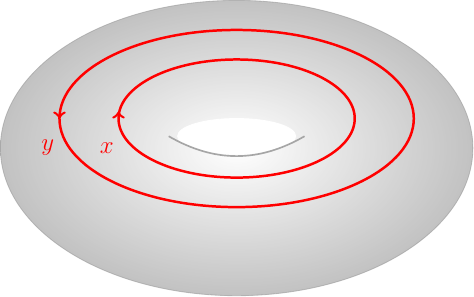}.
 \end{align*}
 Then, the Turaev--Viro invariant is given by $|\mathbb{T}_{x,y}^s|_{\mathcal{C}} = \dim_{\K}\left( \Hom_{\mathcal{C}}(x, y) \right) \in \K$. If $x$ and $y$ are simple, then $|\mathbb{T}_{x,y}^s|_{\mathcal{C}} = \delta_{[x], [y]}$.
\end{Lemma}

\begin{proof}
 Define the graph skeleton $P$ for the cobordism $\mathbb{T}_{x,y}^s$ pictorially by
 \begin{align*}
 P = \adjincludegraphics[valign=c, scale = 0.75]{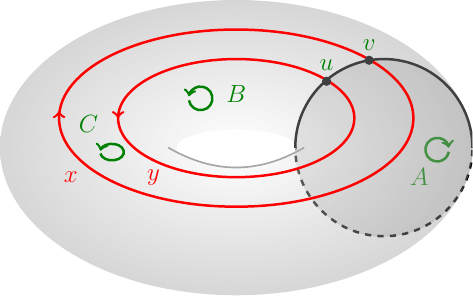}.
 \end{align*}
 It consists of the underlying 2-polyhedron $\mathbb{D}^2 \cup_{\partial \mathbb{D}^2} \bigl(\Sp^1 \times \Sp^1\bigr)$, glued as illustrated in the picture, subject to the additional condition that the free boundary 2-torus $\partial_f\mathbb{T}_{x,y}^s = \Sp^1 \times \Sp^1$ is part of the skeleton $P$. The skeleton $P$ has one 3-cell, three faces $A, B$ and $C$ with orientation as indicated in the image, and two nodes (switches) labeled by $u$ and $v$ at the intersection points of the defect lines with the face $A$. The Euler characteristic of the faces are $\chi(A) = \chi(B) = \chi(C) = 1$. For a~chosen coloring $\mu\colon \Fac(P) \rightarrow \calo(\calc)$ of $P$, the link graphs for the two nodes are
 \begin{align*}
 \Gamma_{\mu}(v) = \adjincludegraphics[valign=c, scale = 0.75]{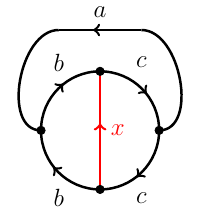}, \, \, \, \, \, \, \, \Gamma_{\mu}(u) = \adjincludegraphics[valign=c, scale = 0.75]{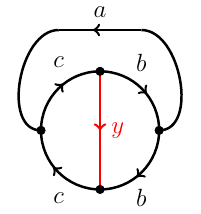},
 \end{align*}
 viewed as maps from $H_{\mu}(v)$ (resp. $H_{\mu}(u)$) to the ground field $\K$ (in the sense that for a vector, e.g., in $H_{\mu}(v)$, its decomposition into pure tensors provides morphism labels for the black vertices in the graph). The labeling in the picture is fixed by setting ${\mu}(A) = a$, ${\mu}(B) = b$ and ${\mu}(C) = c$. Here and in the following, we use red coloring to highlight edges coming from defect lines, which are treated in the same manner as black edges in the graphical calculus. We are ready to start evaluating the invariant: With respect to the graph skeleton $P$, we can compute (leaving out sums over dual bases from the notation as before)
 \begin{align*}
 |\mathbb{T}_{x, y}^s|_{\mathcal{C}} &= \dim(\mathcal{C})^{-|\mathbb{T}_{x, y}^s \setminus P|}\sum_{\mu\colon \mathrm{Fac}(P) \rightarrow   \mathcal{O}(\mathcal{C})}\dim(\mu)\,|\mu| \\
 &= \dim(\mathcal{C})^{-1}\sum_{a, b, c \in \mathcal{O}(\mathcal{C})} \dim(a) \dim(b) \dim(c) \adjincludegraphics[valign=c, scale = 0.75]{images/lemTorEv1.pdf} \, \adjincludegraphics[valign=c, scale = 0.75]{images/lemTorEv2.pdf} \\
 &\overset{(\text{Lemma~\ref{lemDisconn}})}{=} \dim(\mathcal{C})^{-1}\sum_{a, b, c \in \mathcal{O}(\mathcal{C})} \dim(a) \dim(b) \dim(c) \, \, \adjincludegraphics[valign=c, scale = 0.75]{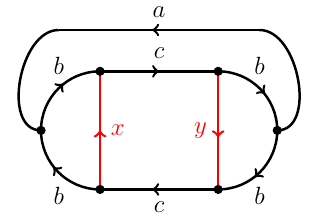} \\
 &\overset{(\text{Lemma~\ref{lemDecId}})}{=} \dim(\mathcal{C})^{-1}\sum_{ b, c \in \mathcal{O}(\mathcal{C})} \dim(b) \dim(c) \adjincludegraphics[valign=c, scale = 0.75]{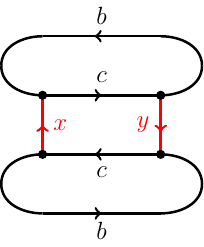} \\
 &\overset{(\text{Lemma~\ref{lemQuad}})}{=} \dim(\mathcal{C})^{-1}\sum_{b \in \mathcal{O}(\mathcal{C})} \dim(b)\adjincludegraphics[valign=c, scale = 0.75]{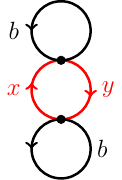} \overset{(\text{Lemma~\ref{lemBub}})}{=} \dim(\mathcal{C})^{-1}\dim(\mathcal{C}) \adjincludegraphics[valign=c, scale = 0.75]{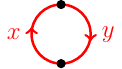}.
 \end{align*}
 The vertices of the final diagram are implicitly labeled by a pair of dual bases for the vector spaces $\Hom_{\mathcal{C}}(\mathbbm{1}, x \otimes y^*)$ and $\Hom_{\mathcal{C}}(x \otimes y^*, \mathbbm{1})$, over which we sum. By Lemma~\ref{lemTrace}, we have
 \begin{align}
 \label{eqTor2simp}
 |\mathbb{T}_{x, y}^s|_{\mathcal{C}}
 = \adjincludegraphics[valign=c, scale = 0.75]{images/lemTorEv6.pdf} = \tr (\id_{\Hom_{\mathcal{C}}(\mathbbm{1}, x \otimes y^*)} ) = \dim_{\K} (\Hom_{\mathcal{C}}(x,y) ).
 \end{align}
 If $x$ and $y$ are both simple, the final expression is equal to $\delta_{[x], [y]}$ by Schur's lemma.
\end{proof}

\begin{Remark} \label{remTorus}
Computing the Turaev--Viro invariant of cobordisms $\varnothing \rightarrow \varnothing$ by using its characterizing properties instead of directly working with the explicit formula is often much less involved. Once we have proven that Theorem~\ref{thmMainTV} holds, we can argue as follows in the situation of Lemma~\ref{lemTorxy}: Let $\{\alpha_i\}_{i = 1}^k$ be a basis for $\Hom_{\mathcal{C}}(x \otimes y^*, \mathbbm{1})$. Notice that applying the move~$\Delta_3^{\alpha_i}$ to~$\mathbb{T}_{x,y}^s$ produces a 3-ball with the graph from \eqref{eqTor2simp} on its boundary. By~\ref{P3} and~\ref{P4}, we can directly compute
 \begin{align*}
 |\mathbb{T}_{x,y}^s|_{\mathcal{C}} \overset{\rm \ref{P3}}{=} \sum_{i = 1}^k|\Delta_3^{\alpha_i} (\mathbb{T}_{x,y}^s )|_{\mathcal{C}} \overset{\rm \ref{P4}}{=} \adjincludegraphics[valign=c, scale = 0.75]{images/lemTorEv6.pdf}.
 \end{align*}
\end{Remark}
We thank Julian Farnsteiner for contributing the following lemma.
\begin{Lemma}
\label{lemProp2}
 Let $x \in \mathcal{O}(\mathcal{C})$ be a simple object and let $B, \mathbb{T}_x^s\colon \varnothing \rightarrow A_1$ be the two cobordisms in $\Cob_3^{\mathcal{C}}$ from the empty manifold to the annulus $A_1 = \Sp^1 \times I$, defined in pictures by
 \begin{align*}
 B = \adjincludegraphics[valign=c, scale = 0.7]{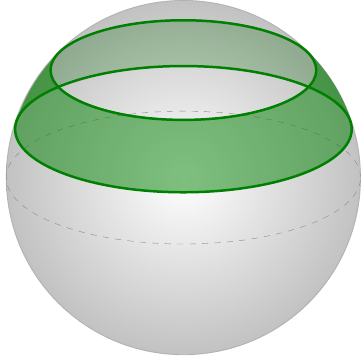}\,, \qquad \mathbb{T}_x^s = \adjincludegraphics[valign=c, scale = 0.7]{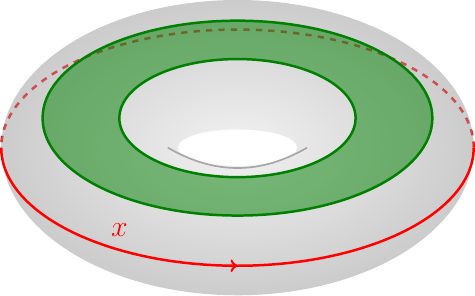}\, ,
 \end{align*}
 where $x$ labels a non-contractible defect loop on the boundary of the torus and the area shaded in green marks the gluing boundary homeomorphic to $A_1$. Then, the following claims hold:
 \begin{enumerate}[$(i)$]\itemsep=0pt
 \item The cobordism $\mathbb{T}_x^s$ induces an isomorphism between the vector space assigned to the annulus by the Turaev--Viro TQFT and the Grothendieck algebra, that is $|A_1|_{\mathcal{C}} \cong K_0(\mathcal{C}) \otimes_{\mathbb{Z}} \K$.
 \item Under the identification of (i), the homomorphism $|\mathbb{T}_x^s|_{\mathcal{C}}\colon \K \rightarrow |A_1|_{\mathcal{C}}$ maps $1_{\K}$ to $[x]$.
 \item Under the same identification, the linear map $|B|_{\mathcal{C}}\colon \K \mapsto |A_1|_{\mathcal{C}}$ sends $1_{\K}$ to the element $\sum_{x \in \mathcal{O}(\mathcal{C})}\dim(x) [x]$.

 \end{enumerate}
\end{Lemma}
\begin{proof}
 We start by defining a map $\varphi\colon K_0(\mathcal{C}) \otimes_{\mathbb{Z}} \K \rightarrow |A_1|_{\mathcal{C}}$ and establishing its injectivity. We then proceed by showing statements (ii) and (iii) and conclude by proving the surjectivity of $\varphi$ and hence (i).

 Let us introduce the notation $\tau_x = |\mathbb{T}_x^s|_{\mathcal{C}}(1_{\K})$ for the vector in $|A_1|_{\calc}$ defined by the cobor\-dism~$\mathbb{T}_x^s$ for any simple $x \in \mathcal{O}(\mathcal{C})$. The cobordism $\overline{\mathbb{T}_x^s}\colon A_1 \rightarrow \varnothing $ obtained from $\mathbb{T}_x^s\colon \varnothing \rightarrow A_1$ by inverting the orientation of the torus and interchanging source and target is sent under the functor $|\cdot|_{\mathcal{C}}$ to a linear map $\tau_x^*:= \left|\overline{\mathbb{T}_x^s}\right|_{\mathcal{C}}\colon |A_1|_{\mathcal{C}} \rightarrow \K$. Hence, for any simple object $y \in \mathcal{O}(\mathcal{C})$, we have, composing the cobordisms, the identity
 \begin{align*}
 \tau_y^*(\tau_x) = \left| \mathbb{T}_x^s \cup_{A_1} \overline{\mathbb{T}_y^s}\right|_{\mathcal{C}}(1_{\K}) = |\mathbb{T}_{x,y}^s|_{\mathcal{C}} = \delta_{x,y},
 \end{align*}
 where we used Lemma~\ref{lemTorxy} in the last step. This implies linear independence of $\{\tau_x\}_{x \in \mathcal{O}(\mathcal{C})}$. Thus, the map $\varphi\colon K_0(\mathcal{C}) \otimes_{\mathbb{Z}} \K \rightarrow |A_1|_{\mathcal{C}}, [x] \mapsto \tau_x$ is injective, from which (ii) follows. Before showing surjectivity, we prove the third claim (iii).

 We equip the cobordism $\mathbb{T}_x^s$ with a graph skeleton $P$, defined pictorially as follows:
 \begin{align*}
 P = \adjincludegraphics[valign=c, scale = 0.75]{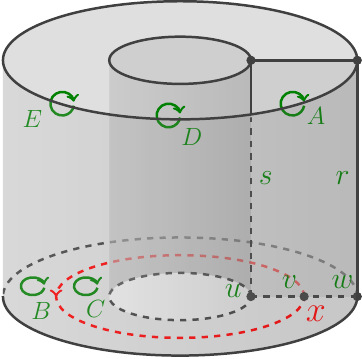}.
 \end{align*}
 The skeleton $P$ is drawn in a cylindrical shape to simplify the graphical presentation. It has five faces $A$, $B$, $C$, $D$ and $E$ with specified orientation, ten rims (including the defect line drawn on the bottom in red) and five nodes. All faces apart from $A$ are contained in the free boundary~$\partial_f\mathbb{T}_x^s$, and the face $D$ wraps around the hole of the torus. We emphasize that $P$ (as any skeleton) does not have faces in the gluing boundary $\partial_g\mathbb{T}_x^s \cong A_1$, but $P$ induces a skeleton~$G_P$ for~$A_1$, consisting of the two marked points on the upper lid together with the three upper edges. For the other three nodes $u$, $v$ and $w$, at which the evaluation to determine $|\mathbb{T}_x^s|_{\mathcal{C}}$ takes place, we can easily read off the link graphs from the graphical presentation of $P$. They are given by
 \begin{align*}
 \Gamma_{\mu}(u) = \adjincludegraphics[valign=c, scale = 0.75]{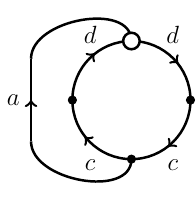}\,,\qquad \Gamma_{\mu}(v) = \adjincludegraphics[valign=c, scale = 0.75]{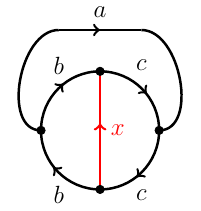}\!\!,\qquad \Gamma_{\mu}(w) = \adjincludegraphics[valign=c, scale = 0.75]{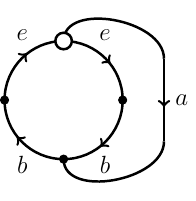}
 \end{align*}
 for an arbitrary coloring $\mu\colon \mathrm{Fac}(P) \rightarrow \mathcal{O}(\mathcal{C})$, where $\mu(A) = a, \mu(B) = b$, etc. Here, we have introduced the convention to represent link graph vertices which come from rims adjacent to the gluing boundary by hollow nodes, and the remaining vertices by solid ones.

 Choose a rim basis $\alpha$ for the rim $r$ and represent $*_r = \alpha \otimes \alpha^*$ with the sum convention introduced in Section~\ref{subSecTV}. Here, by convention, $\alpha_i \in H_{\mu}(e_r)$ for $e_r$ the half-rim of $r$ at the vertex~$w$. Similarly, for the rim $s$ we can write $*_s = \beta \otimes \beta^*$, and these two vectors together determine the vector $*^{\partial} = *_r \otimes *_s$. Now, using Lemma~\ref{lemDisconn}, we can compute the vector $|\mu| \in H_{\mu_{\partial}}^{\partial}(\mathbb{T}_x^s)$, which still depends on the boundary coloring $\mu_{\partial}$:
\begin{align*}
 |\mu| &= |a, \dots, e| = \bigl(\Gamma_{\mu}(u) \otimes \Gamma_{\mu}(v) \otimes \Gamma_{\mu}(w) \otimes \id_{H_{\mu}^{\partial}(\mathbb{T}_x^s)}\bigr)\bigl(*^0 \otimes *^{\partial}\bigr) \\
 &= \adjincludegraphics[valign=c, scale = 0.75]{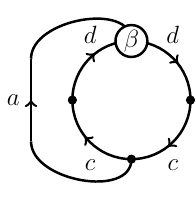} \, \, \adjincludegraphics[valign=c, scale = 0.75]{images/lemTunLinkv.pdf} \, \, \adjincludegraphics[valign=c, scale = 0.75]{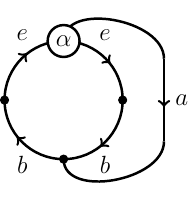} \alpha^* \otimes \beta^* \\
 &= \adjincludegraphics[valign=c, scale = 0.75]{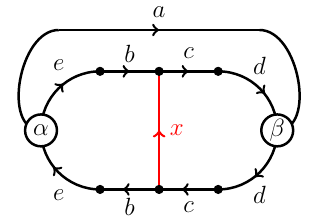} \, \, \alpha^* \otimes \beta^*.
\end{align*}
Note that we have implied sums over dual bases for the solid black vertices as well as for the hollow ones.

 Recall that $G_P$ denotes the skeleton of the annulus $A_1$ with respect to the graph skeleton~$P$. The graph underlying the skeleton $G_P$ has two vertices and three edges, coming from the three faces~$A$,~$D$ and $E$ of $P$ adjacent to the gluing boundary. Notice that for any face $f \in \Fac(P)$ the Euler characteristic $\chi(f)$ is equal to~$1$. We will now determine the $\K$-linear map $|\mathbb{T}_x^s, \varnothing, G_P|_{\mathcal{C}}\colon |\varnothing, \varnothing|_{\calc} \rightarrow |A_1, G_P|_{\calc}$ by explicitly computing
 \begin{align}
 \label{eqGenTunPrep}
 |\mathbb{T}_x^s, &\varnothing, G_P|_{\mathcal{C}} = \bigoplus_{\mu_{A_1} \colon E(G_P) \rightarrow \mathcal{O}(\mathcal{C})}\!\!\frac{\dim(\mathcal{C})^{|A_1\setminus G_P|}}{\dim^{\tr}(\mu_{A_1})}\dim(\mathcal{C})^{-|\mathbb{T}_x^s\setminus P|} \!\!\! \!\!\!\!\!\!\!\!\! \sum_{\mu_0\colon \mathrm{Fac}_{0}(P) \rightarrow \mathcal{O}(\mathcal{C})} \!\!\!\!\!\! \!\!\! \!\!\!\dim(\mu_0 \sqcup \mu_{A_1})|\mu_0 \sqcup \mu_{A_1}| \nonumber \\
 &= \bigoplus_{a, d, e \in \mathcal{O}(\mathcal{C})} \frac{\dim(\calc)\dim(\calc)^{-1}}{\dim(a) \dim(d) \dim(e)} \sum_{b, c \in \calo(\calc)} \dim(a)\dim(b)\dim(c)\dim(d)\dim(e) |a, \dots, e| \nonumber \\
 &= \bigoplus_{a, d, e \in \mathcal{O}(\mathcal{C})}\sum_{b, c \in \calo(\calc)} \dim(b) \dim(c) \adjincludegraphics[valign=c, scale = 0.75]{images/lemTunEv1.pdf} \alpha^* \otimes \beta^* \nonumber \\
 &= \bigoplus_{a, d, e \in \mathcal{O}(\mathcal{C})}\sum_{b, c \in \calo(\calc)} \dim(b) \dim(c) \, \delta_{c,d} \, \delta_{b,e} \adjincludegraphics[valign=c, scale = 0.75]{images/lemTunEv1.pdf}\alpha^* \otimes \beta^* \nonumber \\
 &\overset{\text{\cite[Lemma 12.3]{TV}}}{=} \bigoplus_{a, d, e \in \mathcal{O}(\mathcal{C})} \adjincludegraphics[valign=c, scale = 0.75]{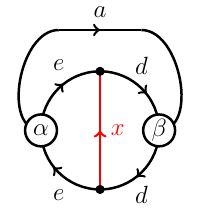}\alpha^* \otimes \beta^* = \bigoplus_{a, d, e \in \mathcal{O}(\mathcal{C})} \adjincludegraphics[valign=c, scale = 0.75]{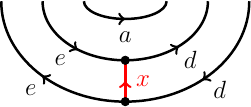},
 \end{align}
 where we have used a lemma from \cite{TV} in the second to last step and evaluated the sum over $\alpha$ and $\beta$ in the last step with the help of Lemma~\ref{lemDisconn}. In the above equation, we view $|\mathbb{T}_x^s, \varnothing, G_P|_{\mathcal{C}}$ as an element in $|A_1, G_P|_{\calc} \subset H(A_1, G_P) = \smash{\bigoplus_{\mu_{A_1}}H_{\mu_{A_1}}(e_r') \otimes H_{\mu_{A_1}}(e_s')}$, where $e_r'$ and $e_s'$ are the two half-rims that are adjacent to the gluing boundary, associated with the rims $r$ and $s$, respectively (cf. Section~\ref{subSecTV}). The sum over simple objects weighted by dimensions gives
 \begin{align*}
 & \sum_{x \in \mathcal{O}(\mathcal{C})}\dim(x) |\mathbb{T}_x^s, \varnothing, G_P|_{\mathcal{C}} \\
& \qquad \overset{\text{Lemma~\ref{lemDecId}}}{=} \bigoplus_{a, d, e \in \mathcal{O}(\mathcal{C})}\adjincludegraphics[valign=c, scale = 0.75]{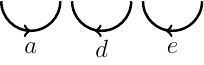} = \bigoplus_{a, d, e \in \mathcal{O}(\mathcal{C})} \widetilde{\mathrm{coev}}_a \otimes \mathrm{coev}_d \otimes \mathrm{coev}_e.
 \end{align*}
 On the other hand, we equip $B$ with a graph skeleton $Q$ in order to compare this with the expression just obtained. We define the skeleton $Q$ graphically by the picture
 \begin{align*}
 Q = \adjincludegraphics[valign=c, scale = 0.75]{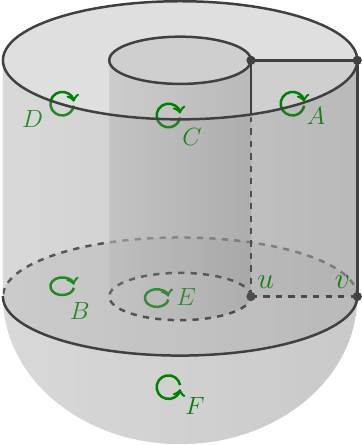}.
 \end{align*}
 Compared to $P$, there are two completely new faces labeled by $E$ and $F$, and the free boundary is covered by the faces $C, D$ and $F$. Again, all of the faces have Euler characteristic 1. Notice that both $P$ and $Q$ induce the same skeleton for the annulus $A_1$, that is $G_P = G_Q$. For a chosen coloring $\mu\colon \Fac(Q) \rightarrow \calo(\calc)$, the two link graphs $\Gamma_{\mu}(u)$ and $\Gamma_{\mu}(v)$ for the vertices $u,v$ are
 \begin{align*}
 \Gamma_{\mu}(u) = \adjincludegraphics[valign=c, scale = 0.75]{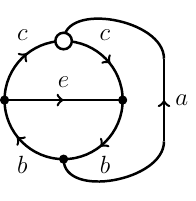}\!, \qquad \Gamma_{\mu}(v) = \adjincludegraphics[valign=c, scale = 0.75]{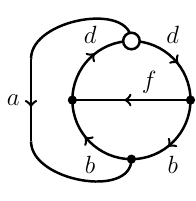}\,.
 \end{align*}
 Just as before, we pick rim bases $\alpha$ and $\beta$ and evaluate $|B, \varnothing, G_Q|_{\mathcal{C}}$:
 \begin{align*}
 |B, &\varnothing, G_Q|_{\mathcal{C}} = \bigoplus_{a, c, d \in \calo(\calc)} \frac{\dim(\calc)\dim(\calc)^{-2}}{\dim(a)\dim(c)\dim(d)}\sum_{b, e, f \in \calo(\calc)}\dim(a) \cdots \dim(f)\,|a, \dots, f| \\
 &= \bigoplus_{a, c, d \in \calo(\calc)}\sum_{b, e, f \in \calo(\calc)} \!\frac{\dim(b) \dim(e) \dim(f)}{\dim(\calc)} \, \adjincludegraphics[valign=c, scale = 0.75]{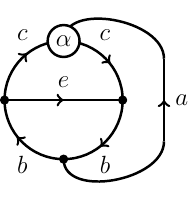} \adjincludegraphics[valign=c, scale = 0.75]{images/lemTunSkelQVFill.pdf} \, \, \alpha^* \otimes \beta^* \\
 &\overset{\text{Lemma~\ref{lemDisconn}}}{=} \bigoplus_{a, c, d \in \calo(\calc)}\sum_{b, e, f \in \calo(\calc)} \!\frac{\dim(b) \dim(e) \dim(f)}{\dim(\calc)}\! \adjincludegraphics[valign=c, scale = 0.75]{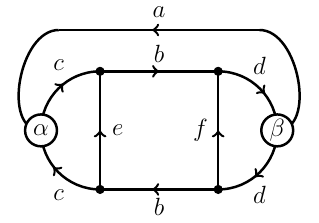}\alpha^* \otimes \beta^* \\
 &\overset{\text{Lemma~\ref{lemDecId}}}{=} \sum_{b \in \calo(\calc)} \frac{\dim(b)}{\dim(\calc)}\bigoplus_{a,c,d \in \calo(\calc)} \adjincludegraphics[valign=c, scale = 0.75]{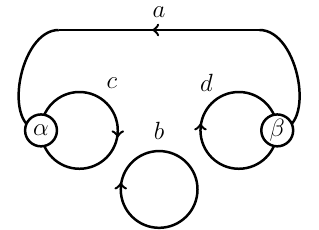} \alpha^* \otimes \beta^*\\
 &\overset{\text{Lemma~\ref{lemDisconn}}}{=} \sum_{b \in \calo(\calc)} \frac{\dim(b)^2}{\dim(\calc)} \, \, \bigoplus_{a, c, d \in \calo(\calc)} \adjincludegraphics[valign=c, scale = 0.75]{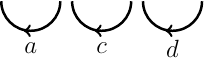}= \bigoplus_{a, c, d \in \calo(\calc)} \mathrm{coev}_a \otimes \mathrm{coev}_c \otimes \mathrm{coev}_d.
 \end{align*}
 By a relabeling, this is identical to the expression we have obtained before. Now, it is clear that in the limit over graph skeleta, the expression just obtained and the one from before still coincide. This yields (iii), as \smash{$\varphi([x]) = \tau_x = |\overline{\mathbb{T}_x^s} |_{\mathcal{C}}(1_{\K})$} and the linear map $\varphi$ is injective.

 What remains is to demonstrate the surjectivity of $\varphi$. For this, we will proceed as follows. We will equip the identity cobordism $A_1 \times I\colon A_1 \rightarrow A_1$ with a graph skeleton $O$ so that the skeleton $G_O$ of $A_1$ coincides with the skeleton $G_P$, and is the same for the input and output surface. We will similarly equip the cobordism $\mathbb{T}_x^s \sqcup \overline{\mathbb{T}_x^s}\colon A_1 \rightarrow A_1$ with a suitable skeleton for any simple object $x \in \calo(\calc)$ so that the induced skeleton of $A_1$ is $G_O$, and proceed by showing
 \begin{align}
 \label{eqProp2last}
 |A_1 \times I, G_O, G_O|_{\calc} = \sum_{x \in \calo(\calc)} |\mathbb{T}_x^s \sqcup \overline{\mathbb{T}_x^s}, G_O, G_O|_{\calc}.
 \end{align}
 Taking note of the fact that in the limit over skeleta we have $\id_{|A_1|_{\calc}} = \lim_G|A_1 \times I, G, G|_{\calc}$, it then follows that any vector $v \in |A_1|_{\calc}$ can be expressed as the image of $\sum_{x \in \calo(\calc)}\tau_x^*(v)[x]$ under the map $\varphi$. This shows that the linear map $\varphi$ is surjective and hence defines a linear isomorphism $|A_1|_{\mathcal{C}} \cong K_0(\mathcal{C}) \otimes_{\mathbb{Z}} \K$. Thus, to prove (i), it remains to show \eqref{eqProp2last}. We define the graph skeleton $O$ by the following image:
 \begin{align*}
 O = \adjincludegraphics[valign=c, scale = 0.75]{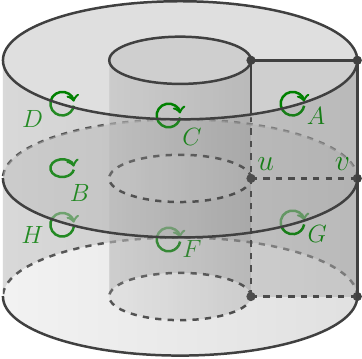}.
 \end{align*}
 The bottom and top lids are both contained in the gluing boundary. There are seven faces in total, labeled $A, \dots, H$. Given a coloring $\mu$ of $O$, the link graphs $\Gamma_{\mu}(u)$ and $\Gamma_{\mu}(v)$ for the two vertices $u$ and $v$ are extracted from $O$ as usual. Then, $|A_1 \times I, G_O, G_O|_{\calc}$ can be evaluated as
 \begin{align*}
 |A_1 \times I, &G_O, G_O|_{\calc} = \bigoplus_{a, c, d, f, g, h \in \calo(\calc)} \frac{\dim(\calc)\dim(\calc)^{-2}}{\dim(a)\dim(c) \dim(d)} \sum_{b \in \calo(\calc)}\! \dim(a) \dim(b) \dim(c) \dim(d) \\
 &\times \dim(f) \dim(g) \dim(h) \, \, \adjincludegraphics[valign=c, scale = 0.75]{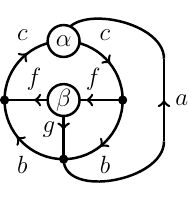} \, \adjincludegraphics[valign=c, scale = 0.75]{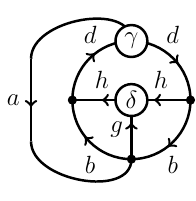} \, \, \alpha^* \otimes \beta^* \otimes \gamma^* \otimes \delta^* \\
 &\overset{\text{Lemma~\ref{lemDisconn}}}{=} \frac{1}{\dim(\calc)}\bigoplus_{a, c, d, f, g, h \in \calo(\calc)} \sum_{b \in \calo(\calc)} \dim(b) \dim(f) \dim(g) \dim(h) \\
 & \times \adjincludegraphics[valign=c, scale = 0.75]{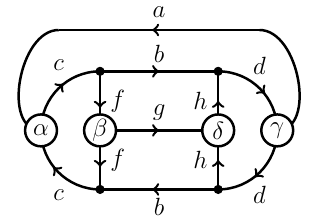} \alpha^* \otimes \beta^* \otimes \gamma^* \otimes \delta^*.
 \end{align*}
On the other hand, we define a graph skeleton $P'$ for the cobordism $\mathbb{T}_y^s \sqcup \overline{\mathbb{T}_y^s}$ by combining two copies of $P$ so that $G_{P'} = G_O$. We calculate, reusing the earlier computations for $P$ and a~suitable choice for the orientations of the skeletal faces:
 \begin{align*}
 \sum_{y \in \calo(\calc)} &|\mathbb{T}_y^s \sqcup \overline{\mathbb{T}_y^s}, G_{P'}, G_{P'}|_{\calc} = \frac{1}{\dim(\calc)} \bigoplus_{a, b, c, d, e, f \in \calo(\calc)} \sum_{y \in \calo(\calc)} \dim(b) \dim(c) \dim(f) \\
 &\times \alpha^* \otimes \beta^* \otimes \gamma^* \otimes \delta^* \,\adjincludegraphics[valign=c, scale = 0.75]{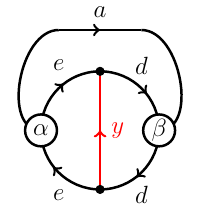} \adjincludegraphics[valign=c, scale = 0.75]{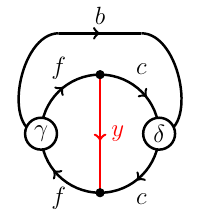}\\
 &\overset{\substack{\text{\cite[Lemma 12.3]{TV}} \\ \text{+ sphericity}}}{=} \frac{1}{\dim(\calc)} \bigoplus_{a, b, c, d, e, f \in \calo(\calc)} \sum_{y \in \calo(\calc)} \dim(y)\dim(b) \dim(c) \dim(f) \cdot \\
 &\times \alpha^* \otimes \beta^* \otimes \gamma^* \otimes \delta^* \adjincludegraphics[valign=c, scale = 0.75]{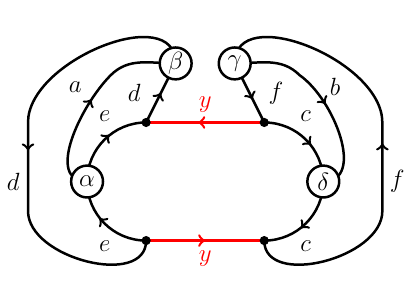} \\
 &\overset{\text{Lemma~\ref{lemQuad}}}{=} \frac{1}{\dim(\calc)} \bigoplus_{a, b, c, d, e, f \in \calo(\calc)} \sum_{y \in \calo(\calc)} \dim(y)\dim(b) \dim(c) \dim(f) \\
 &\times \alpha^* \otimes \beta^* \otimes \gamma^* \otimes \delta^* \adjincludegraphics[valign=c, scale = 0.75]{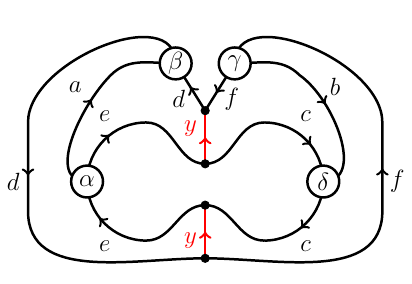}.
 \end{align*}
Using sphericity of $\calc$ (that is, pulling strands to the other side of the diagram), it is possible to transform the diagram in the last expression into a diagram isotopic to the one we computed in the calculation for $|A_1 \times I, G_O, G_O|_{\calc}$ earlier, so that the labeling matches up to renaming of the summation indices. Since $G_{P'} = G_O$, this shows \eqref{eqProp2last} and finishes the proof of the proposition.
\end{proof}

\begin{Corollary}\label{corAnn}
The Turaev--Viro invariant of the manifold $A_1 \times \Sp^1$ is given by $\bigl|{A_1 \times \Sp^1}\bigr|_{\calc} = |{\calo}(\calc)|$, the number of isomorphism classes of simples in $\calc$.
\end{Corollary}
\begin{proof}
 This follows from the fact that $A_1 \times \Sp^1 = \tr(A_1 \times I)$ is the trace of the identity cobordism on $A_1$, so it is sent to $\dim_{\K}(K_0(\mathcal{C}) \otimes_{\mathbb{Z}} \K)$ by Lemma~\ref{lemProp2} under the application of $|\cdot|_{\calc}$.
\end{proof}
\begin{Remark}
 From the above corollary it becomes apparent that the presence of free boundaries already breaks the Morita invariance of the state sum invariant. For any \textit{closed} 3-mani\-fold~$M$, the invariant only depends on the Drinfeld center $\mathcal{Z}(\calc)$, since $|M|_{\calc} = \mathrm{RT}_{\mathcal{Z}(\calc)}(M)$, where $\mathrm{RT}_{\mathcal{Z}(\calc)}$ denotes the Reshitikhin--Turaev invariant. The spherical fusion categories $\calc = \Rep_{\C}(\mathfrak{S}_3)$ and $\calc' = \vect_{\mathfrak{S}_3}$ have equivalent centers $\mathcal{Z}(\calc) \simeq D(\mathfrak{S}_3)\mathrm{-mod}\simeq \mathcal{Z}(\calc')$, but not the same number of isomorphism classes of simple objects, so $\bigl|A_1 \times \Sp^1\bigr|_{\calc}$ and $\bigl|A_1 \times \Sp^1\bigr|_{\calc'}$ are different.
\end{Remark}
\begin{Remark}
 In the proof of Lemma~\ref{lemProp2}, we have derived with \eqref{eqProp2last} the behavior of the Turaev--Viro invariant under the transformation of cutting a cylinder over an annulus $A$ (as part of a defect manifold) in two, transforming only the free boundary. Explicitly, we have
 \begin{align}
 \label{eqAnnCut}
 |A_1 \times I|_{\calc} = \sum_{x \in \calo(\calc)} \bigl|\mathbb{T}_x^s \sqcup \overline{\mathbb{T}_x^s}\bigr|_{\calc}.
 \end{align}
 This should be seen as an analog of the tube-capping move $\Delta_3$ in the absence of defect lines, where the base of the cylinder is an annulus instead of a disk.
\end{Remark}

Using Lemma~\ref{lemProp2}, we can now easily prove that the Turaev--Viro TQFT satisfies~\ref{P2}.
\begin{Proposition}
 \label{propProp2}
 Let $M \in \Cob_3^{\mathcal{C}}(\varnothing, \varnothing)$ and $x \in \mathcal{O}(\mathcal{C})$ a simple object. Then, the relation
 \begin{align*}
 |M|_{\mathcal{C}} = \sum_{x \in \mathcal{O}(\mathcal{C})}\dim(x)|\Delta_2^x(\iota)(M)|_{\mathcal{C}}
 \end{align*}
 holds for any tunnel $\iota\colon C \hookrightarrow M$ with respect to the tunnel move $\Delta_2^x(\iota)$.
\end{Proposition}
\begin{proof}
To isolate the local situation, we fix a second, slightly larger tunnel $\iota'\colon C \hookrightarrow M$ (i.e., $\mathrm{im}(\iota)$ is contained in the image of $\iota'$) so that the embedding agrees on the cylinder axis with the embedding $\iota$. Recall the cobordisms $B$ and $\mathbb{T}_x^s$ from Lemma~\ref{lemProp2}, where $B$ is a ball and~$\mathbb{T}_x^s$ a solid torus decorated with a defect line, both having an embedded annulus $A_1$ as their gluing boundary. Identifying $A_1$ with the lateral surface $\iota'(\Sp^1 \times I)$ and cutting along it, we may decompose the cobordism $M$ as $M = M' \cup_{\iota'(\Sp^1 \times I)}B$ and the cobordism $\Delta_2^x(\iota)(M)$ as
\[
\Delta_2^x(\iota)(M) = M' \cup_{\iota' (\Sp^1 \times I )} \mathbb{T}_x^s.
\] From assertion (ii) and (iii) of Lemma~\ref{lemProp2}, it follows directly that
\[
|B|_{\mathcal{C}} = \sum_{x \in \mathcal{O}(\mathcal{C})}\dim(x) |\mathbb{T}_x^s|_{\mathcal{C}}.
\] Using this relation, we see that
\begin{align*}
 |M|_{\mathcal{C}} &= |M' \cup_{\iota'(\Sp^1 \times I)}B|_{\mathcal{C}} = |M'|_{\mathcal{C}} \circ |B|_{\mathcal{C}} = \sum_{x \in \mathcal{O}(\mathcal{C})}\dim(x)|M'|_{\mathcal{C}} \circ |\mathbb{T}_x^s|_{\mathcal{C}} \\
 &= \sum_{x \in \mathcal{O}(\mathcal{C})}\dim(x)|M'\cup_{\iota'(\Sp^1 \times I)}\mathbb{T}_x^s|_{\mathcal{C}} = \sum_{x \in \mathcal{O}(\mathcal{C})}\dim(x) |\Delta_2^x(\iota)(M)|_{\mathcal{C}},
\end{align*}
which is the identity asserted in the proposition.
\end{proof}

We prove property \ref{P3} for the Turaev--Viro TQFT.
\begin{Proposition}
 \label{propProp3TV}
 Let $M \in \Cob_3(\varnothing, \varnothing)$ be a defect cobordism to which the tube-capping move can be locally applied with respect to a suitably embedded oriented cylinder. For $X_1, \dots, X_n \in \mathcal{C}$ the objects labeling the defect edges on the lateral surface of the cylinder and $\{\alpha_i\}_{i = 1}^k$ a basis of the vector space \smash{$\Hom_{\calc}\bigl(X_1^{\varepsilon(1)} \otimes \dots \otimes X_n^{\varepsilon(n)}, \mathbbm{1}\bigr) $}, the relation \smash{$|M|_{\calc} = \sum_{i = 1}^k|\Delta_3^{\alpha_i}(M)|_{\calc}$} holds.
\end{Proposition}
\begin{proof}
 The Turaev--Viro theory is invariant under swapping the orientation of a labeled defect edge while simultaneously replacing the labeling object with its dual. Thus, without loss of generality, we will assume that the parallel defect strands on the cylinder are oriented in the same direction and that the vector space referred to in the proposition is given by $\Hom_{\calc}\left(X_1 \otimes \dots \otimes X_n, \mathbbm{1}\right)$. Let us denote by $X = \{X_i\}_{i = 1}^n$ the family of objects labeling the cylinder and by $\alpha = \{\alpha_i\}_{i = 1}^k$ the chosen basis.

 Similar to the proof of Proposition~\ref{propProp2}, it suffices by functoriality and tensoriality of $|\cdot|_{\calc}$ to compare the situation locally. Let us denote by $\Sigma_X \cong \mathbb{D}^2$ the disk as a $\calc$-colored surface with point insertions on its boundary, all of which are positively oriented and colored by the individual objects in $X$. Cutting at the bottom and top of the embedded cylinder with respect to which the tube-capping move is applied, we obtain the two cobordisms $\Sigma_X \times I, D_X^{\alpha_i}\colon \Sigma_X \rightarrow \Sigma_X$, where~\smash{$D_X^{\alpha_i}$} is the composition \smash{$\overline{E}^{\alpha_i^*} \sqcup E^{\alpha_i}$} of the two caps introduced at the beginning of Section~\ref{subSecChar}.
 As the cobordism $\Sigma_X \times I$ represents the identity on $\Sigma_X$ in $\Cob_3^{\calc}$, it is mapped by the TQFT $|\cdot |_{\calc}$ to the identity homomorphism on $|\Sigma_X|_{\calc}$.\footnote{This state space has been explicitly computed in~\cite{F} with the result $|\Sigma_X|_{\calc} \cong \Hom_{\calc} (\mathbbm{1}, X_1 \otimes \dots \otimes X_n )$. We do not need this result here, but we will verify this explicitly later for Dijkgraaf--Witten theory in \eqref{eqSS}.}

 We turn to the computation of $|D_X^{\alpha_i}|_{\calc}$. As usual, this works by choosing a graph skeleton $P_i$ for the cobordism for each index $i \in \{1, \dots, k\}$. Technically, this graph skeleton would have to include at least one vertex of \smash{$P_i^{(1)}$} on both of the circular edges in the two boundary components of the skeleton, each adjacent to an edge not in $\partial P_i$, which we omit for ease of presentation (its inclusion would result in a minor modification of the actual computation).

 We therefore define the graph skeleton $P_i$ graphically as the following decorated 2-polyhedron,
 \begin{align*}
 P_i = \adjincludegraphics[valign=c, scale = 0.75]{images/tubeSkel.pdf},
 \end{align*}
 drawn here in the special case $n = 4$. This means that we take $P_i$ as the disjoint union of two disks covering the free boundary of $D_X^{\alpha_i}$, decorated with the defect graph embedded in $\partial_f D_X^{\alpha_i}$. The faces $A_1, \dots, A_n, B_1, \dots, B_n$ are the connected components of $\partial_f D_X^{\alpha_i}$ minus the defect graph, and are labeled by $2n$ simple objects once a coloring has been chosen. There are two coupons, labeled by $\alpha_i$ and $\alpha_i^*$, respectively. The skeleton $G_{P_i}$ of $\Sigma_X$ is given by the upper boundary circle together with the insertions of the objects in $X$, and the lower circle corresponds to $G_{P_i}^{\mathrm{op}}$. Note that $G_{P_i}$ does not depend on the index $i$, so we will simply denote it by $G_P$. We compute the sum of invariants over the collection of cobordisms labeled by the elements of $\alpha$, where we allow ourselves to draw the vertices labeled by $\alpha$ in circular shape:
 \begin{align*}
 \sum_{i = 1}^k|D_X^{\alpha_i}, &G_{P}, G_{P}|_{\calc} = \bigoplus_{a_1, \dots, a_n, b_1, \dots, b_n \in \calo(\calc)}\frac{\dim(\calc)\dim(\calc)^{-2}}{\dim(b_1) \cdots \dim(b_n) }\dim(a_1) \cdots \dim(b_n) \\
 & \times \beta_1^* \otimes \dots \otimes \beta_n^* \otimes \gamma_1^* \otimes \dots \otimes \gamma_n^* \, \,\adjincludegraphics[valign=c, scale = 0.75]{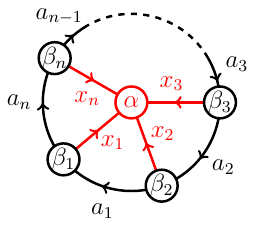} \, \adjincludegraphics[valign=c, scale = 0.75]{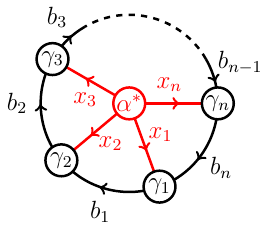} \\[1.5mm]
 &\overset{\substack{\text{Lemma~\ref{lemDisconn} } \\ + \calc \text{ spherical}}}{=} \frac{1}{\dim(\calc)} \bigoplus_{a_1, \dots, a_n, b_1, \dots, b_n \in \calo(\calc)}\dim(a_1) \cdots \dim(a_n) \, \beta_1^* \otimes \dots \otimes \beta_n^* \\[1.5mm]
 & \otimes \gamma_1^* \otimes \dots \otimes \gamma_n^* \, \,\adjincludegraphics[valign=c, scale = 0.75]{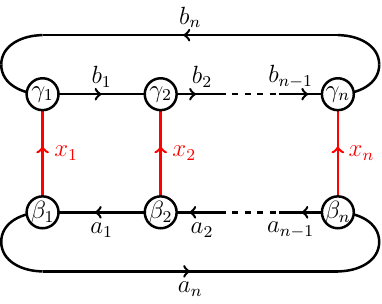} \\[1.5mm]
 &\overset{\text{Lemma~\ref{lemDecId}}}{=} \frac{1}{\dim(\calc)} \bigoplus_{a_1, \dots, a_n, b_1, \dots, b_n \in \calo(\calc)} \sum_{s \in \calo(\calc)} \dim(a_1) \cdots \dim(a_n) \dim(s) \\[1.5mm]
 & \times \beta_1^* \otimes \dots \otimes \beta_n^* \otimes \gamma_1^* \otimes \dots \otimes \gamma_n^* \, \,\adjincludegraphics[valign=c, scale = 0.75]{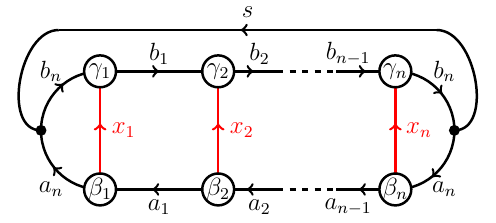} \\[1mm]
 &\overset{\text{Lemma~\ref{lemDisconn}}}{=} \frac{1}{\dim(\calc)} \bigoplus_{a_1, \dots, a_n, b_1, \dots, b_n \in \calo(\calc)} \sum_{s \in \calo(\calc)} \dim(a_1) \cdots \dim(a_n) \dim(s) \\[1.5mm]
 & \times \beta_1^*\! \otimes \dots \otimes\! \beta_n^*\! \otimes \gamma_1^*\! \otimes \dots \otimes\! \gamma_n^* \! \adjincludegraphics[valign=c, scale = 0.74]{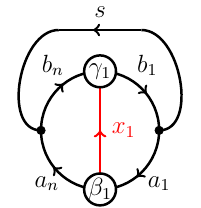}\!\! \adjincludegraphics[valign=c, scale = 0.74]{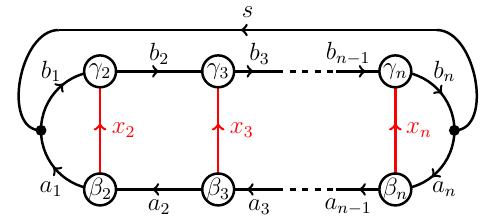} \\[1mm]
 &\overset{\text{Lemma~\ref{lemDisconn}}}{=} \frac{1}{\dim(\calc)} \bigoplus_{a_1, \dots, a_n, b_1, \dots, b_n \in \calo(\calc)} \sum_{s \in \calo(\calc)} \dim(a_1) \cdots \dim(a_n) \dim(s) \\[1.5mm]
 & \times \beta_1^* \otimes \dots \otimes \beta_n^* \otimes \gamma_1^* \otimes \dots \otimes \gamma_n^* \,\adjincludegraphics[valign=c, scale = 0.75]{images/propTube6.pdf} \dots \adjincludegraphics[valign=c, scale = 0.75]{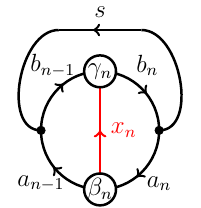},
 \end{align*}
 where we dissolved the vertices labeled by $\alpha$ and $\alpha^*$ and introduced a sum over a simple object. Equipping $\Sigma_X \times I$ with the graph skeleton with the underlying 2-polyhedron
 \begin{align*}
 Q = \left(\Sp^1 \times I\right) \cup_{ \Sp^1 \times \{1/2\}} \mathbb{D}^2
 \end{align*}
 and the obvious stratification $Q^{(1)}$ and $\calc$-colored defect graph (a collection of intervals labeled by the $X_i$), it is evident that the link graphs associated to the switches contained in $\Sp^1 \times \{1/2\}$ have the same form as the graph components obtained in the final expression above. Since the coefficients agree and also $\partial Q = G_P^{\mathrm{op}} \sqcup G_P$, we have
 \begin{align*}
 |\Sigma_V \times I, G_{P}, G_P|_{\calc} = \sum_{i = 1}^k|D_X^{\alpha_i}, G_P, G_P|_{\calc}
 \end{align*}
 and in the limit over surface skeleta this implies the assertion of the proposition.
\end{proof}

Finally, in order to complete the proof of Theorem~\ref{thmMainTV}, we need to show condition \ref{P4}, the normalization condition on 3-balls.
\begin{Proposition} \label{propProp4TV}
Denote as in the paragraph before \eqref{eqBGamma} by $B_{\Gamma} \in \Cob_3^{\calc}(\varnothing, \varnothing)$ a $3$-ball whose boundary sphere is decorated with a $\calc$-colored graph $\Gamma$. Then, the Turaev--Viro invariant ${|B_{\Gamma}|_{\calc} \in \K}$ is given by $\langle \Gamma \rangle$, the value of $\Gamma$ considered as an endomorphism of the tensor unit in the spherical fusion category~$\calc$.
\end{Proposition}
\begin{proof}
 First, we assume that the graph $\Gamma$ is connected. Equip $B_{\Gamma}$ with a skeleton $Q$ with only one region, the boundary $\partial B_{\Gamma} \in \mathrm{Reg}(Q)$. This means that the skeleton has only one 3-cell, so $|B\setminus Q| = 1$. Since $\Gamma$ is connected, each face $a \in \mathrm{Fac}(Q)$ is homeomorphic to an open 2-disk, hence, its Euler characteristic is $\chi(a) = 1$. We will use this fact explicitly later in the proof. Let us abbreviate by $f = |\mathrm{Fac}(Q)|$ the number of faces of the skeleton $Q$. We can then express the invariant $|B_{\Gamma}|_{\calc}$ as the following sum:
 \begin{align}
 \label{eqInvBGamma}
 |B_{\Gamma}|_{\calc} &= \dim(\calc)^{-|B\setminus Q|}\sum_{c\colon \mathrm{Fac}(Q) \rightarrow \calo(\calc)} \dim(c) |c| \nonumber \\
 &= \dim(\calc)^{-1}\sum_{m_1, \dots, m_f \in \calo(\calc)} \biggl(\prod_{j = 1}^f \dim(m_j) \biggr) |m_1, \dots, m_f|.
 \end{align}
 To determine the scalar $|c| = |m_1, \dots, m_f|$ for a fixed coloring $c$, we proceed as follows. The link graph $\Gamma_c(v)$ of each vertex $v$ of $\Gamma$ is given by a wheel graph,
 \begin{align*}
 \Gamma_c(v) = \adjincludegraphics[valign=c, scale = 0.75]{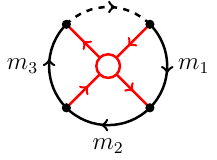},
 \end{align*}
 where the black labels come from the coloring $c$. The data depicted in red is determined by the defect structure and the $\calc$-coloring of $\Gamma$, for which we have suppressed the labels. Now, the scalar $|c|$ can be expressed as the product $|c| = \prod_{v \in V(\Gamma)}\langle \Gamma_c(v) \rangle$, represented by juxtaposition of all link graphs $\Gamma_c(v)$ in the graphical calculus of $\calc$. Here, we slightly abuse notation by omitting the evaluation on the canonical vector $*_c \in H_c(B_{\Gamma})$. Denote by $\Gamma' \subset \mathbb{R}^2$ the $\calc$-colored planar graph obtained from the juxtaposition of link graphs. We take note of some features of $\Gamma'$:
 \begin{itemize}\itemsep=0pt
 \item Each solid black vertex is trivalent, with one red defect edge and two black edges connected to it.
 \item The morphism labeling each solid vertex (under the evaluation on $*_c$) is part of a pair of dual bases. For two vertices labeled by a dual base pair, the objects coloring their red eges and the indices of the simple objects $m_i$ labeling their black edges agree, and the edge orientations are compatible for contraction.
 \end{itemize}
 For the following explanation, we will not differentiate between a colored graph and its evaluation in $\calc$. The main idea is to recover the original graph $\Gamma$ labeling $B_{\Gamma}$ by modifications applied to $\Gamma'$ according to the computation rules of the graphical calculus in $\calc$, while simultaneously keeping track of the effect on the invariant. We proceed as follows. Since the graph $\Gamma$ is connected, it possesses a spanning tree $T \subset \Gamma$. Lemma~\ref{lemDisconn} allows us to fuse dual vertices in $\Gamma'$ along red edges contained in $E(T)$, according to the relation
 \begin{align*}
 \adjincludegraphics[valign=c, scale = 0.75]{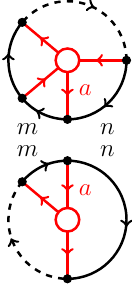} \, \, \, \, \, = \, \, \, \, \,\adjincludegraphics[valign=c, scale = 0.75]{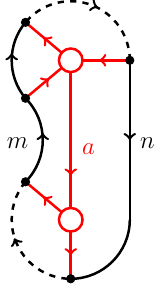}.
 \end{align*}
 This is allowed up to completion of $T$ (that is, until all edges of $T$ are recovered as complete red edges), because we are connecting components that were previously disjoint. The graph $\Gamma''$ arising from this process is a connected graph, so Lemma~\ref{lemDisconn} is no longer applicable. Since $T$ is a spanning tree, its number of edges is given by $|E(T)| = |V(\Gamma)| -1$. By applying Euler's formula, we observe that there are in total
 \begin{align}
 \label{eqCount}
 |E(\Gamma)|- |E(T)| = |V(\Gamma)| + f - 2 - |E(T)| = f-1
 \end{align}
 edges of $\Gamma$ left to connect. We now apply Lemma~\ref{lemDecId} $f-1$ times, as in
 \begin{align}
 \label{eqContr}
 \sum_{m \in \calo(\calc)} \dim(m) \adjincludegraphics[valign=c, scale = 0.75]{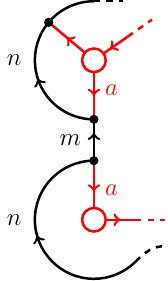} = \, \, \adjincludegraphics[valign=c, scale = 0.75]{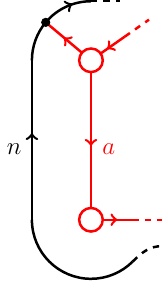},
 \end{align}
 using up the sums over dimension coefficients $\dim(m)$ occurring in \eqref{eqInvBGamma} one by one while contracting the corresponding edge labeled by $m$. Note that the dimensions come with the correct power for contraction due to the fact that the Euler characteristic of all faces is equal to 1, which was a consequence of the connectedness of $\Gamma$. To perform the contraction move, we have to argue why there exists a contractible black edge at each step of the procedure. This is an edge whose label, which is a summation index in \eqref{eqInvBGamma}, appears exactly once in the graph $\Gamma''$ (or any modification of $\Gamma''$ derived by application of \eqref{eqContr}). We present the following combinatorial argument.
 \begin{enumerate}\itemsep=0pt
 \item Consider the original graph $\Gamma \subset \partial B_{\Gamma}$ and the chosen spanning tree $T$. First, we show that there exists a face $a \in P(\Gamma) =\Fac(Q)$ such that there is exactly one edge $e \in E(a)$ with $e \notin E(T)$. This is clear if $\Gamma$ has a vertex with an attached loop. In the case that~$\Gamma$ has no loops, assume the contrary. Since~$T$ is a tree, there can be no cycles in $E(T)$, so there are at least two edges adjacent to every face that are not in~$T$. In other words, the set $S_b := E(b) \cap \left(E(\Gamma)\setminus E(T)\right)$ has at least two elements for every $b \in P(\Gamma)$. We will form a cycle in the dual graph $\Gamma^*$ that contains only edges dual to edges not in $T$ and see that this contradicts the assumption. Pick a starting face $b_0$ (which is a vertex of $\Gamma^*$). Choose one of the two edges in $S_{b_0}$ and connect $b_0$ to a face $b_1$ via the corresponding dual edge. There is at least one other edge in $S_{b_1}$, via whose dual we can connect $b_1$ to a face~$b_2$. Continuing this process, there must be a $k \geq 0$ such that the ordered list $(b_0, \dots, b_k)$ contains the same face twice, since there are only finitely many faces. Then, $\Gamma^*$ contains a cycle $\gamma$ whose edges are dual to edges not in $T$. Removing all edges in $\Gamma$ dual to the edges contained in~$\gamma$, we obtain a graph with two nonempty connected components. By definition of~$\gamma$, two vertices chosen from different components of this graph could not have been previously connected by an edge contained in $E(T)$ in the graph $\Gamma$, but this violates the assumption that $T$ is connected and contains all vertices.
 \item Let $a$ be a face of $\Gamma$ such that a unique edge $e$ adjacent to $a$ does not lie in $T$ and consider the pair of dual vertices $v, w \in V(\Gamma'')$ at $e$. The remaining edges in $E(a)\setminus\{e\}$ have previously been fused while constructing the connected graph $\Gamma''$. Due to the fact that all (dual) vertices in $\Gamma'$ that gave rise to edges in $E(a)\setminus\{e\} \subset E(\Gamma'')$ under fusion possessed a black edge labeled by $m_a \in \mathrm{im}(c)$, this implies that there exists a unique black edge in~$\Gamma''$ from~$v$ to~$w$ labeled by~$m_a$ which can be contracted as in \eqref{eqContr}, possibly after using sphericity of $\calc$.
 \item Form the graph $\Gamma'''$ by contracting $e$. The graph $T' = T \cup \{e\}$ is not a tree but is connected, contains all the vertices of $\Gamma$ and has exactly one cycle which surrounds the face $a$. We can now argue analogously to point 1. and 2. for the existence of a contractible back edge by replacing $T$ by $T'$ and $\Gamma''$ by $\Gamma'''$, and repeat this argument a total of $f-1$ times.
 \end{enumerate}
 After $f-2$ applications of the contraction move, due to \eqref{eqCount}, we have the following: There are exactly two dual vertices remaining and, due to the trivalent nature of the vertices, there are two black edges left to connect. Upon connection, the vertices and one of the black edges (and the sum over its corresponding dimension prefactor) vanish, and the remaining black edge, labeled by some simple object~$m_i$, closes to a loop. The graph $\Gamma$ is now fully recovered, and thus, we see that~\eqref{eqInvBGamma} has turned into
 \begin{align*}
 |B_{\Gamma}|_{\calc} = \dim(\calc)^{-1}\sum_{m_i \in \calo(\calc)} \dim(m_i) \, \adjincludegraphics[valign=c, scale = 0.75]{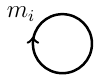} \, \, \langle \Gamma \rangle = \dim(\calc)^{-1}\sum_{m_i \in \calo(\calc)} \dim(m_i)^2 \, \langle \Gamma \rangle = \langle \Gamma \rangle,
 \end{align*}
 proving the asserted equality for the case of a connected graph.

 In the general situation of a disconnected graph, we proceed as follows. Let $\Gamma = \bigsqcup_{i = 1}^k \Gamma_i$ be a~graph with connected components $\Gamma_1, \dots, \Gamma_k$. Between every two components, we can embed a~cylinder in $B_{\Gamma}$ with precisely its lateral surface on $\partial B_{\Gamma}$ so that no defect lines intersect the image of the embedding. Then, we can apply the move $\Delta_3$ to each of these embeddings. The result of the repeated application of $\Delta_3$ is a collection of 3-balls each decorated with graph components, or more precisely, the cobordism $\bigsqcup_{i = 1}^k B_{\Gamma_i}\colon \varnothing \rightarrow \varnothing$. By Proposition~\ref{propProp3TV}, this procedure implies $|B_{\Gamma}|_{\calc} = \bigl|\bigsqcup_{i = 1}^k B_{\Gamma_i}\bigr|_{\calc}$ for the Turaev--Viro invariant, and hence
 \begin{align*}
 |B_{\Gamma}|_{\calc} = \prod_{i = 1}^k|B_{\Gamma_i}|_{\calc} = \prod_{i = 1}^k\langle \Gamma_i\rangle = \langle \Gamma_1 \sqcup \dots \sqcup \Gamma_k \rangle = \langle \Gamma \rangle,
 \end{align*}
 where we used the result for connected graphs obtained earlier. This finishes the proof.
\end{proof}

\begin{proof}[Proof of Theorem~\ref{thmMainTV}]
 Propositions~\ref{propProp1TV},~\ref{propProp2},~\ref{propProp3TV} and~\ref{propProp4TV} together directly imply Theorem~\ref{thmMainTV}.
\end{proof}
\begin{Remark}\label{remTVPPrime}
The arguments entering into the proof of Theorem~\ref{thmMainTV} do not rely in any essential way on the fact that we were considering cobordisms $\varnothing \rightarrow \varnothing$ with empty gluing boundary. In fact, the assertion of Theorem~\ref{thmMainTV} remains true for cobordisms $M \in \Cob_3^{\calc}(\Sigma, \Sigma')$. In the example of the tunnel move, for instance (cf. Proposition~\ref{propProp2}), we can decompose
 \begin{align*}
 M\colon \ \Sigma \xrightarrow{(\Sigma \times I) \sqcup B} \Sigma \sqcup A_1 \xrightarrow{M'} \Sigma', \qquad \Delta_2^x(\iota)(M)\colon \ \Sigma \xrightarrow{(\Sigma \times I) \sqcup \mathbb{T}_x^s} \Sigma \sqcup A_1 \xrightarrow{M'} \Sigma'
 \end{align*}
 with $B$ and $\mathbb{T}_x^s$ as before and deduce the following identity for linear maps $|\Sigma|_{\calc} \rightarrow |\Sigma'|_{\calc}$:
 \begin{align*}
 \allowbreak |M|_{\calc} = |M'|_{\calc} \circ \biggl(\id_{|\Sigma|_{\calc}} \otimes \sum_{x \in \calo(\calc)}\dim(x)|\mathbb{T}_x^s|_{\calc} \biggr) = \sum_{x \in \calo(\calc)}\dim(x)|\Delta_2^x(\iota)(M)|_{\calc}.
 \end{align*}
 By similar reasoning, it follows that the defect TQFT $|\cdot|_{\calc}$ satisfies all the generalized properties~(P1)* through~(P4)* from Remark~\ref{remPPrime}.
\end{Remark}

\subsubsection{Boundary local TQFTs are Turaev--Viro theories}\label{subsecND}
We use the result established in Theorem~\ref{thmMainTV} to prove that the Turaev--Viro defect theory is non-degenerate, from which we derive a uniqueness statement for the Turaev--Viro TQFT.

Recall the notion of a non-degenerate (defect) TQFT (e.g., \cite[Section 15.2.3]{TV}): A boundary-defect TQFT $Z\colon \Cob_3^{\calc} \rightarrow \vect_{\C}$ is called \textit{non-degenerate}, if for any surface $\Sigma \in \Cob_3^{\calc}$, the associated vector space $Z(\Sigma)$ is generated by images of linear maps $Z(M)\colon \K \rightarrow Z(\Sigma)$,
\begin{align*}
 Z(\Sigma) = \biggl \langle \bigcup_{M\colon \varnothing \rightarrow \Sigma} \mathrm{im} ( Z(M)\colon \K \rightarrow Z(\Sigma) )\biggr\rangle_{\K}.
\end{align*}
Theorem~\ref{thmMainTV} allows for a direct proof that the Turaev--Viro defect TQFT $|\cdot|_{\calc}\colon \Cob_3^{\calc} \rightarrow \vect_{\K}$ is non-degenerate for every spherical fusion category $\calc$.

We prepare the argument by replacing a defect cobordism $M\colon \Sigma \rightarrow \Sigma'$ with a family of cobordisms $M^P_{X, \alpha}\colon \Sigma \rightarrow \varnothing \rightarrow \Sigma'$ which factor over the empty manifold and relate to $M$ through the moves $\Delta_1$, $\Delta_2$ and $\Delta_3$.

More precisely, choose some graph skeleton~$P$ for~$M$ and label its faces with a family ${X = \{x_i\}_{i = 1}^p}$ of simple objects $x_i \in \calo(\calc)$. Similarly to the proof of Theorem~\ref{thmMain}, where we transformed a cobordism with empty gluing boundary to a collection of decorated 3-balls, we~can transform~$M$ through a series of applications of pocket, tunnel and tube-capping moves with reference to~$P$ to a cobordism that consists of two (decorated) cylindrical pieces over~$\Sigma$ and~$\Sigma'$, respectively, and a collection of decorated 3-balls. This says that there is a~family~$\{\alpha^{\ell}\}_{\ell = 1}^q$~of~bases~$\alpha^{\ell}$ of vector spaces each of the form \smash{$\Hom_{\calc}\bigl(x_{j_1}^{\varepsilon(1)} \otimes \dots \otimes x_{j_{\ell}}^{\varepsilon(\ell)}, \mathbbm{1}\bigr)$}, such that for any choice of basis vectors $\alpha = \{\alpha_{j_{\ell}}^{\ell}\}_{\ell = 1}^q$, one in each $\alpha^{\ell}$, there are $\calc$-colored graphs
\[
\Gamma(X, \alpha) , \quad \Gamma'(X, \alpha) \quad \text{and} \quad \Upsilon(X, \alpha, 1), \dots, \Upsilon(X, \alpha, r)
\]
 such that applying the moves $\Delta_1$, $\Delta_2^{x_i}(\iota_i)$ and \smash{$\Delta_3^{\alpha_{j_{\ell}}^{\ell}}$} for~appro\-priate tunnels $\iota_i$, the cobordism $M$ transforms into a cobordism $M_{X, \alpha}^P$ that factors~as
\begin{align} \label{eqNDFact}
 M_{X, \alpha}^P\colon \ \Sigma \xrightarrow{\mbox{\normalsize$C_{\Sigma, \Gamma(X, \alpha)}$}} \varnothing \xrightarrow{\mbox{\normalsize$\bigsqcup_{i = 1}^rB_{\Upsilon(X, \alpha, i)}$}} \varnothing \xrightarrow{\mbox{\normalsize$C'_{\Sigma', \Gamma'(X, \alpha)}$}} \Sigma'.
\end{align}
Here, $C_{\Sigma, \Gamma(X, \alpha)}$ and $C'_{\Sigma', \Gamma'(X, \alpha)}$ are topologically cylinders $\Sigma \times I$ and $\Sigma' \times I$, with gluing boundary one copy of $\Sigma$ and $\Sigma'$, respectively. Their free boundary is decorated by the $\calc$-colored graphs $\Gamma(X, \alpha)$ and $ \Gamma'(X, \alpha)$, which are shaped by the $\calc$-coloring of $M$ and the local moves.
\begin{Proposition}
 \label{propND}
 The Turaev--Viro defect TQFT $|\cdot|_{\calc}\colon \Cob_3^{\calc} \rightarrow \vect_{\K}$ is non-degenerate.
\end{Proposition}
\begin{proof}
 Consider the cylinder $\Sigma \times I$ over a $\calc$-colored surface $\Sigma \in \Cob_3^{\calc}$ and equip it with a graph skeleton $P$.
 For all families $X = \{x_i\}_{i = 1}^p$ and $\alpha = \{\alpha_{j_{\ell}}^{\ell}\}_{\ell = 1}^q$, specified as in the last paragraph, we have cobordisms $(\Sigma \times I)_{X, \alpha}^P$ which factor as in \eqref{eqNDFact} as
 \begin{align*}
 \Sigma \xrightarrow{\mbox{\normalsize$C_{\Sigma, \Gamma(X, \alpha)}$}} \varnothing \xrightarrow{\mbox{\normalsize$\bigsqcup_{i = 1}^rB_{\Upsilon(X, \alpha, i)}$}} \varnothing \xrightarrow{\mbox{\normalsize$C'_{\Sigma, \Gamma'(X, \alpha)}$}} \Sigma.
 \end{align*}
 Evaluating $|\cdot|_{\calc}$ on $\Sigma \times I$ and using (P1)*, (P2)* and (P3)* (cf. Remark~\ref{remTVPPrime}), we can decompose any vector $v \in |\Sigma|_{\calc}$ by summing over the labels $X$ and $\alpha$ as
 \begin{align}
 \label{eqNonDeg}
 v = |\Sigma \times I|_{\calc}(v) = \sum_{x_1, \dots, x_p \in \calo(\calc)} \sum_{(j_{\ell})_{\ell}}\bigl|(\Sigma \times I)_{X, \alpha}^P\bigr|_{\calc} = \sum_{x_1, \dots, x_p \in \calo(\calc)} \sum_{(j_{\ell})_{\ell}} \gamma(X, \alpha, v)\,w_{X, \alpha}
 \end{align}
 with the vectors \smash{$w_{X, \alpha} = |C'_{\Sigma, \Gamma'(X, \alpha)}|_{\calc}(1) \in |\Sigma|_{\calc}$}. Here, the coefficients $\gamma(X, \alpha, v) \in \K$ are products of the scalars $\dim(\calc)$ raised to some nonpositive power, $\prod_{i = 1}^p\dim(x_i)$, $\prod_{i = 1}^r\langle\Upsilon(X, \alpha, i)\rangle $ and $|C_{\Sigma, \Gamma(X, \alpha)}|_{\calc}(v)$. But \eqref{eqNonDeg} now shows that $v$ lies in the span of the images of the linear maps defined by the cobordisms $C'_{\Sigma, \Gamma'(X, \alpha)}\colon \varnothing \rightarrow \Sigma$, proving that $|\cdot|_{\calc}$ is non-degenerate.
\end{proof}

Non-degeneracy implies that the Turaev--Viro defect topological field theory is completely characterized by~\ref{P1} to~\ref{P4}.
\begin{Theorem}\label{thmComb}
For any spherical fusion category $\calc$, there is a unique $($up to monoidal iso\-mor\-phism$)$ boundary local defect TQFT $\Cob_3^{\calc} \rightarrow \vect_{\K}$, the Turaev--Viro defect TQFT $|\cdot|_{\calc}$.
\end{Theorem}
\begin{proof}By Theorem~\ref{thmMainTV}, the Turaev--Viro TQFT is a defect TQFT with the desired properties.

Let now $Z\colon \Cob_3^{\calc} \rightarrow \vect_{\K}$ be a defect TQFT satisfying \ref{P1} to \ref{P4}. Since $\Cob_3^{\calc}$ is spherical, we can form the dimension of any surface $\Sigma \in \Cob_3^{\calc}$ by taking the trace over the identity cobordism $\Sigma \times I$ to obtain the cobordism $\Sigma \times \Sp^1 \in \Cob_3^{\calc}(\varnothing, \varnothing)$. Since $\K$ has characteristic zero and $Z$ is symmetric monoidal, the dimension is mapped by $Z$ to the vector space dimension $\dim_{\K}Z(\Sigma) = Z\bigl(\Sigma \times \Sp^1\bigr)$. Due to Theorem~\ref{thmMain}, this agrees with $\bigl|\Sigma \times \Sp^1\bigr|_{\calc} = \dim_{\K}|\Sigma|_{\calc}$. Note that this and Theorem~\ref{thmMain} imply that requirements (b) and (c) of \cite[Lemma~17.2]{TV} are met. Since~$| \cdot |_{\calc}$ is non-degenerate by Proposition~\ref{propND}, requirement (a) is also satisfied, and the assertion of Lemma~17.2 gives the desired monoidal isomorphism between $Z$ and $| \cdot |_{\calc}$.
\end{proof}

Restricting attention to the subcategory $\Cob_3 \subset \Cob_3^{\calc}$, this implies the following.
\begin{Corollary}
 \label{corThmMain2}
 The Turaev--Viro theory $|\cdot|_{\calc}$ is the unique $($up to monoidal isomorphism$)$ topological field theory $Z\colon \Cob_3 \rightarrow \vect_{\K}$ admitting an extension
 \[
 \centering
 \adjincludegraphics[valign=c, scale = 1.16]{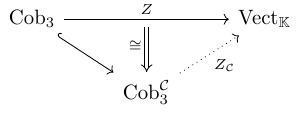}
 \]
 to a boundary local defect theory $Z_{\calc}$ along the inclusion $\Cob_3 \hookrightarrow \Cob_3^{\calc}$. In other words, if $Z$ admits such an extension, then $Z \cong |\cdot|_{\calc}$ as monoidal functors.
\end{Corollary}
\begin{proof}
By Theorem~\ref{thmComb}, $Z_{\calc} \cong | \cdot |_{\calc}$, and restriction along the monoidal functor $\Cob_3 \hookrightarrow \Cob_3^{\calc}$ gives the desired statement.
\end{proof}

\begin{Example}An example of a nontrivial TQFT that does not admit a boundary local defect extension to $\Cob_3^{\calc}$ for any spherical fusion category $\calc$ over $\C$ is the dimensional reduction of four-dimensional Dijkgraaf--Witten theory to three dimensions along the circle. To be precise, let $G$ be a finite nontrivial group and $Z_G^4\colon \Cob_4 \rightarrow \vect_{\C}$ four-dimensional Dijkgraaf--Witten theory. We can define the TQFT $\mathrm{Red}_{\Sp^1}Z_G^4\colon \Cob_3 \rightarrow \vect_{\C}$ by setting $\mathrm{Red}_{\Sp^1}Z_G^4(X) = Z_G^4\bigl(X \times \Sp^1\bigr)$ for~$X$ a surface or a cobordism in $\Cob_3$. For the 3-manifold $X = \Sp^3$, one has
 \begin{align*}
 \mathrm{Red}_{\Sp^1}Z_G^4(\Sp^3) = \frac{\bigl|{\Hom\bigl(\pi_1\bigl(\Sp^3\bigr) \times \pi_1\bigl(\Sp^1\bigr), G\bigr)}\bigr|}{|G|} = \frac{|{\Hom(\mathbb{Z}, G)}|}{|G|} = 1,
 \end{align*}
 whereas $\bigl| {\Sp^3}\bigr|_{\calc} = \dim(\calc)^{-1}$ for any spherical fusion category $\calc$. This only leaves the possibility $\calc = \vect_{\C}$, but the associated Turaev--Viro TQFT is trivial, whereas $\mathrm{Red}_{\Sp^1}Z_G^4$ is not. In fact, as found in~\cite{MW} in the context of extended topological field theory, the dimensional reduction of four-dimensional Dijkgraaf--Witten theory decomposes as a sum of three-dimensional Dijkgraaf--Witten theories associated with centralizer subgroups. Concretely, restricting to the non-extended situation, by \cite[Corollary~4.3]{MW} we have
 \smash{$ \mathrm{Red}_{\Sp^1}Z_G^4 \cong \bigoplus_{[g] \in \mathrm{Conj}(G)} Z_{C_G(g)}^3$},
 where $\mathrm{Conj}(G)$ is the set of conjugacy classes of $G$ and $C_G(g) \leq G$ denotes the centralizer subgroup of $g \in G$. This is consistent with $\mathrm{Red}_{\Sp^1}Z_G^4\bigl(\Sp^3\bigr) = 1$ because of the orbit-stabilizer theorem.

 However, note that three-dimensional Dijkgraaf--Witten theory \textit{is} an example associated to the fusion category $\vect_G$, which will be treated extensively in Section~\ref{secDW}.
\end{Example}
\subsubsection{Further applications of the defect boundary formalism}
\label{subSecApp}
This subsection contains some remarks on the utility of the various moves, minor computations referenced earlier and the proof that the Turaev--Viro theory satisfies the generalized tunnel move. The results discussed here will not be needed in the remainder of the article.
\begin{Remark}
 We finish the series of remarks started in Remarks~\ref{remGenTun} and~\ref{remSkelTun} and show that the Turaev--Viro TQFT respects the generalized tunnel move for a tunnel $\iota\colon A_n \times I \hookrightarrow M$ in a cobordism $M \in \Cob_3^{\calc}(\Sigma, \Sigma')$. As in the proof of Proposition~\ref{propProp2}, we embed a slightly larger tunnel in $\Delta_2^{x}(\iota, A_n)(M)$ and $M$, respectively, and cut along its lateral surface, which is homeomorphic to $\partial A_n \times I \cong A^{\sqcup n +1}$. Locally, we obtain the cobordisms $(\mathbb{T}_x^s)^{\sqcup n +1}\colon \varnothing \rightarrow A^{\sqcup n+1}$, where $\mathbb{T}_x^s$ is from Lemma~\ref{lemProp2}, and $K_n\colon \varnothing \rightarrow A^{\sqcup n+1}$, where $K_n$ is the cylinder over $A_n$ with its lateral surface as gluing boundary. Pictorially, we represent $K_n$ as
 \begin{align*}
 K_n = \adjincludegraphics[valign=c, scale = 0.8]{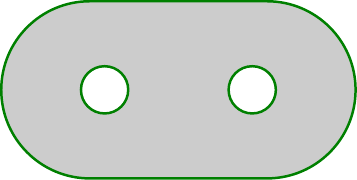} \, \times \, I,
 \end{align*}
 where we have drawn the case $n = 2$ and marked the gluing boundary in green. We compute the homomorphism $|K_n|_{\calc}$, where we represent the cobordisms by their two-dimensional projections:
 \begin{align*}
 |K_n|_{\calc} & = \left| \, \adjincludegraphics[valign=c, scale = 0.8]{images/genTun1.pdf} \, \right|_{\calc} \\
 &\hspace{-2mm}\overset{\text{(P2)*}}{=} \sum_{x \in \calo(\calc)} \dim(x) \left| \, \adjincludegraphics[valign=c, scale = 0.8]{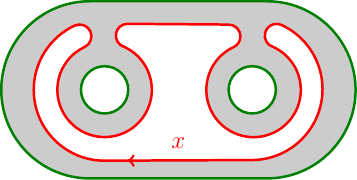} \, \right|_{\calc} \\
 &\hspace{-2mm}\overset{\text{(P3)*}}{=} \sum_{x \in \calo(\calc)}   \dim(x) \left| \, \adjincludegraphics[valign=c, scale = 0.8]{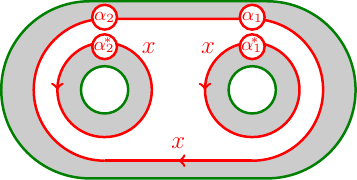} \, \right|_{\calc} \\
 & = \sum_{x \in \calo(\calc)}   \dim(x) \biggl(\bigotimes_{j = 1}^n |\mathbb{T}_{x, \alpha_j^*}^s|_{\calc}\biggr) \otimes |\mathbb{T}_{x, \alpha_1, \dots, \alpha_n}^s|_{\calc}.
 \end{align*}
 Here, we have inserted an appropriate tunnel (with base $A_0 = \mathbb{D}^2$) and used the tube-capping move $n$ times, giving rise to $n$ pairs of dual bases $\alpha_j$, $\alpha_j^*$. The cobordisms \smash{$\mathbb{T}_{x, \alpha_j^*}^s$} and \smash{$\mathbb{T}_{x, \alpha_1, \dots, \alpha_n}^s$} are~modifications of $\mathbb{T}_{x}^s$, where the $x$-labeled defect loop has been decorated with a node labeled by elements of $\alpha_j^*$ and $\alpha_1 \circ \dots \circ \alpha_n$ (with the obvious notation), respectively. The sum over dual bases is implied in the expression as usual. It is easy to see that these evaluate under the application of $| \cdot , G_P|_{\calc}$, where $P$ is the graph skeleton from Lemma~\ref{lemProp2} and $G_P$ the corresponding surface skeleton, to a variant of the expression from \eqref{eqGenTunPrep}, where the strand labeled by $x$ now features nodes decorated with basis elements. Using sphericity of $\calc$ and Lemma~\ref{lemDisconn}, the ${n+1}$ graph components become connected while summing over the dual bases. Applying \cite[Lemma~12.3]{TV} $n$ times, we obtain the $n$-fold tensor product of exactly~\eqref{eqGenTunPrep} while picking up a~correction factor $\dim(x)^{-n}$. Since $\chi(A_n) = 1-n$, we arrive at
 \begin{align*}
 |K_n|_{\calc} = \sum_{x \in \calo(\calc)} \dim(x)^{\chi(A_n)} \, |\mathbb{T}_x^s|_{\calc}^{ \otimes n},
 \end{align*}
 which is precisely the local version of \eqref{eqGenTun}. From here, the global version follows as in Proposition~\ref{propProp2}.
\end{Remark}

\begin{Remark}
 We comment on the realization of generalized Frobenius--Schur indicators as Turaev--Viro invariants of decorated solid tori as presented in~\cite{FS}. Recall from \cite[Definition~2.1]{NS} the notion of the $r$-twisted $n$-th generalized Frobenius--Schur indicator $\nu_{n, r}^X(V)$ for $n, r \in \mathbb{Z}$, $V \in \calc$ an object in a spherical fusion category and $X \in \mathcal{Z}(\calc)$ an object in the Drinfeld center of $\calc$. As is shown in \cite[Proposition 14]{FS} (or \cite[Theorem~5.0.1]{F}), by defining a certain cobordism \smash{$\mathcal{T}_{n,r}^{X, V}\colon \varnothing \rightarrow \varnothing$}, which is a solid torus with a non-contractible defect loop labeled by $X$ in its interior and $n$ non-intersecting defect loops labeled by $V$ that twist around the free boundary, the Frobenius--Schur indicator can be expressed as \smash{$\nu_{n, r}^X(V) = \bigl|\mathcal{T}_{n,r}^{X, V}\bigr|_{\calc}$}. As we do not allow any internal defect lines in this article, we can only rigorously treat the case $X = \mathbbm{1}$, which is equivalent to the absence of the internal defect line. Using \ref{P3} and \ref{P4}, the proof of the above statement then becomes particularly simple. For instance, in the case $n = 3$, $r = 1$,
 \begin{align}
 \label{eqFS}
 \bigl|\mathcal{T}_{3, 1}^{\mathbbm{1}, V}\bigr|_{\calc} \overset{\rm \ref{P3}}{=} \sum_{i = 1}^k \, \left| \! \adjincludegraphics[valign=c, scale = 0.75]{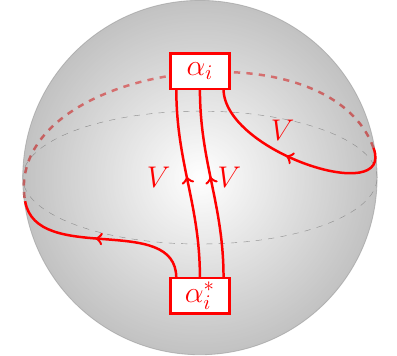} \! \right|_{\calc} \overset{\text{\ref{P4}} + \text{ \cite[(2.4)]{FS}}}{=} \nu_{3, 1}^{\mathbbm{1}}(V),
 \end{align}
 where we have introduced the sum over an appropriate pair of dual bases by cutting along a~cylindrical piece of \smash{$\mathcal{T}_{3, 1}^{\mathbbm{1}, V}$} using \ref{P3}, which is possible since we can locally untwist the defect lines. The general case for $n, r \in \mathbb{Z}$ arbitrary works completely analogously.

 The same proof should be possible even in the case of arbitrary $X \in \mathcal{Z}(\calc)$. For this, internal defect lines and especially junctions between internal and free boundary defect lines have to be included in the formalism presented in Section~\ref{subSecTV}. Internal defect lines have for instance been considered in \cite{F, FS, TV}. An alternative way to include them in the context of Frobenius--Schur indicators has been outlined in \cite[Section~5]{LMWW} using alterfold theory, by removing a tubular neighborhood of the Wilson line and labeling the boundary appropriately.

 Using the half-braidings, the properties~\ref{P3} and~\ref{P4} have obvious generalizations if one allows this extra structure, and the obvious analog of \eqref{eqFS} works in the general setting to provide a very short proof of the equality $\nu_{n, r}^X(V) = \bigl|\mathcal{T}_{n,r}^{X, V}\bigr|_{\calc}$.
\end{Remark}

\begin{Observation}\label{obs}
We discuss the dependence on the choice of embedding defining a tunnel in the context of link invariants, and show how the compatibility~\ref{P2} of the Turaev--Viro invariant with the tunnel move leads to a certain invariance property.
 \begin{enumerate}\itemsep=0pt
 \item Let $L \subset \Sp^3$ be a tame link embedded in the 3-sphere and $\widetilde{L}$ a sufficiently small tubular neighborhood. For a choice of spherical fusion category $\calc$, the Turaev--Viro invariant of the complement $\Sp^3 \setminus \widetilde{L}$ is a link invariant, and we set $\vartheta(L) = \bigl|\Sp^3 \setminus \widetilde{L}\bigr|_{\calc}$.
 \item For two links, even if their numbers of components match, the invariants $\vartheta(L)$ and $\vartheta(L')$ are generally different. As an example, we consider the unlink $L_0$ and the Hopf link $H$ and show that their link invariants are different by using Theorem~\ref{thmMainTV}. First, note that $\Sp^3 \setminus \widetilde{L_0}$ is a solid torus~$\mathbb{T}^s$ with another solid torus embedded along a contractible loop removed from its interior. Cutting twice along an appropriately chosen disk with the tube-capping move~$\Delta_3$, this manifold transforms into $\Sp^2 \times I$, which in turn arises from the pocket move~$\Delta_1$ applied to the 3-ball. Thus, by \ref{P1}, \ref{P2} and \ref{P3}, we find $\vartheta(L_0) = \dim(\calc)$.

 On the other hand, we obtain $\Sp^3\setminus \widetilde{H} \cong A_1 \times \Sp^1$, and by Corollary~\ref{corAnn}, we have $\vartheta(H) = \bigl|A_1 \times \Sp^1\bigr|_{\calc} = |\calo(\calc)|$. These two numbers are in general different from each other. For example, for $\calc = \Rep_{\C}(\mathfrak{S}_3)$, the former is 6, while the latter is~3.
 \item We can generalize the link invariants by decorating the boundary of each component with a defect loop. Denoting by $|L|$ the number of components of~$L$, consider on \smash{$\partial\bigl(\Sp^3 \setminus \widetilde{L}\bigr)$} the collection of defect loops labeled by objects $x_1, \dots, x_{|L|} \in \calo(\calc)$ so that each loop is contained in the boundary of a different component of \smash{$\widetilde{L}$} and is contractible in~\smash{$\widetilde{L}$}. Denote this defect manifold by \smash{$\bigl(\Sp^3 \setminus \widetilde{L}\bigr)_{x_1, \dots, x_{|L|}}$} and set
 \begin{align*}
 \vartheta(L, x_1, \dots, x_{|L|}) :=\bigl|\bigl(\Sp^3 \setminus \widetilde{L}\bigr)_{x_1, \dots, x_{|L|}}\bigr|_{\calc} \in \K,
 \end{align*}
 where it should be noted that $\vartheta(L, \mathbbm{1}, \dots, \mathbbm{1})$ coincides with the unlabeled invariant $\vartheta(L)$.\footnote{There are two different ways to orient each labeled loop, giving rise to two different invariants. We suppress this freedom of choice notationally and simply make one (arbitrary) choice on how to orient the defect loops.}
 \item Although these invariants do not coincide for two links $L$, $L'$ with $|L| = |L'|$, the $\dim(x_1) \cdots \linebreak \dim(x_{|L|})$-weighted sum of the link invariants $\vartheta(L, x_1, \dots, x_{|L|})$ over all simple objects agrees for $L$ and~$L'$, which is a consequence of the tunnel-move compatibility~\ref{P2} of~$| \cdot|_{\calc}$. We see this as follows. Let $L_1, \dots, L_{|L|}$ denote the components of a~link~$L$, labeled by arbitrary simple objects $x_1, \dots, x_{|L|} \in \calo(\calc)$. Consider for each link component $L_j$ a~3\nobreakdash-ball \smash{$B^3 \subset \widetilde{L_j}$} in $\Sp^3$ that intersects \smash{$\partial\widetilde{L_j}$} in \smash{$B^3 \cap \partial\widetilde{L_j} \cong \Sp^1 \times I$}. This distinguishes two disks in the complement of this intersection on the boundary of $B^3$, which define a tunnel $\iota_j\colon C \hookrightarrow \Sp^3 \setminus \bigl(B^3\bigr)^{\circ}$ ending at the specified disks such that \smash{$\mathrm{im} (\iota_j) = \widetilde{L_j}$}. Then, using the tunnel move $\Delta_2^{x_j}(\iota_j)$ and property~\ref{P2} for every $j \in \{1, \dots, |L|\}$, we can in a first step go from the 3-sphere minus $\widetilde{L}$ to the 3-sphere minus a collection of 3-balls, and then use~\ref{P1} and the pocket move $\Delta_1$ to derive
 \begin{align}
 & \sum_{x_1, \dots, x_{|L|} \in \calo(\calc)}\prod_{j= 1}^{|L|} \dim(x_j) \, \vartheta(L, x_1, \dots, x_{|L|}) \nonumber\\
 &\qquad \overset{\rm \ref{P2}}{=} \Biggl|\Sp^3\setminus \bigsqcup_{j = 1}^{|L|}\bigl(B^3\bigr)^{\circ}\Biggr|_{\mathcal{C}} \overset{\rm \ref{P1}}{=} \dim(\calc)^{|L|-1}.\label{eqLink}
 \end{align}
 Indeed, this expression only depends on the number of components of the link. For example, for two (tame) knots $K$, $K'$, we have \[
 \sum_{x \in \calo(\calc)}\dim(x) \vartheta(K, x) = \sum_{x \in \calo(\calc)}\dim(x) \vartheta(K', x).
 \]
 \item Consider once more the Hopf link $H$. For $x \in \calo(\calc)$ simple, $\bigl(\Sp^3\setminus \widetilde{H}\bigr)_{\mathbbm{1}, x} \cong \bigl(A_1 \times \Sp^1\bigr)_x$, where in the latter manifold a defect loop labeled by $x$ runs around the boundary of a contractible disk in a solid torus whose interior is removed from another solid torus to form $A_1 \times \Sp^1$. Using \eqref{eqAnnCut} to cut along an annulus $A_1$ and an analogous computation as in Remark~\ref{remTorus} or Lemma~\ref{lemTorxy}, we obtain for the invariant $\vartheta(H, \mathbbm{1}, x)$ the expression
 \begin{align}
 \label{eqKnotInv}
 \vartheta(H, \mathbbm{1}, x) = \sum_{y \in \calo(\calc)} \dim_{\K}\Hom_{\calc}(x, y \otimes y^*).
 \end{align}
 Due to symmetry, we have $\vartheta(H, \mathbbm{1}, x) = \vartheta(H, x, \mathbbm{1})$. Hence, with \eqref{eqKnotInv}, we have computed the dimension of $|A_x|_{\calc}$, where $A_x$ denotes the annulus with a point labeled by $x$ inserted in the outer boundary circle of $A_1$. Thus, there is a vector space isomorphism
 \begin{align*}
 |A_x|_{\calc} \cong \Hom_{\calc}\biggl(x, \bigoplus_{y \in \calo(\calc)}y \otimes y^*\biggr).
 \end{align*}
 This is consistent with Lemma~\ref{lemProp2}, which gave for $x = \mathbbm{1}$ the result $|A_{\mathbbm{1}}|_{\calc} \cong K_0(\mathcal{C}) \otimes_{\mathbb{Z}} \K$, since the latter is isomorphic as a vector space to \smash{$\Hom_{\calc}\bigl(\mathbbm{1}, \bigoplus_{y \in \calo(\calc)}y \otimes y^*\bigr)$}.
 \item \label{obsEx} Let $\calc = \mathcal{F}_{\varphi}$ denote the Fibonacci category, where $\varphi \in \C$ is the golden ratio (cf. Example~\ref{exProp}\,\ref{itemAlg}). For the Hopf link~$H$, \eqref{eqLink} reads
 \begin{align}
 \label{eqFib}
2 + \varphi = \vartheta(H, \mathbbm{1}, \mathbbm{1}) + 2\varphi\cdot \vartheta(H, \mathbbm{1}, \tau) + \varphi^2 \cdot \vartheta(H, \tau, \tau).
 \end{align}
 In particular, at least one of the above $\tau$-labeled invariants is nonzero, since $\vartheta(H, \mathbbm{1}, \mathbbm{1}) = 2$ as we have seen before. Using \eqref{eqKnotInv}, we have with $\tau^* = \tau$ that
 \begin{align*}
\vartheta(H, \mathbbm{1}, \tau) = \dim_{\C}\Hom_{\calc}(\tau, \mathbbm{1}) + \dim_{\C}\Hom_{\calc}(\tau, \mathbbm{1} \oplus \tau) = 1.
 \end{align*}
 Utilizing \eqref{eqFib}, we obtain the expression $\vartheta(H, \tau, \tau) = -\frac{1}{\varphi} = 1-\varphi$.
 Thus, the link invariants for the Hopf link are completely determined by the list
 \begin{align*}
\vartheta(H, \mathbbm{1}, \mathbbm{1}) = 2, \qquad \vartheta(H, \mathbbm{1}, \tau) = \vartheta(H, \tau, \mathbbm{1}) = 1, \qquad \vartheta(H, \tau, \tau) = 1-\varphi.
 \end{align*}
 The last item in this list can also be computed explicitly, but the direct computation is more involved than the argument based on \eqref{eqLink}.
 \end{enumerate}
\end{Observation}

\section{The Dijkgraaf--Witten TQFT with boundary defects}
\label{secDW}
In this section, we will prove that a restriction of Dijkgraaf--Witten theory with defects as introduced in~\cite{FMM} and reviewed in Section~\ref{subSecDW} provides an example of a boundary local defect TQFT (cf.\ Definition~\ref{defPropInv}). Theorem~\ref{thmComb} then shows that the Dijkgraaf--Witten boundary-defect TQFT for a group $G$ is monoidally isomorphic to the Turaev--Viro boundary-defect TQFT for the spherical fusion category $\vect_G$. This establishes a state sum description of Dijkgraaf--Witten theories with boundary defects. The results are collected in Theorem~\ref{thmDW}.

We will start by adapting the relevant cobordism categories to treat Turaev--Viro theory and Dijkgraaf--Witten theory (cf.\ Sections~\ref{subSecTV} and~\ref{subSecDW}, respectively) on the same footing. Since the Dijkgraaf--Witten TQFT so far does not (directly) incorporate free boundaries, we first have to accommodate the description in that regard.

\subsection{Cobordism categories with free boundaries for Dijkgraaf--Witten theories}
\label{subSecFB}
Recall the category $\Cob_3^{\mathsf{def}}$ from Definition~\ref{defCobDW}. In this category, cobordisms are not allowed to possess a free boundary, only gluing boundary. Since our setting (namely, the one from Definition~\ref{defCobDef}) works with free boundaries, we will define the larger category $\Cob_3^{\mathsf{def}}(\partial)$ that contains $\Cob_3^{\mathsf{def}}$ as a (non-full) subcategory.

\paragraph{Stratifications on manifolds with free boundary decomposition.} Again, we restrict the dimension of our manifolds to $n \leq 3$. Since we want to include free boundaries, we will extend the homogeneity condition from \cite[Definition~3.7]{FMM} to the setting of $n$-manifolds $X$ with a \textit{free boundary decomposition}, that is, a decomposition of the boundary $\partial X = X_{\partial}^g \cup X_{\partial}^f$ into two disjoint parts $X_{\partial}^g$, $X_{\partial}^f$ such that $X_{\partial}^g$ is a closed submanifold of $\partial X$.
\begin{Example} \label{exBound}
 As our goal is to generalize the category of cobordisms in three dimensions to include free boundaries, we are interested in the following examples.
 \begin{enumerate}[(i)]\itemsep=0pt
 \item \label{exFreeBoundCob} Any 3-manifold $M$ with an embedding $h\colon \Sigma \sqcup \Sigma' \hookrightarrow \partial M$ as considered in Definition~\ref{defCobDef} provides an example of a free boundary decomposition, setting $M_{\partial}^g := \mathrm{im}(h) = \partial_gM$ and $M_{\partial}^f = \partial M \setminus \partial _g M = \partial_fM$.
 \item \label{exFreeBoundSurf} For a surface $\Sigma$, defining $\Sigma_{\partial}^f = \partial \Sigma$ and $\Sigma_{\partial}^g = \varnothing$ provides a free boundary decomposition.
 \end{enumerate}
\end{Example}

Let $\Sp_+^n = \{x \in \Sp^n\mid x_n \geq 0\} \subset \Sp^n$ denote the closed half-sphere and $\mathbb{R}_+^n = \{x \in \mathbb{R}^n \mid x_n \geq 0\}$ the closed half-space in $\mathbb{R}^n$. Denote by $cS$ the open cone over a manifold $S$. If $S$ is equipped with a stratification (in the most basic sense of the term, cf. the first appearance of `stratification' in Section~\ref{subSubSecDWCob}), the cone $cS$ carries an induced stratification, which has a $(k+1)$-stratum for every $k$-stratum of $S$ and a unique 0-stratum, the apex $v_c$ of the cone. In the following, let~$\mathbb{R}^n$ and~$ \mathbb{R}_+^n$ have trivial stratification. Then, the products $cS \times \mathbb{R}^k$ and $cS \times \mathbb{R}_+^k$ inherit a~stratification as well. We generalize \cite[Definition~3.7]{FMM}.
\begin{Definition}
 \label{defFBHomStr}
 Let $X$ be an $n$-manifold with a stratification and a free boundary decomposition $\partial X = X_{\partial}^g \cup X_{\partial}^f$. Let $s$ be a $k$-stratum of $X$ and $x \in s$. Define non-compact manifolds $N_x$ (so far without choice of stratification), serving as local models of neighborhoods of $x$, as
 \begin{itemize}\itemsep=0pt
 \item $N_x = c\Sp^{n-k-1} \times \, \mathbb{R}^k$ for $x \in \mathrm{Int}(X)$,
 \item $N_x = c\Sp^{n-k-1} \times \, \mathbb{R}_+^k$ for $x \in \mathrm{Int}(X_{\partial}^g)$,
 \item $N_x = c\Sp_+^{n-k-1} \times \, \mathbb{R}^k$ for $x \in X_{\partial}^f$
 \item and $N_x = c\Sp_+^{n-k-1} \times \, \mathbb{R}_+^k$ for $x \in \partial X_{\partial}^g$.
 \end{itemize}
 We call the stratum $s$ \textit{homogeneous} if there exists a stratification of the relevant sphere $\Sp^{n-k-1}$ or half-sphere $\Sp_+^{n-k-1}$, inducing a stratification on $N_x$ as outlined above, such that for any $x \in s$, there is an open neighborhood $U_x \subset X$ of $x$ and an isomorphism $N_x \rightarrow U_x$ of manifolds with stratification mapping $(v_c, 0)$ to $x$. We call the stratification on $X$ homogeneous if it consists entirely of homogeneous strata.
\end{Definition}
The situation in \cite{FMM} is recovered by choosing the free boundary decomposition given by $X_{\partial}^g = \partial X$ and $X_{\partial}^f = \varnothing$.

Similar to the situation $X_{\partial}^f = \varnothing$ considered in \cite{FMM}, homogeneity forces strata intersecting $X_{\partial}^g$ to be transversal to $X_{\partial}^g$. Moreover, it follows that the (manifold-) interior $\mathrm{Int}(s)$ of a stratum $s$ that intersects $X_{\partial}^f$ is contained entirely in $X_{\partial}^f$. We will refer to such strata of $X$ as \textit{free boundary strata}.

We will keep the saturation condition from \cite[Definition~3.7]{FMM} essentially unchanged: A homogeneous stratum $s$ is called \textit{saturated} if for every $x \in s$ the sphere or half-sphere defining the neighborhood $N_x$ has a stratum in every codimension. The stratification is saturated if all strata are saturated. Similarly to \cite[Lemma~3.8]{FMM}, a homogeneous and saturated stratification of a manifold with free boundary decomposition induces a stratification with the same properties on $X_{\partial}^g$, with free boundary decomposition specified by $(X_{\partial}^g)_{\partial}^f = \partial X_{\partial}^g$ and $(X_{\partial}^g)_{\partial}^g = \varnothing$.

Akin to \cite[Definition 3.9]{FMM}, we can also consider oriented stratifications.
\begin{Definition}
 \label{defFBOrStr}
 Let $X$ be an $n$-manifold with free boundary decomposition. An \textit{oriented} stratification of $X$ is a homogeneous and saturated stratification where every stratum $s$ of codimension $\mathrm{codim}(s) \in \{1, 2\}$ carries an orientation, such that every free boundary stratum of codimension 1 has the same orientation as the boundary of $X$ (cf. the convention of Section~\ref{subSecTQFTs} for induced orientations).
\end{Definition}
We will keep the conventions on framings of 1-strata as in \cite[Definition~3.10]{FMM}. Note that a free boundary 1-stratum can always be framed.
Following the convention in~\cite{FMM}, a stratified $n$-manifold with free boundary decomposition will from now on be assumed to carry a homogeneous, saturated, oriented and framed stratification.

The assignment of (Dijkgraaf--Witten type) defect data to a manifold with free boundary decomposition remains virtually unchanged in comparison to Definition~\ref{defDefData}. The only difference is that to a free boundary $(n-1)$-stratum $s$, we assign a $G_{L(s)}$-set $M_s$, since there is only one adjacent 3-stratum $L(s)$, defined with respect to the normal vector on the boundary. The modifications to the action groupoids $\cald_u$ assigned to strata from these data are obvious and we can repeat Definition~\ref{defDefData} in this more general situation verbatim.

The result of these definitions is a sensible notion of \textit{stratified manifold with free boundary decomposition and defect data}. We can now define the category $\Cob_3^{\mathsf{def}}(\partial)$.
\begin{Definition}
 The symmetric monoidal category $\Cob_3^{\mathsf{def}}(\partial)$ is the category with objects stratified surfaces $\Sigma$ with defect data and free boundary decomposition as in Example~\ref{exBound}\,\ref{exFreeBoundSurf}.

 Morphisms $\Sigma \rightarrow \Sigma'$ are equivalence classes of pairs $(M, h)$, called cobordisms, where $M$ is a~stratified manifold with free boundary decomposition and defect data, and $h\colon \overline{\Sigma} \sqcup \Sigma' \hookrightarrow \partial M$ is an embedding of stratified manifolds with defect data that induces the free boundary decomposition on $M$ as in Example~\ref{exBound}\,\ref{exFreeBoundCob}. Two such cobordisms $(M, h)$ and $(M', h')$ are equivalent whenever there is an isomorphism of stratified manifolds with defect data $f\colon M \rightarrow M'$ such that $f \circ h = h'$.

 The identity and the rest of the data of a symmetric monoidal category are defined analogously to Definition~\ref{defCobDW}.
\end{Definition}

Henceforth, we shall simply speak of stratified surfaces and stratified cobordisms/3-manifolds whenever we refer to objects or morphisms in this category, where it is understood that they come equipped with the free boundary decompositions from Example~\ref{exBound}.

\paragraph{Free boundary extensions.} There is a canonical non-full inclusion functor $J\colon \Cob_3^{\mathsf{def}} \hookrightarrow \Cob_3^{\mathsf{def}}(\partial)$. In order to extend Dijkgraaf--Witten theory $Z_{\rm DW}\colon \Cob_3^{\mathsf{def}} \rightarrow \vect_{\C}$ to this larger cobordism category, we define a symmetric monoidal functor
$R\colon \Cob_3^{\mathsf{def}}(\partial) \rightarrow \Cob_3^{\mathsf{def}}$
in the other direction by which we can precompose $Z_{\rm DW}$. We formulate some natural criteria that such a functor ought to satisfy.
\begin{Definition}\label{defFBE}
 We say that a symmetric monoidal functor $R\colon \Cob_3^{\mathsf{def}}(\partial) \rightarrow \Cob_3^{\mathsf{def}}$ \textit{realizes a~free boundary extension} if it satisfies the following conditions:
 \begin{enumerate}[(F1)]\itemsep=0pt
 \item \label{F1} For any surface $\Sigma \in \Cob_3^{\mathsf{def}}(\partial)$, there is an embedding in $R(\Sigma)$ as a stratified manifold with defect data and for any cobordism $M \in \Cob_3^{\mathsf{def}}(\partial)(\Sigma_1, \Sigma_2)$, there is an embedding in $R(M)$ as a stratified manifold with defect data such that the diagram
 \[
 \centering
 \adjincludegraphics[valign=c, scale = 1.16]{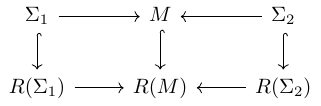}
 \]
 commutes, where the horizontal arrows are provided by the cobordism and its image under the functor $R$.
 \item \label{F2} For any surface $\Sigma$, the stratification on $R(\Sigma)$ satisfies $\Sigma^k = R(\Sigma)^k$ for $k \leq 1$
 with respect to the embedding $\Sigma \hookrightarrow R(\Sigma)$. The additional 2-strata in the case $R(\Sigma)\setminus\Sigma \neq \varnothing$ are all labeled by the trivial group $1$. Similarly, for any cobordism $M$, it holds that $M^k = R(M)^k$ for $k \leq 2$ with respect to $M \hookrightarrow R(M)$, with the additional 3-strata labeled by the trivial group.
 \end{enumerate}
\end{Definition}
\begin{Remark}\label{remFBE}\quad
\begin{enumerate}[(i)]\itemsep=0pt
 \item We can precompose any TQFT $\ Z\colon \Cob_3^{\mathsf{def}} \rightarrow \vect_{\C}$ with a functor $R$ as above to obtain a symmetric monoidal functor \smash{$Z^R := Z \circ R\colon \Cob_3^{\mathsf{def}}(\partial) \rightarrow \vect_{\C}$}. Note that we do not demand the natural condition \smash{$R\circ J \cong \id_{\Cob_3^{\mathsf{def}}}$}. The name `extension' will be justified by the fact that, nonetheless, the axioms imply that \smash{$Z_{\rm DW}^R|_{\Cob_3^{\mathsf{def}}} \cong Z_{\rm DW}$} for Dijkgraaf--Witten theory, as we see in Proposition~\ref{propEqExt}.
 \item Functoriality of $R$ implies that $R(\Sigma \times I) \!=\! R(\Sigma) \times I$ as cobordisms in $\allowbreak \Cob_3^{\mathsf{def}}(\partial)(R(\Sigma), R(\Sigma))$, since both 3-manifolds represent the class of the identity on $R(\Sigma)$.
 \item Concerning condition~\ref{F2}, it should be noted that a $G_{L(s)}$-set $M_s$ labeling a 1-stratum $s$ of a surface $\Sigma$ needs to be interpreted as a $G_{L(s)} \times 1^{\mathrm{op}}$-set in $R(\Sigma)$, with similar obvious modifications in the other cases. The condition is stricter than it could be: we could allow for more complicated stratifications of $R(\Sigma)\setminus\Sigma$, as long as they are labeled by trivial group data. The choice above is made for simplicity.
\end{enumerate}
\end{Remark}

We observe that Dijkgraaf--Witten theory is built in a way that extending by two different functors realizing a free boundary extension yields equivalent symmetric monoidal functors $\Cob_3^{\mathsf{def}}(\partial) \rightarrow \vect_{\C}$.
\begin{Proposition}\label{propEqExt}
 For any two symmetric monoidal functors $R, \widetilde{R}\colon \Cob_3^{\mathsf{def}}(\partial) \rightarrow \Cob_3^{\mathsf{def}}$ realizing a free boundary extension, there is a natural monoidal isomorphism \smash{$\eta\colon Z_{\rm DW}^R \Rightarrow Z_{\rm DW}^{\widetilde{R}}$}, where $Z_{\rm DW}^R$ is obtained by precomposing the Dijkgraaf--Witten TQFT $Z_{\rm DW}$ from~{\rm \cite{FMM}} with~$R$. Moreover, the restriction \smash{$Z_{\rm DW}^R|_{\Cob_3^{\mathsf{def}}}$} is monoidally isomorphic to $Z_{\rm DW}$.
\end{Proposition}
\begin{proof}
 Decomposing \smash{$Z_{\rm DW}\colon \Cob_3^{\mathsf{def}} \xrightarrow{C} \mathsf{Span}^{\mathsf{fib}}(\mathsf{repGrpd)} \xrightarrow{L} \vect_{\C}$} (see Section~\ref{subSecDW}), we only have to check the existence of a natural monoidal isomorphism $\alpha\colon C \circ R \Rightarrow C \circ \widetilde{R}$ and can then set $\eta = L\alpha$, because $L$ is monoidal.
 From the explicit characterization \cite[Proposition~5.12]{FMM} of the gauge groupoid, we see that there is a canonical isomorphism $\phi_X$ between the gauge groupoids \smash{$\mathcal{A}^{R(X)}\sslash \mathcal{G}^{R(X)}$} and \smash{$\mathcal{A}^{\widetilde{R}(X)}\sslash \mathcal{G}^{\widetilde{R}(X)}$} for $X$ an object or a cobordism in $\Cob_3^{\mathsf{def}}(\partial)$, mapping a gauge configuration $(A, h)$ of~$R(X)$ to a gauge configuration \smash{$\bigl(\widetilde{A}, \widetilde{h}\bigr)$} of \smash{$\widetilde{R}(X)$} defined as follows. Choosing embeddings $X \subset R(X)$ and $X \subset \widetilde{R}(X)$ as stratified manifolds with defect data by~\ref{F1}, for any $n$- or $(n-1)$-stratum $t$ contained in $X$, we can identify components $A_t = \widetilde{A}_t$ or $h_t = \widetilde{h}_t$, respectively. By~\ref{F2}, for an $n$-stratum $t_R \subset R(X) \setminus X$, the component $A_{t_R}\colon \Pi_1(t_R) \rightarrow \bullet \sslash 1$ is necessarily trivial (independent of the topology of $t_R$), and for the same reason components of \smash{$\bigl(\widetilde{A}, \widetilde{h}\bigr)$} associated to $n$-strata $t_{\widetilde{R}} \subset \widetilde{R}(X) \setminus X$ must be chosen to be trivial. The same logic applies to gauge transformations.
 So gauge data for $R(X)$ and $\widetilde{R}(X)$ are in one-to-one correspondence, and the functor $\phi_X$ provides the obvious match.

 Furthermore, strata of codimension 2 and 3 only appear in the image of the embedding of $X$ and associated defect data agree by~\ref{F2} in $R(X)$ and $\widetilde{R}(X)$. Thus, for any surface $\Sigma \in \Cob_3^{\mathsf{def}}(\partial)$ and any cobordism $M$, one has
 \begin{align}
 \label{eqQD}
 F_{R\Sigma} = F_{\widetilde{R}\Sigma} \circ \phi_{\Sigma}, \qquad \mu_{\widetilde{R}M} \circ \phi_M = \mu_{RM}.
 \end{align}
 We need to define the components \smash{$\alpha_{\Sigma}\colon \mathcal{A}^{R(\Sigma)}\sslash \mathcal{G}^{R(\Sigma)} \rightarrow \mathcal{A}^{\widetilde{R}(\Sigma)}\sslash \mathcal{G}^{\widetilde{R}(\Sigma)}$} of the natural transformation $\alpha\colon C \circ R \Rightarrow C \circ \widetilde{R}$ as morphisms in $\mathsf{Span}^{\mathsf{fib}}(\mathsf{repGrpd)}$ (cf. Definition~\ref{defFibSpan}). Consider the upper and lower spans $s$ and $\widetilde{s}$ in the diagram
 \[
 \centering
 \adjincludegraphics[valign=c, scale = 1.16]{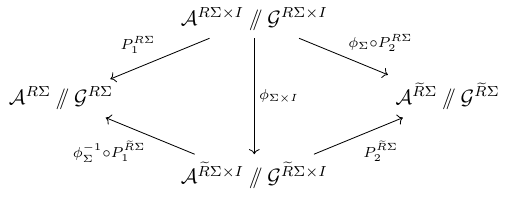}
 \]
 paired with the natural transformations \smash{$F_{R\Sigma} \circ P_1^{R\Sigma} \xRightarrow{\mu_{R\Sigma \times I}} F_{R\Sigma} \circ P_2^{R\Sigma} = F_{\widetilde{R}\Sigma} \circ \phi_{\Sigma}\circ P_2^{R\Sigma}$} and \smash{$F_{R\Sigma} \circ \phi_{\Sigma}^{-1} \circ P_1^{\widetilde{R}\Sigma} = F_{\widetilde{R}\Sigma} \circ P_1^{\widetilde R\Sigma} \xRightarrow{\mu_{\widetilde R\Sigma \times I}} F_{\widetilde{R}\Sigma} \circ P_2^{\widetilde R\Sigma}$}, respectively. Here, the projection functors \smash{$P_i^{R\Sigma}$} come from the inclusions of $R\Sigma$ into the identity cobordism $R\Sigma \times I$ (and the same for $\widetilde{R}$). Since $\phi_{\Sigma}$ is an isomorphism of categories and $s$ and $\widetilde{s}$ arise by composing $\phi_{\Sigma}$ with one of the legs of a span coming from a morphism in the image of $C$, the spans are indeed fibrant. Due to~\ref{F1}, the diagram commutes and by~\eqref{eqQD} it follows that $(s, \mu_{R\Sigma \times I})$ and \smash{$(\widetilde s, \mu_{\widetilde R\Sigma \times I})$} represent the same morphism in $\mathsf{Span}^{\mathsf{fib}}(\mathsf{repGrpd)}$. We therefore set \smash{$\alpha_{\Sigma} = [s, \mu_{R \Sigma \times I}]$}.

 Due to \eqref{eqQD}, we essentially only need to check naturality for $\alpha$ on the level of spans. The naturality square for a cobordism $M\colon \Sigma_1 \rightarrow \Sigma_2$ is then given by
 \[
 \centering
 \adjincludegraphics[valign=c, scale = 1.16]{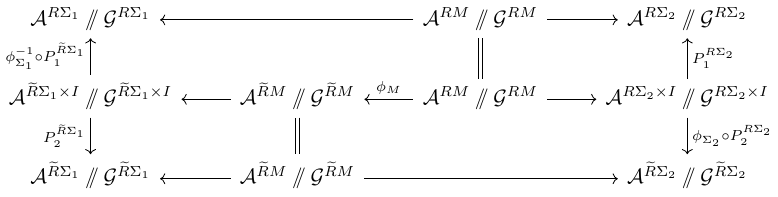}
 \]
 Note that the left and right vertical arrows come from spans representing $\alpha_{\Sigma_1}$ and $\alpha_{\Sigma_2}$, respectively. The upper right and the lower left squares are both pullback diagrams, since $R\Sigma_1 \times I$ is the identity cobordism. They represent the composition of the classes of the top and right, and left and bottom spans, respectively. The upper left and the lower right square commute due to condition~\ref{F1}, since restricting gauge configurations or transformations happens by restricting from $M \subset R(M)$ to $\Sigma_i \subset R(\Sigma_i)$ and $M \subset \widetilde R(M)$ to $\Sigma_i \subset \widetilde R(\Sigma_i)$ along the same embedding $\Sigma_i \hookrightarrow M$, and gauge data in the complements are trivial. Hence, the two composed spans are equivalent, proving the naturality of $\alpha$. The monoidality of $\alpha$ is a consequence of the fact that $\phi_{RX \, \sqcup \,RX'} = \phi_{RX} \times \phi_{R X'}$ and $F_{R\Sigma \, \sqcup \, R\Sigma'} = F_{R\Sigma} \otimes F_{R \Sigma'}$ under the identification
 \begin{align*}
 \cala^{RX \sqcup RX'} \sslash \calg^{RX \sqcup RX'} \cong \cala^{RX} \sslash \calg^{RX} \times \cala^{RX'} \sslash \calg^{RX'}.
 \end{align*}
 The last assertion follows by showing that $C \circ R \circ J \cong C$, using arguments very similar to those presented above.
\end{proof}

We show existence of a functor $\Cob_3^{\mathsf{def}}(\partial) \rightarrow \Cob_3^{\mathsf{def}}$ realizing a free boundary extension by fixing an explicit construction, based on the double of a manifold.
\begin{itemize}\itemsep=0pt
 \item For a surface $\Sigma$ with possibly non-empty boundary, we set $\Sigma_c = \Sigma \cup_{\partial \Sigma} \overline{\Sigma}$, where we glue along the identity morphism $\id\colon \partial \Sigma \rightarrow \partial \overline{\Sigma}$. The resulting manifold $\Sigma_c$ is oriented and closed. The stratification of $\Sigma_c$ is as demanded in Definition~\ref{defFBE}: It consists of the stratification in $\Sigma \subset \Sigma_c$ and the trivial stratification in $\overline{\Sigma}\setminus \partial \Sigma \subset \Sigma_c$, which means that there are only additional 2-strata after gluing. The defect labeling is such that strata contained in $\Sigma \subset \Sigma_c$ carry the same labels as before (cf.\ Remark~\ref{remFBE}) and the additional 2-strata in $\overline{\Sigma}$ are each labeled by the trivial group $1$.
 \item For $M\colon \Sigma_1 \rightarrow \Sigma_2$ a cobordism in the category $\Cob_3^{\mathsf{def}}(\partial)$, we set
 \begin{align*}
 M_c = M \cup_{\mathrm{cl}(\partial_fM)} \overline{M},
 \end{align*}
 where $\mathrm{cl}(\partial_fM)$ denotes the closure of the free boundary $\partial_fM$ in $M$. Again, we demand that on the piece $\overline{M} \setminus \mathrm{cl}(\partial_f M) \subset M_c$ the stratification is trivial, with 3-strata labeled by the trivial group $1$, and that the stratification stays the same otherwise. The result of this procedure is a cobordism $M_c\colon \left(\Sigma_1\right)_c \rightarrow \left(\Sigma_2\right)_c$.
\end{itemize}

\begin{Proposition} \label{propCobF}
 The assignment $(-)_c\colon \Cob_3^{\mathsf{def}}(\partial) \rightarrow \Cob_3^{\mathsf{def}}$ is a symmetric monoidal functor realizing a free boundary extension as in Definition~{\rm \ref{defFBE}}.
\end{Proposition}
\begin{proof}
It is clear that $(-)_c$ respects composition and the symmetric structure given by disjoint union. Moreover, $\varnothing_c = \varnothing$, so the assignment is unital. Consider the identity morphism $M = \Sigma \times I\colon \Sigma \rightarrow \Sigma$ for a surface $\Sigma$. Then, $\partial_fM = \partial \Sigma \times I$ and we obtain that $M_c = \Sigma_c \times I$, which represents the identity morphism on $\Sigma_c$. Thus, $(-)_c$ is a symmetric monoidal functor. It is moreover clear by construction that the two properties in Definition~\ref{defFBE} are satisfied.
\end{proof}

Note that $(-)_c \circ J$ is certainly not isomorphic to $ \id_{\Cob_3^{\mathsf{def}}}$, as can be seen by evaluating it on any nonempty closed surface.
\begin{Example}\label{exCompl}
Consider the 3-ball with a disk embedded in its boundary, $\mathbb{D}^2 \hookrightarrow \partial B^3$. We can regard the ball as a cobordism $B^3\colon \mathbb{D}^2 \rightarrow \varnothing$ in $\Cob_3^{\mathsf{def}}(\partial)$ with the complement of the embedded disk in the boundary as the only 2-stratum and trivially stratified interior. The interior is labeled by a finite group $G$ and the boundary stratum by a $G$-set $M$. Then \smash{$\bigl(B^3\bigr)_c$} is a 3-ball with an interior 2-stratum homeomorphic to a disk whose boundary is glued to the equator, labeled by $M$ and separating two regions labeled by $G$ and $1$, respectively. This defines a cobordism $\mathbb{S}^2 \rightarrow \varnothing$ in $\Cob_3^{\mathsf{def}}$ from the 2-sphere with its equator as a 1-stratum and induced labels to the empty manifold.
\end{Example}

We can now extend Dijkgraaf--Witten TQFT $Z_{\rm DW}$ to the defect cobordism category that allows free boundaries along the functor $(-)_c\colon \Cob_3^{\mathsf{def}}(\partial) \rightarrow \Cob_3^{\mathsf{def}}$ from Proposition~\ref{propCobF}.
\begin{Definition}\label{defExtDW}
We define the extension of Dijkgraaf--Witten TQFT to the category $\Cob_3^{\mathsf{def}}(\partial)$ by $Z_{\rm DW}^{\partial} := Z_{\rm DW}^{(-)_c} = Z_{\rm DW} \circ (-)_c\colon \Cob_3^{\mathsf{def}}(\partial) \rightarrow \vect_{\C}$.
\end{Definition}

We had already seen in Proposition~\ref{propEqExt} that restriction to $\Cob_3^{\mathsf{def}}$ recovers the functor $Z_{\rm DW}$.

\paragraph{The subcategory $\Cob_3^G$.} Because we are interested in defect TQFTs that allow defects only on the free boundary, we need to restrict Dijkgraaf--Witten theory to a suitable subcategory of decorated cobordisms. In the following, let $G$ be a finite group. Consider the non-full monoidal subcategory $\Cob_3^G \subset \Cob_3^{\mathsf{def}}(\partial)$ defined as follows:
\begin{itemize}\itemsep=0pt
 \item Objects are given by stratified surfaces $\Sigma$ with trivial stratification in the interior $\mathrm{Int}(\Sigma)$, where all 2-strata are labeled by the group $G$. Each 1-stratum (necessarily contained in the boundary $\partial \Sigma$) is labeled by the regular $G$-set $_GG$.
 \item Morphisms are given by stratified defect cobordisms $M\colon \Sigma_1 \rightarrow \Sigma_2$ with a trivially stratified interior and all 3-strata labeled by $G$. The free boundary $\partial_fM$ can be stratified and each 2-stratum (contained in the free boundary $\partial_f M$) is labeled by the regular $G$-set.
\end{itemize}

We add two further minor, technical restrictions.
First, no isolated loops are allowed on the free boundaries of cobordisms in $\Cob_3^G$ (cf. footnote~\ref{footNoLoops}). Second, we specify that the representations $\rho_e \colon \cald_e \rightarrow \vect_{\C}$ attached to edges $e$ (and boundary vertices) and the natural transformations $\sigma_v \colon \rho_v^t \Rightarrow \rho_v^s$ at vertices $v$ are constant on isomorphism classes of $\cald_e$ and $\cald_v$, respectively.

Comparing this cobordism category with the cobordism category $\Cob_3^{\vect_G}$ from Section~\ref{subSecTQFTs} with labels in the spherical fusion category $\vect_G$ yields the expected relationship.
\begin{Lemma}\label{lemEqCob}
 There is a symmetric monoidal equivalence $\Theta\colon \Cob_3^{\vect_G} \xrightarrow{\simeq} \Cob_3^G$ for every finite group $G$, where the category $\Cob_3^{\vect_G}$ is the one defined in Definition~{\rm \ref{defCobDef}}.
\end{Lemma}

The proof of the lemma will be referenced repeatedly and used as a dictionary in Section~\ref{subSecDWProp}.
\begin{proof}
 We first describe how the functor $\Theta$ acts on the topological level. A surface $\Sigma \in \Cob_3^{\vect_G}$ with $\vect_G$-colored points on $\partial \Sigma$ is mapped to the same manifold $\Sigma$, seen as a stratified surface with the obvious stratification: The vertices are the colored points, the edges are the connected components of $\partial \Sigma$ with the colored points removed, and the 2-strata are the connected components of $\mathrm{Int}(\Sigma)$. The stratification is homogeneous and saturated (in the sense of Definition~\ref{defFBHomStr}), the orientations of the vertices carry over from the orientation of colored points, and the orientation of the edges is specified by the orientation of the boundary. There is a unique choice of framing for the vertices.

 Consider now a cobordism $M\colon \Sigma_1 \rightarrow \Sigma_2$ with a $\vect_G$-colored graph $\Gamma$ on its free boundary. The manifold underlying the stratified cobordism is sent by $\Theta$ to itself. Taking the vertices $v \in V(\Gamma)$ as the 0-strata, the edges $e \in E(\Gamma)$ as the 1-strata and the plaquettes $p \in P(\Gamma)$ as the 2-strata in $M$, we obtain together with the connected components of $\mathrm{Int}(M)$ as 3-strata a stratification of $M$, which, similar to the case of surfaces, is homogeneous, saturated, oriented and framed.

 On the level of algebraic defect labels, the functor $\Theta$ assigns the group $G$ to $\mathrm{Int}(M)$ and $\mathrm{Int}(\Sigma)$, respectively, and the regular $G$-set to any plaquette $p$ of $M$ or edge of $\Sigma$. We now explain what $\Theta$ assigns to codimension-2 and -3 strata.

 The 1-strata $e \in S_1^{\Theta(M)}$ need to be labeled by representations $\rho_e\colon \mathcal{D}_e = G \times G \sslash G \rightarrow \vect_{\C}$, where in the action groupoid $\cald_e$, the group $G$ acts diagonally on the factors from the left. Any edge $e \in E(\Gamma)$ comes with a $G$-graded vector space $V^e = \bigoplus_{g \in G}V_g^e$, defining a functor $\underline{G} \rightarrow \vect_{\C}$ by the equivalence $\vect_G \simeq \mathsf{Fun}\left(\underline{G}, \vect_{\C}\right)$. Precomposing with the equivalence $\pi\colon G \times G \sslash G \rightarrow \underline{G}, (g_{\ell(e)}, g_{r(e)}) \mapsto g_{r(e)}^{-1}g_{\ell(e)}$, where $\ell(e)$ and $r(e)$ mark the left and right plaquette of $e$, we obtain a representation for any edge $e$ and edge label $V^e \in \vect_G$ by
 \begin{align*}
 \rho_e\colon \ \mathcal{D}_e \rightarrow \vect_{\C}, \qquad (g_{\ell(e)}, g_{r(e)}) \mapsto V_{g_{r(e)}^{-1}g_{\ell(e)}}^e,
 \end{align*}
 which the functor $\Theta$ assigns to the 1-stratum $e$ in $\Theta(M)$. Note that this representation is constant on classes in $\pi_0(\cald_e)$. This works in complete analogy for representations labeling the 0-strata of a surface $\Theta(\Sigma)$.

 Let now $v \in S_0^{\Theta(M)}$ denote a 0-stratum. The associated action groupoid is given by $\mathcal{D}_v = \bigl(\prod_{p \in P(v)}G\bigr) \sslash G$. The functors $D_{s(e)}\colon \mathcal{D}_{v} \rightarrow \mathcal{D}_e$ and $D_{t(e')}\colon \mathcal{D}_{v} \rightarrow \mathcal{D}_{e'}$ are the obvious projections associated to local strata $s(e)\colon v \rightarrow e$ or $t(e')\colon v \rightarrow e'$. The functor $\Theta$ is required to assign a~natural transformation $\sigma_v\colon \rho_v^t \Rightarrow \rho_v^s$, where the functor $\rho_v^s$ (similarly $\rho_v^t$) is assembled as the following tensor product, indexed by local strata associated to starting edge ends:
 \begin{align*}
 \rho_v^s = \bigotimes_{s(e)\colon v \rightarrow e}\rho_e \circ D_{s(e)}\colon \ \mathcal{D}_v \rightarrow \vect_{\C}.
 \end{align*}
 This is done as follows. The vertex in $V(\Gamma)$ corresponding to $v$ comes together with a labeling morphism $\alpha_v$, for which, due to pivotality, we can assume without loss of generality
 \begin{align*}
 \alpha_v \in \vect_G\biggl(\C, \bigotimes_{e \in E(v)}V^{\varepsilon(e)}\biggr).
 \end{align*}
 Here, the order in the tensor product is induced by the orientation of the boundary $\partial M$ and is fixed up to cyclic permutations. The notation convention is that $V^{\varepsilon(e)} = V^e$ in the case that $e$ is oriented away from $v$ and $V^{\varepsilon(e)} = (V^e)^*$ in the case it is oriented toward $v$. Here, $\varepsilon(e)$ depends on the vertex $v$, so we sometimes write $\varepsilon(v, e)$ instead of $\varepsilon(e)$ to avoid confusion.
 The morphism $\alpha_v$ defines a natural transformation $\widetilde{\sigma}_v\colon \C \Rightarrow \bigotimes_{e \in E(v)} \rho_e^{\varepsilon(e)}\circ D_{\varepsilon(e)}$ between two functors $\cald_v \rightarrow \vect_{\C}$, where $\C$ denotes the constant functor, as follows. First, as a notational tool, define the subset
 \begin{align*}
 \Omega_v = \bigg\{h \in G^{E(v)} \, \bigg| \prod_{e \in E(v)}h_e^{\varepsilon(v, e)} = 1_G \bigg\} \subset G^{E(v)},
 \end{align*}
 where the elements are multiplied from left to right with respect to the cyclic order on $E(v)$ and the convention is that $h_e^{\varepsilon(e)} = h_e$ for $e$ oriented away from $v$ and $h_e^{\varepsilon(e)} = h_e^{-1}$ in the other case. The morphism $\alpha_v$ can be decomposed as \smash{$\alpha_v = \sum_{h \in \Omega_v}\alpha_h^v$}, where \smash{$\alpha_h^v\colon \C \rightarrow \bigotimes_{e \in E(v)} \bigl(V_{h_e}^e \bigr)^{\varepsilon(e)}$} with \smash{$\bigl(V_{h_e}^e\bigr)^{\varepsilon(e)} = \bigl(V^{\varepsilon(e)}\bigr)_{h_e^{\varepsilon(e)}}$}. With the map \smash{$\varphi_v\colon G^{P(v)} \rightarrow \Omega_v, \varphi_v(g) = ( \pi \circ D_{\varepsilon(e)}(g) )_{e \in E(v)}$}, we~define the components of $\widetilde{\sigma}_v$ for $g \in G^{P(v)}$ as $(\widetilde{\sigma}_v)_g = \alpha_{\varphi_v(g)}^v$.
 Now note that due to the pivotal symmetric monoidal structure on \smash{$\vect_{\C}^{\cald_v} = \Fun(\mathcal{D}_v, \vect_{\C})$} inherited from $\vect_{\C}$,
 \begin{align*}
 \mathrm{Nat}\biggl(\C, \bigotimes_{e \in E(v)} \rho_e^{\varepsilon(e)}\circ D_{\varepsilon(e)}\biggr) \cong \mathrm{Nat}\bigl(\C, \rho_v^s \otimes \bigl(\rho_v^t\bigr)^*\bigr) \cong \mathrm{Nat}\bigl( \rho_v^t, \rho_v^s \bigr).
 \end{align*}
 We let $\sigma_v$ be the image of $\widetilde{\sigma}_v$ under this isomorphism. This fixes the assignment of defect data for all codimensions.

 It is clear that $\Theta$ defines a symmetric monoidal functor. That $\Theta$ is an equivalence of categories requires explicit verification only on the level of defect data. This follows due to representations and natural transformations labeling codimension-2 and codimension-3 strata in $\Cob_3^G$ being constant on isomorphism classes, and the fact that the map $\varphi_v\colon G^{P(v)} \rightarrow \Omega_v$ induces a bijection $\Omega_v \cong G^{P(v)}/G = \pi_0(\cald_v)$. In particular, each natural transformation $\sigma_v$ labeling a 0-stratum that is constant on isomorphism classes determines a unique morphism $\alpha_v$ in $\vect_G$ by reversing the assignment from before.
\end{proof}

Most of the time, we will not bother typing out the action of $\Theta$ and simply write $X$ for $\Theta(X)$.

We now define Dijkgraaf--Witten theory on the cobordism category $\Cob_3^{\vect_G}$ by pulling back the restriction of $Z_{\rm DW}^{\partial}$ from Definition~\ref{defExtDW} to the subcategory $ \Cob_3^G$ along the equivalence $\Theta$, which gives a symmetric monoidal functor
\begin{align*}
 Z_G\colon \ \Cob_3^{\vect_G} \rightarrow \vect_{\C}.
\end{align*}
This defines a boundary-defect TQFT in the sense of Definition~\ref{defTQFTDef}, which we again refer to as defect Dijkgraaf--Witten theory. By Proposition~\ref{propEqExt}, restriction of $Z_G$ to the monoidal subcategory $\Cob_3 \subset \Cob_3^{\vect_G}$ yields Dijkgraaf--Witten TQFT without defects.

For the remainder of the article, we restrict our attention to $Z_G$.

\subsection{Boundary locality of Dijkgraaf--Witten theory}
\label{subSecDWProp}
In the following, our goal is to verify boundary locality of the TQFT $Z_G\colon \Cob_3^{\vect_G} \rightarrow \vect_{\C}$. We start by showing property \ref{P1} for $Z_G$, which is the first assumption required to apply Theorem~\ref{thmMain}. Recall that $\dim (\vect_G) = |G|$.
\begin{Proposition}
 \label{propDW1}
 Let $M \in \Cob_3^{\vect_G}(\varnothing, \varnothing)$ and denote by $\Delta_1(M)$ the cobordism resulting from an application of the pocket move $\Delta_1$. Then, $Z_G(\Delta_1(M)) = |G| \cdot Z_G(M)$.
\end{Proposition}
\begin{proof}
 As in the proof of Proposition~\ref{propProp1TV}, by the axioms of a TQFT, it is enough to compare~the situation locally. Recall from Proposition~\ref{propProp1TV} the two cobordisms $B\colon \varnothing \rightarrow \mathbb{S}^2$ and $C\colon \varnothing \rightarrow \mathbb{S}^2$, topologically represented by a 3-ball and a cylinder over the 2-sphere, respectively. Showing the equality $Z_G(C) = |G|\cdot Z_G(B)$ then yields the proof of the proposition.

 We compute $Z_G(B)$ and $Z_G(C)$ directly. Since $B$ has trivial free boundary, we have that $B_c = B \sqcup B$, where the first copy is labeled by $G$ and the second by the trivial group. The gauge groupoid of the second component is trivial, so we simply have $ \allowbreak \mathcal{A}^{B_c} \sslash \mathcal{G}^{B_c} \simeq \mathcal{A}^{{B_c} \bullet} \sslash \mathcal{G}^{{B_c} \bullet} \cong \mathcal{A}^{{B} \bullet} \sslash \mathcal{G}^{{B} \bullet}$. For the connected manifold $B$, there is only one stratum, namely the 3-stratum $s$ such that $\hat{s} = B^3$, and we may choose a set of coherent basepoints $V$ by picking a single point $p \in \partial B$. Then, we have by Proposition~\ref{propGG}:
 \begin{itemize}\itemsep=0pt
 \item Gauge configurations are given by functors $A\colon \Pi_1(\hat{s}, p) \rightarrow \bullet \sslash G$. Since the fundamental group of $B^3$ is trivial, we have that $\Pi_1(\hat{s}, p) \cong \bullet \sslash \pi_1\bigl(B^3, p\bigr) = \bullet$ is the trivial groupoid, so~$A$ is the trivial constant functor.
 \item Gauge transformations are given by maps $\gamma\colon \{p\} \rightarrow G$, which are simply group elements.
 \end{itemize}
 It follows that \smash{$\mathcal{A}^B \sslash \mathcal{G}^B \cong \bullet \sslash G$}. By the same reasoning, we obtain that ${\smash{\mathcal{A}^{\partial B} \sslash \mathcal{G}^{\partial B}} \cong \bullet \sslash G}$. Restricting configurations and transformations to the boundary gives the identity functor $P_{\partial} = \id\colon \smash{\mathcal{A}^B \sslash \mathcal{G}^B \rightarrow \mathcal{A}^{\partial B} \sslash \mathcal{G}^{\partial B}}$, so we obtain as the image of $B\colon \varnothing \rightarrow \Sp^2$ under the functor $C \circ (-)_c \circ \Theta\colon \Cob_3^{\vect_G} \rightarrow \mathsf{Span}^{\mathsf{fib}}(\mathsf{repGrpd)}$ the span
 \[
 \centering
 \adjincludegraphics[valign=c, scale = 1.16]{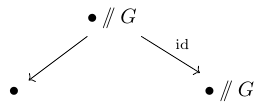}
 \]
 representing a morphism in the subcategory $\mathsf{Span}^{\mathsf{fib}}(\mathsf{Grpd)}$. By Proposition~\ref{propFunL}, $Z_G(\varnothing) = \C$ and $Z_G\bigl(\Sp^2\bigr) = \bigl\langle\pi_0\bigl(\mathcal{A}^{\partial B} \sslash \mathcal{G}^{\partial B}\bigr)\bigr\rangle_{\C} \cong \C$, and $Z_G(B) \in \C$ can, by \eqref{eqSpGrpd}, be expressed as
 \begin{align*}
 Z_G(B) = \bigl|\mathrm{PC}_{\bullet}\bigl(\mathcal{A}^{\partial B} \sslash \mathcal{G}^{\partial B}\bigr)\bigr| \cdot \bigl|\bigl\{\bullet \bigl| \mathcal{A}^{ B} \sslash \mathcal{G}^{ B} \bigr| \bullet\bigr\}\bigr| = |\bullet \sslash G| \cdot |\mathrm{Fib}((\bullet, \id), \bullet)| = |G|^{-1}.
 \end{align*}
 We move on to compute the invariant associated to $C$. The cobordism $C$ admits a free boundary homeomorphic to a 2-sphere, and so $D := C_c\colon \varnothing \rightarrow \Sp^2 \sqcup \Sp^2$ is a cylinder $\Sp^2 \times I$, where the two boundary components constitute the gluing boundary homeomorphic to two 2-spheres. There are two 3-strata $s_1$, $s_2$ with $\hat{s}_1 \cong \hat{s}_2 \cong \Sp^2 \times I$, labeled by $G$ and $1$, respectively. The 2-stratum $t$ that separates $s_1$ and $s_2$, which is homeomorphic to a 2-sphere, is again labeled by the $G \times 1^{\mathrm{op}}$-set $G$. We choose the set of coherent basepoints $V = \{p,q_1, q_2\}$, where $p$ is a point on $t$ and $q_1, q_2 \in \partial D$ are points in the boundary, with $q_1$ lying in the component of the gluing boundary labeled by $G$ and $q_2$ lying in the trivial component. From the description of $\mathcal{A}^{D} \sslash \mathcal{G}^{D}$, it is clear that we can disregard components of gauge configurations or transformations in $\hat{s}_2$. We find
 \begin{itemize}\itemsep=0pt
 \item Gauge configurations are given by pairs $(A, h)$ of functors $A\colon \Pi_1\left(\hat{s}_1, \{p,q_1\}\right) \rightarrow \bullet \sslash G$ and a group element $h \in G$. Since any loop in $\hat{s}_1$ is null-homotopic, the functor $A$ acts trivially on loops. As a consequence, there is no further condition on the pair $(A,h)$ and $A$ corresponds to the group element $a$ it assigns to the homotopy class of a path $\tau\colon [0,1] \rightarrow \hat{s}_1$ with $\tau(0) = q_1$, $\tau(1) = p$.
 \item Gauge transformations $(a, h) \Rightarrow (a', h')$ are given as pairs $(\gamma, \delta) \in G \times G$ associated to maps $\{p, q_1\} \rightarrow G$ such that the conditions $a' = \gamma a \delta^{-1}$ coming from the path $\tau$ and $h' = \gamma h$ coming from the point $p$ hold.
 \end{itemize}
 Altogether, we obtain $\mathcal{A}^{ D} \sslash \mathcal{G}^{ D} \cong G \times G \sslash G \times G$ defined by the above action. Restricting gauge data to the boundary $\partial D = \partial B$ gives the projection functor acting on morphisms as
 \begin{align*}
 P_{\partial}\colon \ \mathcal{A}^{ D} \sslash \mathcal{G}^{ D} \rightarrow \mathcal{A}^{ \partial D} \sslash \mathcal{G}^{\partial D} = \bullet \sslash G, \qquad \bigl((\gamma, \delta)\colon (a, h) \Rightarrow \bigl(\gamma a\delta^{-1}, \gamma h\bigr)\bigr) \mapsto \delta.
 \end{align*}
 As above for $B$, we obtain a span of groupoids in $\mathsf{Span}^{\mathsf{fib}}(\mathsf{Grpd)}$ given as
 \[
 \centering
 \adjincludegraphics[valign=c, scale = 1.16]{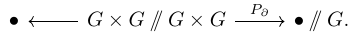}
 \]
 The fiber of $P_{\partial}$ is the groupoid $\bigl\{\bullet \bigl| \mathcal{A}^{ D} \sslash \mathcal{G}^{ D} \bigr| \bullet\bigr\} = \mathrm{Fib}(P_{\partial}, \bullet)$ equivalent to $G \times G \sslash G$ equipped with the diagonal action. This, in turn, is equivalent to the discrete groupoid $\underline{G}$. Thus, we get
 \begin{align}\label{eqSpanPocket}
 Z_G(C) = Z_{\rm DW}(D) = |\bullet \sslash G| \cdot |\mathrm{Fib}(P_{\partial}, \bullet)| = |G|^{-1} |G| = 1.
 \end{align}
 This gives the right comparison factor as before and finishes the proof.
\end{proof}

\begin{Remark}One has to be careful not to oversimplify the relevant groupoids. The gauge groupoid of $D$ is equivalent to the trivial groupoid. However, we cannot simply make this replacement, since the resulting span $\bullet \leftarrow \bullet \rightarrow \bullet \sslash G$ is not a fibrant span of groupoids (and is clearly not equivalent to the one from \eqref{eqSpanPocket} with respect to the equivalence relation from Definition~\ref{defFibSpan}). In fact, if one were to use this span instead, one would obtain a wrong result.
\end{Remark}
We continue by showing that $Z_G$ satisfies \ref{P2} with respect to the tunnel move.

\begin{Proposition}\label{propDW2}
 Let $M \in \Cob_3^{\vect_G}(\varnothing, \varnothing)$ be a cobordism and $\iota\colon C \hookrightarrow M$ a tunnel. For a~simple object $\C_g \in \vect_G$, consider the move $\Delta_2^{\C_g}(\iota)$ with respect to the tunnel $\iota$. Then,
 \begin{align}
 \label{eqDWTun}
 Z_G(M) = \sum_{g \in G}\dim (\C_g)Z_G\bigl(\Delta_2^{\C_g}(\iota)(M)\bigr)
 \end{align}
 holds, with $\dim (\C_g) = 1$ the dimension of the simple object.
\end{Proposition}

\begin{proof}
 One justifies as in the proof of Proposition~\ref{propProp2} that it is enough to compare the situation locally. Recall the two cobordisms $B$ and $\mathbb{T}_g^s = \mathbb{T}_{\C_g}^s$ from Lemma~\ref{lemProp2}, which we interpret here as cobordisms $B\colon A_1 \rightarrow \varnothing$ and $\mathbb{T}_g^s\colon A_1 \rightarrow \varnothing$ from the annulus $A_1$ to the empty manifold (this is simply done by flipping their orientations). Here, $B$ is a 3-ball and $\mathbb{T}_g^s\colon A_1 \rightarrow \varnothing$ is a solid torus with a non-contractible defect loop embedded in its boundary labeled by the object $\C_g$.
 In a first step, as usual, we have to apply $(-)_c \circ \Theta$. We obtain the following two cobordisms in $\Cob_3^{\mathsf{def}}$ as the images of $B$ and $\mathbb{T}_g^s$, respectively:
 \begin{itemize}\itemsep=0pt
 \item $T\colon \mathbb{T} \rightarrow \varnothing$, where $T \cong \mathbb{T}^s$ is a stratified solid torus viewed as a cobordism from its boundary torus $\partial \mathbb{T}^s = \mathbb{T}$ to the empty manifold. The stratification of $T$ consists of two 3-strata~$s_1$,~$s_2$ homeomorphic to open balls, where $s_1$ is labeled by $G$ and $s_2$ by the trivial group. There are two 2-strata~$t_1$,~$t_2$ homeomorphic to open disks, both labeled by the regular $G$-set $_G G$. The cobordism $T$ is depicted in the figure below, where the $G$-labeled strata are shaded in blue.
 \item $T_g\colon \mathbb{T} \rightarrow \varnothing$, a solid torus with two 3-strata homeomorphic to tori labeled by $G$ and $1$ respectively, two 2-strata shaped like annuli labeled by the $G$-set $_GG$, separated by a 1-stratum $e$ which is a non-contractible loop (with suppressed node labeled by the identity) and is labeled by the representation $\rho_e\colon G \times G \sslash G \rightarrow \vect_{\C}$ induced by $\C_g$ as in Lemma~\ref{lemEqCob}. The cobordism is shown below:
 \end{itemize}
 \begin{align*}
 T = \adjincludegraphics[valign=c, scale = 0.75]{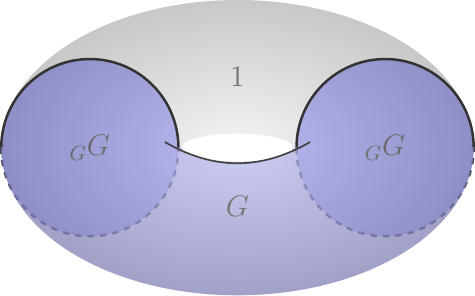}\,, \qquad T_g = \adjincludegraphics[valign=c, scale = 0.75]{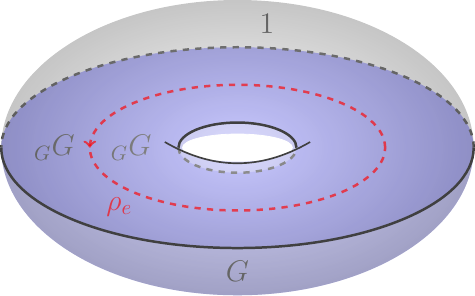}\,.
 \end{align*}
 In fact, the value of $Z_{\rm DW}$ on the cobordism $T_g$ has been computed in greater generality in \cite[Example~7.10]{FMM}. Here, we have that $Z_{\rm DW}(\partial T_g) = \bigoplus_{h \in G} \C h$ (as we will see explicitly later in the proof when we consider $T$, which has the same boundary) and the example specializes to the linear map defined on basis elements as
 \begin{align*}
 Z_{\rm DW}(T_g)\colon Z_{\rm DW}(\partial T_g) \rightarrow \C , \qquad h \mapsto \tr_{(\C_g)_h}(\id) = \delta_{g,h},
 \end{align*}
 taking into account the fact that simple objects in $\vect_G$ are one-dimensional. Thus, we can evaluate the right-hand side of the local version of \eqref{eqDWTun} applied to $h \in G$ as
 \begin{align}
 \label{eqSoS}
 \sum_{g \in G} \dim (\C_g)Z_G(\mathbb{T}_g^s)(h) = \sum_{g \in G} \delta_{g,h} = 1.
 \end{align}
 We turn to the computation of the left-hand side $Z_{\rm DW}(T)$. As always, we need to determine the gauge groupoids for bulk and boundary. As every loop contained in $\hat{s}_1$ is contractible, $\mathcal{A}^T \sslash \mathcal{G}^T$ easily follows from the boundary groupoid $\mathcal{A}^{\partial T} \sslash \mathcal{G}^{\partial T}$, which we compute in the following. In fact, these groupoids will turn out to be equal. We choose a set of coherent basepoints $V = \{p, q\}$ for the induced stratification on $\partial T$ with $p \in t_1^{\partial}$, $q \in t_2^{\partial}$, where $t_1^{\partial}, t_2^{\partial}$ denote the unique respective preimages of $t_1$, $t_2$ under the map $\iota_1\colon S_1^{\partial T} \rightarrow S_2^T$ that associates boundary to bulk strata. Similarly, denote by $s_i^{\partial}$ the preimages of $s_i$ under the map $\iota_2\colon S_2^{\partial T} \rightarrow S_3^T$. Topologically, $t_i^{\partial}$ are circles and $s_i^{\partial}$ is an (open) cylinder mantle. One has
 \begin{itemize}\itemsep=0pt
 \item Gauge configurations are given by functors $A\colon \Pi_1(\hat{s}_1^{\, \partial}, \{p,q\}) \rightarrow \bullet \sslash G$ and pairs $(h_p, h_q) \in G \times G$ associated to the respective points satisfying that for any closed path $\tau_i\colon [0,1] \rightarrow \hat{t}_i^{ \, \partial}$ with start- and endpoint $p$ or $q$, the conditions $h_p = A([\tau_1])h_p$ and $h_q = A([\tau_2])h_q$ hold, respectively (again, the components of gauge configurations indexed by $s_2^{\partial}$ are trivial). Thus, $A$ acts trivially on these loops. If $\tau$ denotes a path connecting $q$ to $p$ without looping around the cylinder $\hat{s}_1^{\, \partial}$, and $\sigma_p$ denotes a representative of a generator $[\sigma_p] \in \pi_1(\hat{t}_1^{\, \partial}) \cong \mathbb{Z}$, then the homotopy class of every path $\omega$ in $\hat{s}_1^{\, \partial}$ from $q$ to $p$ can be written as $[\omega] = [\sigma_p^k \star \tau]$ for some $k \in \mathbb{Z}$, whence $A([\omega]) = A([\sigma_p])^k \cdot A([\tau]) = A([\tau])$ and the functor $A$ corresponds to a group element $a \in G$.
 \item Gauge transformations $(a, h_p, h_q) \Rightarrow (a', h_p', h_q')$ correspond to pairs $(\gamma_p, \gamma_q) \in G \times G$ such that $a' = \gamma_pa\gamma_q^{-1}$, $h_p' = \gamma_ph_p$ and $h_q' = \gamma_q h_q$.
 \end{itemize}
 The restriction functor $P_{\partial}\colon \mathcal{A}^{\partial T} \sslash \mathcal{G}^{\partial T} \rightarrow \mathcal{A}^{ T} \sslash \mathcal{G}^{ T}$ is the identity. This together with the equivalence $\mathcal{A}^{\partial T} \sslash \mathcal{G}^{\partial T} \simeq \underline{G} $ allows us to replace the span of gauge groupoids by $\underline{G} \leftarrow \underline{G} \rightarrow \bullet$ and thus, we find $Z_{\rm DW}(\partial T) = \bigoplus_{h \in G} \C h$. The matrix components of the morphism $Z_{\rm DW}(T)$ (recall~\eqref{eqSpGrpd}) are given as
 \begin{align*}
 \langle h | Z_{\rm DW}(T) | \bullet \rangle = |P_{\bullet}(\bullet)| \cdot \bigl|\bigl\{h \big| \mathcal{A}^{ T} \sslash \mathcal{G}^{ T} \big| \bullet\bigr\} \bigr| = 1,
 \end{align*}
 hence $Z_G(B^3) = Z_{\rm DW}(T)$ is the linear map sending every $h \in G$ to $1$, so together with \eqref{eqSoS}, the equality asserted in the proposition follows.
\end{proof}

We proceed by showing~\ref{P4} for the defect TQFT $Z_G$ and postpone the proof of~\ref{P3} for the moment. The reason for this is that it is convenient for us to apply~\ref{P4} in the proof of~\ref{P3} (even though it is not strictly necessary).
\begin{Proposition}
\label{propDW4}
 Let $B_{\Gamma} \in \Cob_3^{\vect_G}(\varnothing, \varnothing)$ be a $3$-ball decorated with a $\vect_G$-labeled graph $\Gamma$ as introduced in~\eqref{eqBGamma}. Denoting by $\langle \Gamma \rangle \in \C$ the scalar associated to $\Gamma$ seen as an endomorphism of the tensor unit in the spherical fusion category $\vect_G$, it holds that $Z_G(B_{\Gamma}) = \langle \Gamma \rangle$.
\end{Proposition}
\begin{proof}
The result $\Sp_{\Gamma}$ of applying $\Theta \circ (-)_c$ to $B_{\Gamma}$ is a stratified 3-sphere. We describe its stratification and algebraic labeling in $\Cob_3^G$ to compute the gauge groupoid $\mathcal{A}^{\mathbb{S}_{\Gamma}} \sslash \mathcal{G}^{\mathbb{S}_{\Gamma}}$:
\begin{itemize}\itemsep=0pt
 \item There are two 3-strata $s_1$, $s_2$ homeomorphic to open 3-balls, where $s_1$ is labeled by the group $G$ and $s_2$ is labeled by the trivial group $1$.
 \item The 2-strata are given by the set $P = P(\Gamma)$ of plaquettes of the graph~$\Gamma$, each homeomorphic to $A_n^{\circ}$ for some $n \in \mathbb{N}$, an open disk with $n$ smaller disjoint open disks removed from the interior. Each plaquette is labeled by the regular $G \times 1^{\mathrm{op}}$-set~$G$.
 \item The 1-strata are in correspondence with the edges $e \in E(\Gamma)$, labeled by representations $\rho_e\colon G \times G \sslash G \times 1 \rightarrow \vect_{\C}$ constructed from the $G$-graded vector spaces $V^e$ labeling the edges of $B_{\Gamma}$, as provided in the proof of Lemma~\ref{lemEqCob}.
 \item The 0-strata are the vertices $v \in V(\Gamma)$, decorated with natural transformations $\sigma_v\colon \rho_v^s \Rightarrow \rho_v^t$ derived from the vertex labels $\alpha_v$ in~$\vect_G$ as in the proof of Lemma~\ref{lemEqCob}.
\end{itemize}
As always, we will compute the gauge groupoid of $\Sp_{\Gamma}$ by computing the reduced gauge groupoid $\mathcal{A}^{\mathbb{S}_{\Gamma} \bullet} \sslash \mathcal{G}^{\mathbb{S}_{\Gamma \bullet}}$. We choose as a set of coherent basepoints the set of vertices $V = V(\Gamma)$ and find the following (cf. Proposition~\ref{propGG}):
\begin{itemize}\itemsep=0pt
 \item Gauge configurations are given by pairs $(A, h)$ of functors $A\colon \Pi_1\left(\hat{s}_1, V\right) \rightarrow \bullet \sslash G$ and a family of maps $h_p\colon V(p) \rightarrow G$ for all $p \in P(\Gamma)$ such that for paths $\tau\colon [0,1] \rightarrow \hat{p}$ with $\delta(0) = v$, $\delta(1) = v'$, where $v, v' \in V(p)$, it holds that $h_p(v')= A([\iota_{\ell(p)} \circ \tau]) \cdot h_p(v)$. Although the map $\iota_{\ell(t)}$ is not necessarily injective, we will leave it out in the sequel to ease notation. This leaves the argument unaffected.
 \item Gauge transformations $(A, h) \Rightarrow (A', h')$ are given by maps $\gamma\colon V \rightarrow G$ such that for any path $\delta\colon [0,1] \rightarrow \hat{s}_1$ from $v$ to $v'$, where $v, v' \in V$, and for any plaquette $p \in P(\Gamma)$ with $w \in V(p)$, it holds that $A'([\delta]) = \gamma(v') \cdot A([\delta]) \cdot \gamma(v)^{-1}$ and $h_p'(w) = \gamma(w) \cdot h_p(w)$.
\end{itemize}
We will now show that \smash{$\mathcal{A}^{\mathbb{S}_{\Gamma} \bullet} \sslash \mathcal{G}^{\mathbb{S}_{\Gamma \bullet}} \simeq \bigl(\prod_{p \in P(\Gamma)}G\bigr) \sslash G =: G^P \sslash G$}, which relies crucially on the contractibility of the 3-ball. For this, we first prove that there always exists a gauge transformation $\gamma\colon (A,h) \Rightarrow (1_G, g)$, where $1_G\colon \Pi_1 (\hat{s}_1, V ) \rightarrow \bullet \sslash G$ denotes the constant functor sending every morphism to the trivial group element and $g = (g_p)_{p \in P}$ is a family of group elements $g_p \in G$, or, put differently, a family of constant functions $V(p) \rightarrow G$.

First, for a subset $S \subset P$ of the plaquettes, let
\begin{align*}
 W(S) := \{p \in P \mid V(p) \cap V(s) \neq \varnothing \text{ for some }s \in S\} \subset P,
\end{align*}
and, moreover, set $W^k(S) = W\bigl(W^{k-1}(S)\bigr)$. Choose a plaquette $p_0 \in P$ and a group element $g_{p_0} \in G$ arbitrarily. Since $\Gamma$ is a finite graph and its dual graph $\Gamma^*$ is always connected, it holds that $W^m(p_0) = P$ for some $m \in \mathbb{N}$, since sharing an edge for two plaquettes implies them having a common vertex. We define the family $g$ recursively: For every $q \in W^{k+1}(p_0) \setminus W^k(p_0)$, choose a neighboring plaquette $p \in W^k(p_0)$ and set $g_{q} = g_ph_p(v)^{-1}h_{q}(v)$, where $v$ is a vertex shared by~$p$ and~$q$. Due to $W^m(p_0) = P$, this defines a full family $g = (g_p)_{p \in P}$. Once we have confirmed well-definedness of the assignment, we would like to set $\gamma(w) := g_ph_p(w)^{-1}$, for which we have to make sure that $g_{p'}h_{p'}(w)^{-1} = g_ph_p(w)^{-1}$ whenever $p$ and $p'$ share the vertex $w$. This boils down to checking the following two cases:
\begin{itemize}\itemsep=0pt
 \item Suppose that well-definedness for $W^k(p_0)$ has already been verified and that the two group elements assigned to two elements $p, p' \in W^k(p_0)$ induce the one assigned to $q$, where the notation is as above. This means that $q$ shares a vertex $v$ with $p$ and a vertex~$v'$ with~$p'$. Assuming the relation $g_q = g_ph_p(v)^{-1}h_q(v)$, we define a path $\delta\colon [0,1] \rightarrow \cup_{t \in {W^k}(p_0)}\mathrm{cl}(t)$ with $\delta(0) = v$, $\delta(1) = v'$ such that $g_{q} = g_{p'}h_{p'}(v')^{-1}A([\delta])h_{q}(v)$:

 Choose a sequence of plaquettes $p = p_0, \dots, p_{\ell} = p'$, forming a path in the connected graph $\Gamma^*$ with $p_i$ and $p_{i+1}$ sharing the vertex $v_{i+1}$, and choose paths $\delta_i\colon [0,1] \rightarrow \hat{p_i}$ starting at $v_i$ and ending at $v_{i+1}$. Here, our notation is such that $v_0 = v$ and $v_{\ell+1} = v'$. Since $A([\delta_i])h_{p_i}(v_i) = h_{p_i}(v_{i+1})$, the relations imply that setting $\delta = \delta_{\ell} \star \dots \star \delta_0$ gives the desired path. As $\pi_1\left(\hat{s}_1\right)$ is trivial, the path $\delta$ is homotopy equivalent relative to its endpoints to a~path in $\mathrm{cl}(q)$ that lifts to a~path $\sigma$ with $\mathrm{im} (\sigma) \subset \hat{q}$ and $A([\sigma])h_q(v) = h_q(v')$. We find
 \begin{align*}
 g_q = g_{p'}h_{p'}(v')^{-1}A([\delta])h_{q}(v) = g_{p'}h_{p'}(v')^{-1}A([\sigma])h_{q}(v) = g_{p'}h_{p'}(v')^{-1}h_{q}(v'),
 \end{align*}
 showing that both $p$ and $p'$ induce the same group element for $q$.
 \item If $p$ and $p'$ induce the group elements assigned to two neighboring plaquettes $q$ and $q'$ sharing a vertex $v$, one needs to show the consistency relation $g_{q'} = g_q h_q(v)^{-1}h_{q'}(v)$. This is done similarly to the first case.
\end{itemize}
Now, with the above definition of $\gamma$, we also have the last remaining relation, namely
\begin{align*}
 \gamma(v')A([\tau])\gamma(v)^{-1} = g_ph_p(v')^{-1}A([\tau])h_p(v)g_p^{-1} = 1_G
\end{align*}
for any path $\tau\colon [0,1] \rightarrow \hat{s}_1$ from $v$ to $v'$. This proves that indeed, $(A,h) \cong (1_G, g)$, which says that every gauge configuration is equivalent to one that is constant on a plaquette. However, two such objects $(1_G, g)$ and $(1_G, \widetilde{g})$ may be isomorphic. That is the case if and only if $g_p(\widetilde{g}_{p})^{-1} = g_{p'}(\widetilde{g}_{p'})^{-1}$, that is, $g, \widetilde{g} \in \mathcal{O}$ for some orbit $\mathcal{O} \in G^P/G$. This proves $\mathcal{A}^{\mathbb{S}_{\Gamma}} \sslash \mathcal{G}^{\mathbb{S}_{\Gamma}} \simeq G^P \sslash G$.

Let now $\Gamma^{\mathbb{S}_{\Gamma}}$ denote the graph labeled in \smash{$\vect_{\C}^{\mathcal{A}^{\mathbb{S}_{\Gamma}} \sslash \mathcal{G}^{\mathbb{S}_{\Gamma}}}$} that defines the natural transformation $\mu_{\mathbb{S}_{\Gamma}}$ (cf.\ Section~\ref{subSecGG}). Since $\mathbb{S}^3$ is closed, Example~7.8 from \cite{FMM} applies and $\mu_{\mathbb{S}_{\Gamma}}\colon \C \Rightarrow \C$ corresponds to a map $\mathcal{A}^{\mathbb{S}_{\Gamma}} \rightarrow \C, h \mapsto \mu_h^{\mathbb{S}_{\Gamma}}$, where $h = (h_p)_{p \in P} \in G^{P}$, which is constant on isomorphism classes. The invariant can then be expressed as
\begin{align}
 \label{eqBallInv}
 Z_G(B_{\Gamma}) = Z_{\rm DW}(\mathbb{S}_{\Gamma}) = \sum_{[h] \in \pi_0(\mathcal{A}^{\mathbb{S}_{\Gamma}} \sslash \mathcal{G}^{\mathbb{S}_{\Gamma}})}\frac{\mu_h^{\mathbb{S}_{\Gamma}}}{|{\Aut(h)}|} = \frac{1}{|G|}\sum_{h \in G^P} \mu_h^{\mathbb{S}_{\Gamma}}.
\end{align}
We want to relate the labeling of the $\vect_{G}$-labeled graph $\Gamma$ to the one of the graph $\Gamma^{\mathbb{S}_{\Gamma}}$. For an edge $e \in E(\Gamma)$, recall the functors
\begin{align*}
 P_e^s \colon \ G^P\sslash G \simeq \mathcal{A}^{\mathbb{S}_{\Gamma}} \sslash \mathcal{G}^{\mathbb{S}_{\Gamma}} \xrightarrow{P_v} \mathcal{D}_v \xrightarrow{D_{s(e)}} \mathcal{D}_e = G \times G \sslash G,
\end{align*}
and $P_e^t = D_{t(e)} \circ P_{v'}$, where $s(e)\colon v \rightarrow e$ and $t(e)\colon v' \rightarrow e$ are the local strata defined by the vertices $v, v'$ where $e$ starts and ends. In general, the functor \smash{$P_v\colon \mathcal{A}^{\mathbb{S}_{\Gamma}} \sslash \mathcal{G}^{\mathbb{S}_{\Gamma}} \rightarrow \cald_v$} maps $(A, h)$ to $(h_p(v))_{p \in P(v)}$. Using the above equivalences replaces every gauge configuration by one constant on plaquettes, and so both functors~$P_e^s$,~$P_e^t$ are simply given as the obvious projections: A~(constant) plaquette labeling $(h_p)_{p \in P} \in G^P$ maps to the pair of group elements in $\cald_e$ given by the element $h_{\ell(e)}$ residing on the plaquette on the left of a given edge $e$ (with respect to the orientation of~$e$ and the plaquettes) and the group element $h_{r(e)}$ on the right plaquette. That is, both functors act as
\begin{align*}
 G^P \sslash G \rightarrow G \times G \sslash G, \qquad (h_p)_{p \in P} \mapsto \left(h_{\ell(e)}, h_{r(e)}\right),
\end{align*}
and we write $P^e := P_e^s = P_e^t$. Similarly, we see that the natural transformation $\rho_e P_e^d\colon \rho_e \circ P_e^s \Rightarrow \rho_e \circ P_e^t$ labeling the new bivalent vertices $v_e$ of $\Gamma^{\mathbb{S}_{\Gamma}}$ corresponds to the identity transformation. Thus, these vertices do not contribute to the scalar value of the graph as an endomorphism in the symmetric monoidal category \smash{$\vect_{\C}^{G^P\sslash G}$}, and we can simply ignore the additional bivalent vertices in the remainder of the proof. Recall that to a vertex $v \in V(\Gamma) \subset V\bigl(\Gamma^{\mathbb{S}_{\Gamma}}\bigr)$, the natural transformation $\mu_v =\sigma_v P_v\colon \rho_v^t \circ P_v \Rightarrow \rho_v^s \circ P_v$ is associated, where $\sigma_v$ is obtained from a morphism $\alpha_v \in \vect_G\bigl(\C, \bigotimes_{e \in E(v)}V^{\varepsilon(e)}\bigr)$ as in the proof of Lemma~\ref{lemEqCob}.
Each $\alpha_v$ can be identified with an element of its target and may be decomposed as
\begin{align*}
 \alpha_v = \sum_{g \in \Omega_v} \alpha_g^v \in \bigoplus_{g \in \Omega_v} \bigotimes_{e \in E(v)}\bigl(V^{\varepsilon(e)}\bigr)_{g_e^{\varepsilon(e)}} \subset \bigotimes_{e \in E(v)}V^{\varepsilon(e)},
\end{align*}
where we used the notation from Lemma~\ref{lemEqCob}. Similar to the subsets $\Omega_v \subset G^{E(v)}$, we define
\begin{align*}
 \Omega = \biggl\{(g_e)_{e \in E} \in G^E \, \bigg| \prod_{e \in E(v)} g_e^{ \varepsilon(v, e)} = 1_G \text{ for all } v \in V\biggr\} \subset G^E,
\end{align*}
which we can view as the set of `admissible' edge colorings, and we have an obvious family of projections $\pi_v^E\colon \Omega \rightarrow \Omega_v$. This allows us to decompose $\langle \Gamma \rangle = \sum_{g \in \Omega}\langle \Gamma \rangle_g$, where the components~$\langle \Gamma \rangle_g$ are built from the components~$\alpha_{\pi_v^E(g)}^v$ by tensoring and composition. The projection $\pi\colon G \times G \rightarrow G$, $(g,h) \mapsto h^{-1}g$ and the functor $P^e$ induce a map
\begin{align*}
 \varphi\colon \ G^P \rightarrow \Omega, \qquad h = (h_p)_{p \in P} \mapsto ((\pi \circ P^e)(h) )_{e \in E} = \bigl(h_{r(e)}^{-1}h_{\ell(e)}\bigr)_{e \in E}.
\end{align*}
We claim that this is a well-defined and surjective map, such that for any $g \in \Omega$, the preimage is the orbit of some $h \in G^P$ under the left $G$-action, that is $\varphi^{-1}(g) = Gh$. In fact, we have:
\begin{itemize}\itemsep=0pt
 \item We check that $\varphi$ is well-defined. Let $h \in G^P$ be a plaquette labeling. We need to show that \smash{$\prod_{e \in E(v)}\varphi(h)_e^{\varepsilon(e)} = 1_G$} for every $v \in V(\Gamma)$. Assume that $e, e' \in E(v)$ are two neighboring edges, where~$e$ follows $e'$ in the cyclic order of $E(v)$. Then, in the case that, for instance, both edges point toward~$v$, it holds that $h_{r(e')} = h_{\ell(e)}$ and if, e.g., only $e'$ points outwards at~$v$, then $h_{\ell(e)} = h_{\ell(e')}$. In the first case, the factor $\varphi(h)_{e'}^{-1}\varphi(h)_e^{-1}$ occurring in the product evaluates to $h_{\ell(e')}^{-1}h_{r(e)}$, in the second case, the factor $\varphi(h)_{e'}\varphi(h)_e^{-1}$ contracts to $h_{r(e')}^{-1}h_{r(e)}$. The remaining cases are similar. Moving once around the vertex according to the cyclic order of $E(v)$, we see that the above product is trivial, so $\varphi$ maps to $\Omega$.
 \item We prove surjectivity of the map $\varphi$.
 \begin{itemize}\itemsep=0pt
 \item Let $g \in \Omega$ be an admissible edge coloring. We will first define $h \in G^P$ with $\varphi(h) = g$ and then check that this assignment is well-defined. For this, we consider the dual graph $\Gamma^*$, where the vertices are the plaquettes, that is $V(\Gamma^*) = P$, and the edges are perpendicular to the edges $E$, but in one-to-one correspondence. We define an orientation on the dual edges so that $e$ points from $r(e)$ to $\ell(e)$. The defining conditions for $\Omega$ now translate into the requirement that the oriented product of the group elements $g_e$ (or their inverses) associated with the edges in every loop of~$\Gamma^*$ be trivial. Since $\Gamma^*$ is connected (even though $\Gamma$ does not have to be), we can choose a~spanning tree $T \subset \Gamma^*$. We select a starting vertex $p_0 \in V(T)$ and a group element~$h_{p_0}$ and color, starting from $p_0$, the remaining vertices in $T$ by the rule that $h_{p'} = h_pg_e$ for $e \in E(T)$ an edge pointing from~$p$ to~$p'$.
 \item Now, if this relation holds for all edges, then $h = (h_p)_{p \in P} \in G^P$ defines a plaquette labeling such that $\varphi(h) = g$. For this, we have to check that the relation above holds if an edge $e \in E(\Gamma^*)\setminus E(T)$ not contained in $T$ connects two dual vertices~$p$ and~$p'$. Due to the properties of a spanning tree, we can always find a path from~$p'$ to~$p$ by connecting edges $e_1, \dots, e_k \in E(T)$, which closes to a loop by adding the edge~$e$. Then, it holds that \smash{$g_{e_1}^{\varepsilon(e_1)} \cdot \dots \cdot g_{e_k}^{\varepsilon(e_k)} \cdot g_e = 1_G$}, where the signs are evaluated with respect to the vertices of $\Gamma$ on the left of the loop (with respect to the orientation on the boundary). But this directly implies that $h_{p'} = h_pg_e$.
 \end{itemize}

 \item For the second claim, assuming that $\varphi(h) = \varphi(h')$ for $h, h' \in G^P$ implies for any two neighboring plaquettes, and by connectivity of $\Gamma^*$ for all plaquettes $p$ and $p'$, that $ h_{p}(h'_{p})^{-1} = h_{p'}(h'_{p'})^{-1}$. From that, it immediately follows that $h, h' \in \mathcal{O}$ for some orbit $\calo \in G^P/G$, so $\varphi^{-1}(\varphi(h)) = \mathcal{O} = Gh$. Surjectivity of $\varphi$ implies the claim.
\end{itemize}
Hence $\varphi$ induces a bijection $\overline{\varphi}\colon G^P/G \rightarrow \Omega$ from the orbit space to the set of edge colorings. We can now conclude, using that $G$ acts freely on $G^P$, that
\begin{align*}
 \langle \Gamma \rangle = \sum_{g \in \Omega}\langle \Gamma \rangle_g = \sum_{\calo \in G^P/G} \langle \Gamma \rangle_{\overline{\varphi}(\calo)} = \frac{1}{|G|}\sum_{h \in G^P}\langle \Gamma \rangle_{\varphi(h)} = \frac{1}{|G|}\sum_{h \in G^P}\mu_h^{\mathbb{S}_{\Gamma}}.
\end{align*}
For the last equality $\langle \Gamma \rangle_{\varphi(h)} = \mu_h^{\mathbb{S}_{\Gamma}}$, we have used that the left-hand side arises by composing morphisms $\alpha_{\pi_v^E(\varphi(h))}^v = (\widetilde{\sigma}_v)_{P_v(h)}$ (cf.\ the proof of Lemma~\ref{lemEqCob}),
the right-hand side arises by composing morphisms $(\mu_v)_h = (\sigma_v)_{P_v(h)}$ and the underlying diagram representing this morphism in $\vect_{\C}$ can be transformed into $\Gamma$ by isotopy and using the symmetric braiding, mapping $\sigma_v$ to $\widetilde{\sigma}_v$. Together with \eqref{eqBallInv}, we get $Z_G(B_{\Gamma}) = \langle \Gamma \rangle$, finishing the proof.
\end{proof}

The final property we have to show is~\ref{P3}.
\begin{Proposition}\label{propDW3}
Let $M \in \Cob_3^{\vect_G}(\varnothing, \varnothing)$ be a defect cobordism to which the tube-capping move can be applied with respect to an embedding of an oriented cylinder. Denote by $V_1, \dots, V_n \in \vect_G$ the objects labeling the defect edges on the cylinder and let further $\{\alpha_i\}_{i = 1}^k$ be a basis of $\vect_G\bigl(V_1^{\varepsilon(1)} \otimes \dots \otimes V_n^{\varepsilon(n)}, \C\bigr)$. Then, $Z_G(M) = \sum_{i = 1}^{k}Z_G\bigl(\Delta_3^{\alpha_i}(M)\bigr)$ holds for Dijkgraaf--Witten theory $Z_G$.
\end{Proposition}
\begin{proof}
As in the proof of Proposition~\ref{propProp3TV}, we can assume that all signs are positive. We adopt the notation $V$ for the family of $G$-graded vector spaces $\{V_i\}$ and $\alpha$ for the basis $\{\alpha_i\}$ given in the proposition. Under the identification from \eqref{eqDual},
 we denote by $\alpha_i^* \in \vect_G (\C, V_1 \otimes \dots \otimes V_n )$ the elements of the dual basis $\alpha^*$.

 As usual, it suffices to consider the situation locally around the cylindrical region to which~$\Delta_3^{\alpha_i}$ is applied. On the one hand, as in Proposition~\ref{propProp3TV}, we have the cylinder ${\Sigma_V \times I\colon \Sigma_V \rightarrow \Sigma_V}$, where the $\vect_G$-colored surface $\Sigma_V \cong \mathbb{D}^2$ is a disk whose boundary is decorated with $n$ positively oriented points colored by the objects~$V_i$. As the cylinder $\Sigma_V \times I$ represents the identity cobordism, we have $Z_G(\Sigma_V \times I) = \id_{Z_G(\Sigma_V)}$.

 The cobordism coming from the move $\Delta_3^{\alpha_i}$ can be decomposed as \smash{$\Sigma_V \xrightarrow{E^{\alpha_i}} \varnothing \xrightarrow{\overline{E}^{\alpha_i^*}} \Sigma_V$}, where~\smash{$E^{\alpha_i}$} and \smash{$\overline{E}^{\alpha_i^*}$} had been introduced just before \eqref{eqMovTube}. We will prove the following two facts:
 \begin{enumerate}[(i)]\itemsep=0pt
 \item There is a canonical isomorphism of vector spaces $Z_G(\Sigma_V) \cong \vect_G (\C, V_1 \otimes \dots \otimes V_n )$.
 \item Under the identification above, the linear map $Z_G(E^{\alpha_i})\colon Z_G(\Sigma_V) \rightarrow \C$ is given by $f \mapsto (\alpha_i \circ f)(1)$.
 \end{enumerate}
 First, we show how the proposition follows from the above claims. Gluing \smash{$\overline{E}^{\alpha_i^*}$} and $E^{\alpha_j}$ along~$\Sigma_V$, we obtain a ball $B_{\Gamma(\alpha_i^*, \alpha_j)}$, where $\Gamma(\alpha_i^*, \alpha_j)$ is the obvious graph with two vertices that can be evaluated as $\langle \Gamma(\alpha_i^*, \alpha_j)\rangle = \alpha_j \circ \alpha_i^* = \delta_{i,j}\id_{\C}$. Thus, by Proposition~\ref{propDW4} and the claimed identities, we have
 \begin{align*}
 \alpha_i\bigl(Z_G\bigl(\overline{E}^{\alpha_i^*}\bigr)(1)\bigr) \overset{\text{(i) and (ii)}}{=} Z_G\bigl(E^{\alpha_i} \cup_{\mathbb{D}^2} \overline{E}^{\alpha_i^*}\bigr) = Z_G\bigl(B_{\Gamma(\alpha_i^*, \alpha_j)}\bigr) \overset{\ref{P4}}{=} \delta_{i,j},
 \end{align*}
 and thus $Z_G\bigl(\overline{E}^{\alpha_i^*}\bigr) = \alpha_i^*$. Hence, using \eqref{eqTV49}, we can compute{\samepage
 \begin{align*}
 \sum_{i = 1}^k Z_G\bigl(\overline{E}^{\alpha_i^*} \cup_{\varnothing} E^{\alpha_i}\bigr)(f) = \sum_{i = 1}^kZ_G\bigl(\overline{E}^{\alpha_i^*}\bigr) \circ Z_G(E^{\alpha_i})(f) = \sum_{i = 1}^k \alpha_i^* \circ \alpha_i \circ f = f,
 \end{align*}
 and as $Z_G(\Sigma_V \times I)$ is the identity, this proves the proposition.}

 Thus, what is left is to prove the above two claims. Applying $(-)_c \circ \Theta$, we obtain the cobordism $E_i = (E^{\alpha_i})_c\colon \mathbb{S}_V \rightarrow \varnothing$. Here, $E_i$ is a 3-ball, stratified as follows: All strata of codimension 1 and higher are contained in a closed disk embedded in such a way that removing it cuts the 3-ball into two disjoint pieces, each also homeomorphic to a 3-ball, one colored by $G$, the other by the trivial group. The only 0-stratum is a vertex labeled by $\alpha_i$ in the center of the disk, where the $n$ 1-strata that are radii of the disk are attached, labeled by the objects $V_i$. The 2-strata are given by the complement of the 0- and 1-strata in the disk, each homeomorphic to a disk and labeled by the regular $G$-set. The sphere $\mathbb{S}_V$ has the induced, orientation-reversed stratification.

 The above describes a fine stratification (all strata are balls of the respective dimensions) and allows us to use \cite[Corollary~5.23]{FMM} to compute the gauge groupoid as $\mathcal{A}^{E_i}\sslash \mathcal{G}^{E_i} \cong G^P \sslash G$, where $P = S_2^{E_i}$ is the set of 2-strata. Noticing that $\iota_1\colon S_1^{\partial Ei} \hookrightarrow S_2^{E_i}$ is a bijection, we obtain the same groupoid for the boundary gauge groupoid. By restriction of gauge data, it follows that the functor \smash{$P_{\partial}\colon \mathcal{A}^{E_i}\sslash \mathcal{G}^{E_i} \rightarrow \mathcal{A}^{\partial E_i}\sslash \mathcal{G}^{\partial E_i}$} is the identity. We aim to compute the limit of
 \begin{align*}
 F_{\mathbb{S}_V} = F_{\overline{\partial E}_i} = \bigotimes_{v \in S_0^{\overline{\partial E}_i}} F_v^{\overline{\partial E}_i}\colon \ \mathcal{A}^{\partial E_i}\sslash \mathcal{G}^{\partial E_i} \rightarrow \vect_{\C},
 \end{align*}
 since the TQFT assigns the space $Z_G(\Sigma_V)= \lim F_{\mathbb{S}_V}$. The components $F_v^{\overline{\partial E}_i}$ are given by
 \begin{align*}
 F_v^{\overline{\partial E}_i} \colon \ G^S \sslash G \xrightarrow{P_v} G \times G \sslash G \simeq \underline{G} \xrightarrow{\rho_e} \vect_{\C},
 \end{align*}
 where $e$ is the 1-stratum of $E_i$ inducing the boundary stratum $v$ (that is, $\iota_0(v) = e$), and $\bigoplus_{x \in G}\rho_e(x)$ with $\rho_e(x) = (V_j)_x$ is the $G$-graded vector space $e$ is labeled by. Using that the orientations on the boundary and on $\mathbb{S}_V$ are flipped, we obtain
 \begin{align}
 Z_G(\Sigma_V) &= \lim F_{\mathbb{S}_V} = \bigoplus_{Gh \in G^P/G} F_{\mathbb{S}_V}(h)^{\Aut(h)} = \bigoplus_{Gh \in G^P/G} \bigotimes_{e \in E} \rho_e\bigl(h_{\ell(e)}^{-1} h_{r(e)}\bigr) \nonumber \\ &= \bigoplus_{g \in \Omega_w} \bigotimes_{e \in E} \rho_e(g_e) \cong \vect_G(\C, V_1 \otimes \dots \otimes V_n),\label{eqSS}
 \end{align}
 where $w$ is the sole vertex labeled by $\alpha_i$, $E = E(w) = S_1^{E_i}$ and we reindexed through the set $\Omega_w = \big\{g \in G^E \mid \prod_{e \in E}g_e = 1_G \big\} \subset G^E$ and map $\varphi_w\colon G^P \rightarrow \Omega_v$ from the proof of Lemma~\ref{lemEqCob}. The inverse map $\psi$ of the isomorphism in \eqref{eqSS} simply maps a map $f$ of $G$-graded vector spaces to its evaluation at $1 \in \C$, or more precisely to $\sum_{Gh \in G^P/G}f_{\varphi_w(h)}(1)$, where we decompose into direct sum components. This shows the first statement~(i).

 With this preparation, we turn to the proof of (ii). First, note that the projection functor $P_w\colon \cala^{E_i} \sslash \calg^{E_i} \rightarrow \cald_w$ is the identity, where $w$ is the unique vertex of $E_i$. The natural transformation $\mu_{E_i} = \bigl\langle \Gamma^{E_i}\bigr\rangle\colon F_{\mathbb{S}_V} \circ P_{\partial} \Rightarrow \C$ therefore has components
 \begin{align*}
 \mu_h^{E_i} = (\sigma_wP_w)_h = (\sigma_w)_h = (\alpha_i)_{\varphi_w(h)},
 \end{align*}
 where $\sigma_w\colon \rho_v^t \Rightarrow \C$ is the natural transformation derived from $\alpha_i$ labeling the central vertex~$w$, as in the proof of Lemma~\ref{lemEqCob} (here, $\sigma_w = \widetilde{\sigma}_w^*$). With this, we want to determine the map $L(s, \mu_{E_i})\colon \lim F_{\mathbb{S}_V} \rightarrow \C$ from Proposition~\ref{propFunL}, where $s$ denotes the span $\cala^{\partial E_i} \sslash \calg^{\partial E_i} \leftarrow \cala^{E_i} \sslash \calg^{E_i} \rightarrow \bullet$. Using that all fibers $\bigl\{h \mid \mathcal{A}^{E_i}\sslash \mathcal{G}^{E_i}|\bullet\bigr\}$ of the gauge groupoid are trivial and recalling notation from Section~\ref{subSecLin}, we compute for $f \in \vect_G(\C, V_1 \otimes \dots \otimes V_n) $:
 \begin{align*}
 Z_G(E^{\alpha_i})(f) &= Z_{\rm DW}(E_i)(f) = (L(s, \mu_{E_i}) \circ \psi)(f) =\sum_{[h] \in \pi_0 (\mathcal{A}^{\partial E_i}\sslash \mathcal{G}^{\partial E_i} )}\bigl(\iota_{\bullet} \circ S_{\mu_{E_i}}^{h, \bullet} \circ \pi_h \circ \psi\bigr)(f) \\ &= \sum_{Gh \in G^P/G}\mu_h^{E_i}(f_{\varphi_w(h)}(1)) = \sum_{Gh \in G^P/G} (\alpha_i)_{\varphi_w(h)}(f_{\varphi_w(h)}(1)) \\
 &= \sum_{g \in \Omega_v} (\alpha_i)_g (f_g (1)) = (\alpha_i \circ f)(1).
 \end{align*}
 This computation shows the second claimed identity and thus finishes the proof.
\end{proof}

The main result of this section now follows from the preceding propositions.
\begin{Theorem}
 \label{thmDW}
 Let $G$ be a finite group and $Z_G\colon \Cob_3^{\vect_G} \rightarrow \vect_{\C}$ boundary-defect Dijkgraaf--Witten theory. Then, the following are true:
 \begin{enumerate}[$(i)$]\itemsep=0pt
 \item $Z_G$ is boundary local, that is, it satisfies conditions {\rm \ref{P1}--\ref{P4}} from Definition~{\rm \ref{defPropInv}}.
 \item $Z_G$ is monoidally isomorphic to the Turaev--Viro boundary-defect TQFT $|\cdot|_{\vect_G}$ for the spherical fusion category~$\vect_G$.
 \end{enumerate}
\end{Theorem}
\begin{proof}
 With the results of Proposition~\ref{propDW1}, \ref{propDW2}, \ref{propDW3} and~\ref{propDW4}, (i) follows, and Theorem~\ref{thmComb} for the spherical fusion category of $G$-graded vector spaces $\mathcal{C} = \vect_G$ yields $Z_G \cong |\cdot|_{\vect_G}$ and hence statement~(ii).
\end{proof}

We have thus found, as a corollary of the above theorem, a proof of the following well-known result (for a proof that the invariants for closed manifolds are identical, see, for example, \cite[Theorem~H.1]{TV}).
\begin{Corollary}\label{corDW}
For every finite group $G$, the restrictions of $|\cdot|_{\vect_G}$ and $Z_G$ to the subcategory $\Cob_3 \subset \Cob_3^{\vect_G}$ are isomorphic as monoidal functors, that is, Turaev--Viro theory for $\vect_G$ without defects and untwisted Dijkgraaf--Witten theory for $G$ without defects are equivalent.
\end{Corollary}
\begin{proof}
 This is a direct consequence of Theorem~\ref{thmDW} and Corollary~\ref{corThmMain2}.
\end{proof}

\begin{Remark}
Theorem~\ref{thmDW} is expected to hold in the twisted setting as well. However, this requires a construction of a defect TQFT incorporating cocycle data, which was not available at the time of writing.
\end{Remark}

We highlight the insights gained by considering boundary locality in the context of Dijkgraaf--Witten theory.
\begin{Remark}
A direct proof of Theorem~\ref{thmDW} and Corollary~\ref{corDW} should certainly be possible without making use of boundary locality. In the setting of invariants (without defects), this is done in \cite[Theorem~H.1]{TV} by choosing a triangulation and adapting the description of the Dijkgraaf--Witten invariant so that it can be computed from this additional datum, and we suspect that one can perform a similar (but more involved) proof in our setting. However, we believe that our approach offers several advantages:
\begin{itemize}\itemsep=0pt
 \item Dijkgraaf--Witten theory is most naturally formulated in a way that requires (essentially) no choice of additional data (like skeleta). The boundary locality picture respects this description and allows us to compute and prove statements directly in this framework.
 \item Thereby, the proof naturally organizes into checking the various properties of boundary locality rather than constructing a direct comparison between the two theories. No such comparison is necessary here, and the resulting proof blueprint should extend to other invariants/TQFTs in the same way, provided an extension to boundary defects is established.
 \item The proof further simplifies since boundary locality enters twice: Once to characterize the invariant uniquely, and a second time to prove non-degeneracy of the state sum to extend the results for the invariant to the (defect) TQFT (cf.~Section~\ref{subsecND}). Thus, an analysis can be performed entirely on the level of invariants, whereas leaving out defects from the picture would possibly result in additional complications at the gluing boundary.
 \item Finally, boundary locality gives rise to a conceptual explanation of theories with a state sum description via defects and exhibits these as an integral part of Turaev--Viro theories.
 Additionally, the comparison result for Turaev--Viro and Dijkgraaf--Witten theories (Theorem~\ref{thmDW}) is established automatically for TQFTs with boundary defects, not just ordinary ones.
\end{itemize}
\end{Remark}

\subsubsection*{Acknowledgments}
We thank Julian Farnsteiner, Aaron Hofer, Maximilian Ludwig, Catherine Meusburger, Paul Wedrich and Lukas Woike for
valuable discussions. We also thank the anonymous referees for their thorough reading and helpful suggestions to improve the article. The authors acknowledge support by the Deutsche Forschungsgemeinschaft (DFG, German Research Foundation) – SFB 1624 – ``Higher structures, moduli spaces and integrability'' - 506632645. C.S.\ is partially funded under Germany's Excellence
Strategy – EXC 2121 ``Quantum Universe'' – 390833306.

\pdfbookmark[1]{References}{ref}
\LastPageEnding


\begin{thebibliography}{99}
\footnotesize\itemsep=0pt

\bibitem{BK}
Balsam B., Kirillov Jr. A., Turaev--{V}iro theory as an extended {TQFT},
 \href{http://arxiv.org/abs/1004.1533}{arXiv:1004.1533}.

\bibitem{Ba}
Bartlett B., Three-dimensional {TQFT}s via string-nets and two-dimensional
 surgery, \href{https://doi.org/10.4171/qt/260}{\textit{Quantum Topol.}}
 \textbf{17} (2026), 509--533,
 \href{http://arxiv.org/abs/2206.13262}{arXiv:2206.13262}.

\bibitem{CGPT}
Costantino F., Geer N., Patureau-Mirand B., Turaev V., Kuperberg and
 {T}uraev--{V}iro invariants in unimodular categories,
 \href{https://doi.org/10.2140/pjm.2020.306.421}{\textit{Pacific~J. Math.}}
 \textbf{306} (2020), 421--450,
 \href{http://arxiv.org/abs/1809.07991}{arXiv:1809.07991}.

\bibitem{CGPV}
Costantino F., Geer N., Patureau-Mirand B., Virelizier A., Non compact
 (2+1)-{TQFT}s from non-semisimple spherical categories,
 \href{http://arxiv.org/abs/2302.04509}{arXiv:2302.04509}.

\bibitem{EGNO}
Etingof P., Gelaki S., Nikshych D., Ostrik V., Tensor categories, \textit{Math.
 Surveys Monogr.}, Vol.~205, \href{https://doi.org/10.1090/surv/205}{American
 Mathematical Society}, Providence, RI, 2015.

\bibitem{ENO}
Etingof P., Nikshych D., Ostrik V., On fusion categories,
 \href{https://doi.org/10.4007/annals.2005.162.581}{\textit{Ann. of Math.}}
 \textbf{162} (2005), 581--642,
 \href{http://arxiv.org/abs/math.QA/0203060}{arXiv:math.QA/0203060}.

\bibitem{FMM}
Faria~Martins J., Meusburger C., A geometrical description of untwisted 3d
 {D}ijkgraaf-{W}itten {TQFT} with defects,
 \href{https://doi.org/10.1016/j.aim.2026.110923}{\textit{Adv. Math.}}
 \textbf{494} (2026), 110923, 102~pages,
 \href{http://arxiv.org/abs/2410.18049}{arXiv:2410.18049}.

\bibitem{F}
Farnsteiner J., Generalized {F}robenius--{S}chur indicators as {T}uraev--{V}iro
 invariants, {M}aster's {T}hesis, University of Hamburg, 2020.

\bibitem{FS}
Farnsteiner J., Schweigert C., Frobenius--{S}chur indicators and the mapping
 class group of the torus,
 \href{https://doi.org/10.1007/s11005-022-01513-6}{\textit{Lett. Math. Phys.}}
 \textbf{112} (2022), 39, 30~pages,
 \href{http://arxiv.org/abs/2104.12742}{arXiv:2104.12742}.

\bibitem{FSW}
Fuchs J., Schweigert C., Woike L., A first course in topological field theory,
 \textit{Univ. Lecture Ser.}, Vol.~80,
 \href{https://doi.org/10.1090/ulect/080}{American Mathematical Society},
 Providence, RI, 2026.

\bibitem{LMWW}
Liu Z., Ming S., Wang Y., Wu J., 3-alterfolds and quantum invariants,
 \href{http://arxiv.org/abs/2307.12284}{arXiv:2307.12284}.

\bibitem{LMWW2}
Liu Z., Ming S., Wang Y., Wu J., Alterfold topological quantum field theory,
 \textit{Sci. China Math.}, {t}o appear,
 \href{http://arxiv.org/abs/2312.06477}{arXiv:2312.06477}.

\bibitem{M}
Meusburger C., State sum models with defects based on spherical fusion
 categories, \href{https://doi.org/10.1016/j.aim.2023.109177}{\textit{Adv.
 Math.}} \textbf{429} (2023), 109177, 97~pages,
 \href{http://arxiv.org/abs/2205.06874}{arXiv:2205.06874}.

\bibitem{MW}
M\"uller L., Woike L., Dimensional reduction, extended topological field
 theories and orbifoldization,
 \href{https://doi.org/10.1112/blms.12427}{\textit{Bull. Lond. Math. Soc.}}
 \textbf{53} (2021), 392--403,
 \href{http://arxiv.org/abs/2004.04689}{arXiv:2004.04689}.

\bibitem{NS}
Ng S.H., Schauenburg P., Congruence subgroups and generalized
 {F}robenius--{S}chur indicators,
 \href{https://doi.org/10.1007/s00220-010-1096-6}{\textit{Comm. Math. Phys.}}
 \textbf{300} (2010), 1--46,
 \href{http://arxiv.org/abs/0806.2493}{arXiv:0806.2493}.

\bibitem{TV}
Turaev V., Virelizier A., Monoidal categories and topological field theory,
 \textit{Progr. Math.}, Vol.~322,
 \href{https://doi.org/10.1007/978-3-319-49834-8}{Birkh\"auser}, Cham, 2017.

\bibitem{W}
Walker K., A universal state sum,
 \href{http://arxiv.org/abs/2104.02101}{arXiv:2104.02101}.

\end{thebibliography}
\end{document}